\documentclass[oneside,12pt]{article}

\usepackage{amsmath,amsfonts,amssymb,amsthm,color,url}

\usepackage[hypertexnames=false,hidelinks]{hyperref}

\usepackage[left=1in,right=1in]{geometry}

\usepackage[affil-it,auth-lg]{authblk}

\usepackage{mathptmx}
\usepackage[T1]{fontenc}

\SetMathAlphabet{\mathcal}{normal}{OMS}{lmsy}{m}{n}
\SetMathAlphabet{\mathcal}{bold}{OMS}{lmsy}{m}{n}
\SetSymbolFont{largesymbols}{normal}{OMX}{lmex}{m}{n}
\SetSymbolFont{largesymbols}{bold}{OMX}{lmex}{m}{n}

\usepackage[bf,small,raggedright]{titlesec}

\usepackage[title,toc]{appendix}

\newcommand{\BCol}{\mathcal{B}}
\newcommand{\SDis}{\mathcal{D}}

\DeclareMathOperator{\strat}{\circ}

\newcommand{\Ft}{\mathcal{F}}
\newcommand{\ft}[1]{\widehat{#1}}
\newcommand{\Meas}{\mathcal{M}}
\newcommand{\LStrat}{\mathcal{L}}
\newcommand{\Renorm}{\Gamma}
\newcommand{\dblbrak}[2]{\left\langle\mkern-4mu\left\langle #1,#2 \right\rangle\mkern-4mu\right\rangle}
\newcommand{\R}{\mathbb{R}}

\newcommand{\Rd}{\mathbb{R}^d}

\DeclareMathOperator{\Div}{\mathrm{div}}

\usepackage{bbm}
\newcommand{\1}{\mathbbm{1}}
\newcommand{\F}{\mathcal{F}}

\newcommand{\N}{\mathbb{N}}
\newcommand{\ep}{\epsilon}

\newcommand{\tensor}{\otimes}
\renewcommand{\S}{\mathbb{S}\,}
\newcommand{\E}{\mathbb{E}}
\renewcommand{\P}{\mathbb{P}}
\newcommand{\p}{\mathbb{P}}

\newcommand{\leqc}{\lesssim}
\newcommand{\leqs}{\lesssim}

\newtheorem{thm}{Theorem}[section]
\newtheorem{prop}[thm]{Proposition}
\newtheorem{cor}[thm]{Corollary}
\newtheorem{lem}[thm]{Lemma}

\newtheorem{hyp}{Hypothesis}[section]

\theoremstyle{definition}
\newtheorem{definition}{Definition}[section]

\theoremstyle{remark}
\newtheorem{remark}{Remark}[section]

\newcommand{\loc}{\mathrm{loc}}
\newcommand{\w}{\mathrm{w}}

\newcommand{\LLL}[2]{L_{t}^{#1}(L_{x,v}^{#2})}
\newcommand{\LLs}[1]{L_{x,v}^{#1}}
\newcommand{\LLt}[1]{L_{t,x}^{#1}}
\newcommand{\LLLs}[1]{L_{t,x,v}^{#1}}

\newcommand{\CL}[1]{C_{t}(L_x^{#1})}
\newcommand{\CLw}[1]{C_{t}([L_x^{#1}]_{w})}
\newcommand{\CLLw}[1]{C_{t}([L_{x,v}^{#1}]_{w})}
\newcommand{\CLL}[1]{C_{t}(L_{x,v}^{#1})}

\newcommand{\LLsM}[1]{L_{t,x}^{#1}(\Meas_{v})}
\newcommand{\LLsMw}[1]{L_{t,x}^{#1}(\Meas_{v}^{*})}

\newcommand{\MMMsw}[1]{ \Meas_{t,x,v}^{*} }
\newcommand{\MMMs}[1]{ \Meas_{t,x,v} }

\newcommand{\dee}{\mathrm{d}}
\newcommand{\ds}{\dee s}
\newcommand{\dt}{\dee t}
\newcommand{\dv}{\dee v}
\newcommand{\dw}{\dee w}
\newcommand{\dx}{\dee x}

\newcommand{\dxi}{\dee \xi}
\newcommand{\dz}{\dee z}
\newcommand{\dr}{\dee r}

\usepackage{autonum}

\begin{document}

\title{On the Boltzmann Equation with\\
  Stochastic Kinetic Transport: \\
  Global Existence of\\
  Renormalized Martingale Solutions}
\author{
  Samuel Punshon-Smith\thanks{University of Maryland, College Park - \url{punshs@math.umd.edu}}\\
  \and
  Scott Smith\thanks{Max Planck Institute for Mathematics in the Sciences, Leipzig - \url{ssmith@mis.mpg.de}}}

\date{\today}

\maketitle

\begin{abstract}
This article studies the Cauchy problem for the Boltzmann equation with stochastic kinetic transport. Under a cut-off assumption on the collision kernel and a coloring hypothesis for the noise coefficients, we prove the global existence of renormalized (in the sense of DiPerna/Lions \cite{diperna1989cauchy}) martingale solutions to the Boltzmann equation for large initial data with finite mass, energy, and entropy.  Our analysis includes a detailed study of weak martingale solutions to a class of linear stochastic kinetic equations.  This study includes a criterion for renormalization, the weak closedness of the solution set, and tightness of velocity averages in $L^1$.
\end{abstract}
\newpage
\tableofcontents
\newpage

\section{Introduction}\label{sec:Intro}
The Boltzmann equation
\begin{equation}\label{eq:Boltzmann-Deterministic}
  \begin{aligned}
    &\partial_t f + v\cdot\nabla_x f + \Div_v(\mathcal{X} f) = \BCol(f,f),\\
    &f|_{t=0} = f_0,
  \end{aligned}
\end{equation}
on $[0,T]\times \R^{2d}$ is a nonlinear integro-differential equation describing the evolution of a rarefied gas, dominated by binary collisions, and in the presence of a external force field $\mathcal{X}$. The function $f(t,x,v)\in \R$ describes the density of particles at time $t\in[0,T]$, position $x\in\Rd$, with velocity $v\in \Rd$, starting at $t=0$ from an initial density $f_0(x,v)$. The nonlinear functional $f \mapsto \BCol(f,f)$, known as the collision operator, acts on the velocity variable only, and accounts for the effect of collisions between pairs of particles; it will be described in more detail below.
 
Several studies have been conducted regarding the well-posedness of the Cauchy problem for the Boltzmann equation (\ref{eq:Boltzmann-Deterministic}) with a fixed (deterministic) external force, for instance \cite{Asano1987-dr,Bellomo1989-wg,Duan2006-ie,ukai2005global}. In general, the external force field $\mathcal{X}$ may depend on $(t,x,v)\in \R\times\Rd\times\Rd$. Such external forces may arise when considering the influence of gravity such as in the treatment of the Rayleigh-Benard problem in the kinetic regime \cite{esposito1998solutions,arkeryd2010stability}. In fact, many external forces are not fixed, and are instead coupled with the density $f$ in a self consistent way. This is the case, for example, with the Vlasov-Poisson-Boltzmann and Vlasov-Maxwell-Boltzmann equations (see \cite{Cercignani1988-aw,Lions1994-jf} and references therein for more details on these systems).

This article focuses instead on the Cauchy problem for the Boltzmann equation with {\it random} external forcing.  In particular, we are interested in the following SPDE
\begin{equation}\label{eq:general-stochastic-Boltz}\tag{SB}
\begin{aligned}
 &\partial_t f + v\cdot \nabla_x f + \Div_v (f\sigma_{k}\strat \dot{\beta}_{k}) = \BCol(f,f),\\
 &f|_{t=0} = f_0,
\end{aligned}
\end{equation}
where $\{\beta_k\}_{k\in\N}$ are one-dimensional Brownian motions and $\{\sigma_k\}_{k\in\N}$ are a family of vector fields $\sigma_{k}:\R^{2d} \to \R^{d}$ with $\Div_{v}\sigma_{k}=0$.  An implicit summation is taken over $k \in \N$, and the expression $\Div_v(f\sigma_k\strat \dot{\beta}_{k})$ denotes a transport type multiplicative noise, white in time and colored in $(x,v)$, where the product $\circ$ is interpreted in the Stratonovich sense. 
 
Physically, we view the quantity
\begin{equation}
(t,x,v) \mapsto \sum_{k \in \N}\sigma_{k}(x,v)\dot{\beta}_{k}(t)
\end{equation}
as an environmental noise acting on the gas. In the absence of collisions, all particles evolve according to the stochastic differential equation
\begin{equation}\label{eq:SDE_Characteristics}
  \dee X_t = V_t\dt,\quad \dee V_t = \sum_{k\in\N}\sigma_k(X_t,V_t)\circ \dee\beta_k(t)
\end{equation} 
and are only distinguished from one another according to their initial location in the phase space. Let $\Phi_{s,t}(x,v)$ be the stochastic flow associated with the SDE (\ref{eq:SDE_Characteristics}), that is, $t\mapsto \Phi_{s,t}(x,v) = (X_t,V_t)$ solves (\ref{eq:SDE_Characteristics}) and satisfies $\Phi_{s,s}(x,v) = (x,v)$. The Stratonovich form of the noise and the fact that $\Div_v \sigma_k =0$ ensures that the flow $\Phi_{s,t}$ is volume preserving (with probability one). The density of the collision-less gas is then given by $f_t(x,v) = f_0(\Phi_{0,t}^{-1}(x,v))$ and evolves according to the free stochastic kinetic transport equation
\begin{equation}\label{eq:free-stochastic-transport}
\begin{aligned}
  &\partial_t f + v\cdot\nabla_x f + \Div_v(f \sigma_k\circ \dot{\beta}_k) = 0,\\
  &f|_{t=0} = f_0.
\end{aligned}
\end{equation}

The presence of collisions interrupts the stochastic transport process. In the low volume density regime, binary collisions are dominant and can be described by the Boltzmann collision operator $\mathcal{B}(f,f)$. The stochastic Boltzmann equation (\ref{eq:general-stochastic-Boltz}) accounts for both stochastic transport and binary collisions. In fact, formally (\ref{eq:general-stochastic-Boltz}) can be written in mild form,
\begin{equation}
  f_t = f_0\circ \Phi^{-1}_{0,t} + \int_0^t \mathcal{B}(f_s,f_s)\circ \Phi^{-1}_{s,t}\ds.
\end{equation}

The stochastic Boltzmann equation (\ref{eq:general-stochastic-Boltz}) can be interpreted as the so-called Boltzmann-Grad limiting description of interacting particles subject to the {\it same} environmental noise. In the deterministic setting, the Boltzmann-Grad problem has been studied extensively in the literature (see \cite{gallagher2013newton} for a recent review). In the stochastic setting, the Boltzmann-Grad problem has (to our knowledge) not yet been studied. However, a mean field limit to the Vlasov equation with stochastic kinetic transport has been shown recently by Coghi and Flandoli \cite{coghi2014propagation}.

To our knowledge, this is the first study to obtain mathematically rigorous results on the Boltzmann equation with a random external force. However, a number of results on the {\it fluctuating Boltzmann equation} are available in the Math and Physics literature \cite{bixon1969boltzmann,fox1970contributions,montroll2012fluctuation,spohn2012large,spohn1981fluctuations,spohn1983fluctuation,ueyama1980stochastic}. In particular, the articles of Bixon/Zwanzig \cite{bixon1969boltzmann} and Fox/Uhlenbeck \cite{fox1970contributions} outline a formal derivation of Landau and Lifshitz's equations of fluctuating hydrodynamics \cite{landau1959fluid}, from the fluctuating linear Boltzmann equation. The connection with macroscopic fluid equations arises from studying the correlation structure of the fluctuations at the level of the kinetic description. A more rigorous treatment of the fluctuation theory for the Boltzmann equation and its connection to the Boltzmann-Grad limit is given by Spohn \cite{spohn2012large,spohn1981fluctuations,spohn1983fluctuation}. 

Although our perspective differs from that of \cite{fox1970contributions} and \cite{bixon1969boltzmann}, we do expect to obtain various stochastic hydrodynamic equations (with colored noise) in different asymptotic regimes, using a Chapman-Enskog expansion and the moments method of Bardos/Golse/Levermore \cite{bardos1991fluid}.  In fact, one of the original motivations for this article was to understand which of the common forms of noise in the stochastic fluids literature can be obtained by considering fluctuations of the stochastic kinetic description relative to an equilibrium state. This will be addressed in detail in future works. 

The goal of this article is to investigate global solutions to \eqref{eq:general-stochastic-Boltz} starting from general `large' initial data $f_0 \in L^1(\R^{2d})$. If the noise coefficients $\sigma$ are identically zero, then this problem has already been addressed in the seminal work of DiPerna/Lions \cite{diperna1989cauchy}, where existence of renormalized solutions is proved.  Our work is heavily inspired by \cite{diperna1989cauchy}, relying on a number of their insights together with various classical properties of the Boltzmann equation.  Rather than give a detailed review, in the next subsection we will explain how these observations from the deterministic theory lead to the notion of renormalized martingale solution to \eqref{eq:general-stochastic-Boltz} in the present context.  Finally, we should mention that our initial motivation for the choice of noise was heavily inspired by a number of interesting works on stochastic transport equations (see for instance  \cite{delarue2014noise,flandoli2011interaction,flandoli2010well,fedrizzi2016regularity}). Finally, we should mention the work \cite{brzezniak2016existence} on the 2-d stochastic Euler equations with a very similar noise to the one in this paper. 

\subsection{Notation and statement of the main result}

We begin by giving a summary of the main notation used in the paper. To simplify the appearance of the function spaces, we will use a number of abbreviations. Since the functions $f(t,x,v)$ we are dealing with are typically kinetic densities the variables $x$ and $v$ will be reserved for positions and velocity variables, while $t$ will be reserved for time. We will typically fix an arbitrary but finite time $T>0$ and consider evolution on the time interval $[0,T]$ and denote by $\LLL{q}{p}$ the space $L^q([0,T]; L^p(\R^{2d}))$ of real valued functions $f(t,x,v)$ on $[0,T]\times\R^{2d}$ with norm
\begin{equation}
  \|f\|_{\LLL{q}{p}} := \left(\int_0^T\left(\iint_{\R^{2d}}|f(t,x,v)|^p\dv\dv\right)^{q/p}\dt\right)^{1/q},
\end{equation}
with the usual interpretation in terms of the essential supremum when $p$ or $q$ are $\infty$. Naturally, $\LLLs{p}$ is also short for the Lebesgue space $L^p([0,T]\times\R^{2d})$. A  similar subscript notation will hold for Sobolev spaces, namely $W^{s,p}_{x,v}$ will denote the space of functions $f\in L^p_{x,v}$ whose $s$-th order derivatives also belong to $L^p_{x,v}$. Given a test function $\varphi \in C^{\infty}_c(\R^{2d})$ and a function $f$ in $L^p_{x,v}$ we will often denote the pairing by
\begin{equation}
  \langle f, \varphi\rangle := \iint_{\R^{2d}} f\, \varphi \,\dv\dx.
\end{equation}
This should, however, should not be confused with the velocity `average'
\begin{equation}
  \langle f\rangle := \int_{\R^d} f\dv,
\end{equation}
which integrates in $v$ only.

In general, given a topological space $U$, with a particular topology, we denote by $C_t(U)$ the space of continuous functions $f:[0,T]\to U$. Furthermore, given a Banach space $B$, we denote the space endowed with its weak topology by $[B]_{\w}$, and therefore space of weakly continuous functions from $[0,T]$ to $B$ is $C_t([B]_w)$. Also, given a Lebesgue space $L^p_{x,v}$ (or Sobolev Space), we denote $[L^p_{x,v}]_{\loc}$ the usual Frechet space of locally integrable functions endowed with it's natural local topology.

We should also introduce some probablisitic notation. For a given probability space $(\Omega, \mathcal{F}, \P)$, we let $\E$ be the expectation associated to the probability measure $\P$. Also, given a Banach space $B$ (not necessarily separable) with norm $\|\,\cdot\,\|_B$, we will denote by $L^p(\Omega; B)$, $p\in [1,\infty]$ the Lebesgue-Bochner space of equivalence classes of strongly measurable maps $f:\Omega \to B$ (random variables) such that $\E\|f\|_B^p$ is finite. We will often use the space
\begin{equation}
L^{\infty-}(\Omega;B) := \bigcap_{p \geq 1} L^p(\Omega;B).
\end{equation}
cosisting of random variables all moments finite (but not $L^\infty$). We will often suppress the dependence of a random variable on the probablistic variable $\omega \in \Omega$, except when explicitly needed.

Let us now discuss the basics of the Boltzmann equation and introduce the analytical framework for the problem. We refer the reader to the books \cite{Cercignani1988-aw,Cercignani2013-vz} and the excellent set of notes \cite{Golse2005-zh} for a comprehensive introduction to the Boltzmann equation, as well as the review \cite{villani2002review}. The collision operator $\mathcal{B}(f,f)$ describes the rate of change in particle density due to collisions. It contains all the information about collision rates between particles with different velocities. More precisely, it is defined through its action in $v$ as
\begin{equation}\label{eq:Boltzmann-Coll-Def}
  \BCol(f,f)(v) = \iint_{\R^d\times\S^{d-1}} (f^\prime f^\prime_* - ff_*)(v,v_*,\theta)b(v-v_*,\theta)\dee\theta\dee v_*,
\end{equation}
where $f_*$, $f^\prime$, and $f^\prime_*$ are shorthand for $f(v_*) , f(v^\prime)$, and $f(v^\prime_*)$, while $(v^\prime, v^\prime_*)$ denote pre-collisional velocities
\begin{equation}
\begin{cases}
v^\prime = v - (v-v_*)\cdot\theta\,\theta\\
v_*^\prime = v_* +(v-v_*)\cdot\theta\,\theta.
\end{cases}
\end{equation}
Note that $(v^\prime, v^\prime_*)$, parametrized by $\theta\in \S^{d-1}$, are solutions to the equations describing pairwise conservation of momentum and energy,
\begin{equation} \label{eq:ConservationLaws}
  \begin{aligned}
    v^\prime + v^\prime_* &= v + v_*\\
  |v^\prime|^2 + |v^\prime_*|^2 &= |v|^2 + |v_*|^2.
  \end{aligned}
\end{equation}
The collision kernel $b(v-v_*,\theta) \geq 0$ is determined by details of the inter-molecular forces between particles and describes the rate at which particles with relative velocity $v-v_*$ collide with deflection angle $\theta\cdot (v-v_*)/|v - v_*|$. In this article, for technical reasons and simplicity of exposition, we restrict our attention to bounded, integrable kernels, though we intend to investigate (in a future work) the possibility of treating more singular kernels as in Alexandre/Villani \cite{alexandre2002boltzmann} and other works. Our assumption on the collision kernel is the following:
\begin{hyp}\label{hyp:Collision-Kernel}
The collision kernel $b(z,\theta)$ depends solely on $|z|$ and $|z\cdot\theta|$ only, and satisfies,
  \begin{equation}
    b \in L^1(\R^{2d}\times \S^{d-1})\cap L^\infty(\R^{2d}\times \S^{d-1}). 
  \end{equation}
\end{hyp}
Since the nonlinear term $\BCol(f,f)$ is quadratic in $f$, further properties of the operator must be exploited in order to obtain a priori bounds. A classical observation is that the symmetry assumptions on the collision kernel $b$ imposed in Hypothesis \ref{hyp:Collision-Kernel} and the definition of $(v^\prime, v^\prime_*)$ imply that for each smooth $\xi: \Rd \to \R$,
\begin{equation}\label{eq:symmetry_identity}
\begin{aligned}
&\int_{\R^{d}}\xi(v)\mathcal{B}(f,f)\dv\\ 
&\hspace{.5in}= \frac{1}{4}\iiint_{\R^{2d}\times\S^{d-1}}(f^\prime f^\prime_* - ff_*)(\xi_{*}+\xi-\xi_{*}'-\xi')\,b(v-v_*,\theta)\,\dee\theta\dv_*\dv.
\end{aligned}
\end{equation} 
Any quantity $\xi(v)$ such that $\xi_* + \xi  = \xi^\prime_* + \xi^\prime$, is called a collision invariant. For any collision invariant $\xi(v)$, \eqref{eq:symmetry_identity} implies that
\begin{equation}
  \int_{\Rd} \xi(v) \BCol(f,f)(v)\dv = 0.
\end{equation} 
As a result of the definition of $(v^\prime,v^\prime_*)$, the quantities $\{1, \{v_{i}\}_{i=1}^d, |v|^{2}\}$ are collision invariants. Therefore, multiplying both sides of (\ref{eq:general-stochastic-Boltz}) by a collision invariant and integrating in $v$, the collision operator vanishes
\begin{equation}\label{eq:Collision_Invariant_Evol}
  \partial_t\Big(\int_{\Rd} \xi(v) f\,\dv\Big) + \Div_x\Big(\int_{\Rd} v\xi(v) f\,\dv\Big) = \Big(\int_{\Rd} \nabla \xi(v)\cdot\sigma_k\,f\,\dv\Big) \strat \dot{\beta_k}.
\end{equation}
In the case that $\xi(v) = 1 + |v|^2$ in \eqref{eq:Collision_Invariant_Evol}, one can close on estimate on $\xi(v)f$, provided we have the following coloring hypothesis on $\sigma$:
\begin{hyp}\label{hyp:Noise-Coefficients}
For each $k \in \N$, the noise coefficient $\sigma_{k}: \R^{2d} \to \R^{d}$ satisfies $\Div_{v}\sigma_{k}=0$. In addition, the sequence $\sigma = \{\sigma_{k}\}_{k \in \N}$ obeys:
\begin{equation}
  \tag{H1}\label{eq:Noise-Assumption-1} \|\sigma\|^2_{\ell^2(\N; L^\infty_x(W^{1,\infty}_v))}:= \sum_{k\in\N} \|\sigma_{k}\|_{L^{\infty}_{x,v}}^2 + \|\nabla_v\sigma_k\|_{L^{\infty}_{x,v}}^2 < \infty
  \end{equation}
\end{hyp}
\begin{remark}
  Note that Hypothesis \ref{hyp:Noise-Coefficients} states that the coefficients are bounded in both $x$ and $v$, but Lipschitz in $v$ only.
\end{remark}
More generally, in Section \ref{sec:FormalBounds} we show that Hypothesis \ref{hyp:Noise-Coefficients} implies that a solution $f$ to \eqref{eq:general-stochastic-Boltz} satisfies the following formal a priori bound
\begin{equation} \label{eq:aPriori_Bound}
\E\|(1+|x|^{2}+|v|^{2})f\|_{\LLL{\infty}{1}}^{p} \leq C_{p},
\end{equation}
for all $p \in [1,\infty)$ and some positive constant $C_{p}$ (depending on $p$).  In addition, a further $L\log{L}$ estimate on $f$ is available due to the entropy structure of \eqref{eq:general-stochastic-Boltz}. To obtain this, let $\Renorm : \R \to \R$ be a sufficiently smooth function, which we will refer to as a renormalization. Since we use Stratonovich noise and $\Div_v \sigma_k = 0$, if $f$ is a solution of \eqref{eq:general-stochastic-Boltz}, then formally $\Renorm(f)$ should satisfy:
\begin{equation}\label{eq:general-renorm-stochastic-Boltz}\tag{RSB}
\begin{aligned}
 &\partial_t \Renorm(f) + v\cdot \nabla_x \Renorm(f) + \Div_v (\Renorm(f)\sigma_{k}\strat \dot{\beta}_{k}) = \Renorm'(f)\BCol(f,f),\\
 &\Renorm(f)|_{t=0} = \Renorm(f_0).
\end{aligned}
\end{equation}   
In particular, taking $\Renorm(f)=f \log f$ in \eqref{eq:general-renorm-stochastic-Boltz} and integrating in $v$ yields
\begin{equation}\label{eq:entropy-dis}
  \partial_t\Big(\int_{\Rd} f\log{f}\,\dv \Big) + \Div_x\Big(\int_{\Rd} v\,f\log{f}\,\dv\Big)= - \SDis(f),
\end{equation}
where 
\begin{equation}\label{eq:Entopy-Dissipation-Def}
\begin{aligned}
\SDis(f) &:= \frac{1}{4}\iiint_{\R^{2d}\times\S^{d-1}}d(f)(t,x,v,v_*,\theta)\,b(v-v_*,\theta)\,\dee\theta\dv_*\dv,\\
d(f) &:=  (f^\prime f^\prime_* - ff_*)\log\left(\frac{f^\prime f^\prime_*}{f f_*}\right) \geq 0.
\end{aligned}
\end{equation}
Equation (\ref{eq:entropy-dis}) describes the local dissipation of the entropy density $\int_{\Rd} f\log{f}\dv$. The quantity $\SDis(f)$ is referred to as the entropy dissipation, and inherits non-negativity from $d(f)$. Since $f\log f$ is unsigned, we cannot immediately use \eqref{eq:entropy-dis} to obtain an $L\log{L}$ bound. However, combining this with \eqref{eq:aPriori_Bound}, in Section \ref{sec:FormalBounds} we show that for all $p \in [1,\infty)$ 
\begin{equation} \label{eq:aPriori_Bound-Final}
\E\|f \log f\|_{\LLL{\infty}{1}}^{p}, \quad\E \|\SDis(f)\|_{\LLt{1}}^{p} \leq C_{p}. 
\end{equation}
Although the a priori bounds \eqref{eq:aPriori_Bound} and \eqref{eq:aPriori_Bound-Final} provide a useful starting point, they are unfortunately insufficient to give a meaning to $\BCol(f,f)$ in the sense of distributions. For bounded kernels, one can obtain an $L^1_v$ estimate on $\mathcal{B}(f,f)$,
\begin{equation}
\|\mathcal{B}(f,f)\|_{L^{1}_{v}} \leq C \|f\|^{2}_{L^{1}_{v}}.
\end{equation} 
However, since $\mathcal{B}(f,f)$ acts pointwise in $x$, the operator $f \mapsto \mathcal{B}(f,f)$ sends $\LLs{1}$ to $L^{0}_{x}(L^{1}_{v})$ (a measurable function in $x$). A key observation of DiPerna and Lions \cite{diperna1989cauchy} is that the renormalized collision operator $f \to (1+f)^{-1}\mathcal{B}(f,f)$ is better behaved.  More precisely, the following inequality holds:
\begin{equation}\label{eq:renorm-coll-bound}
\|(1+f)^{-1}\mathcal{B}(f,f)\|_{\LLLs{1}} \leqs \|\mathcal{D}(f)\|_{\LLt{1}}+\|f\|_{\LLLs{1}}.
\end{equation}
Thus, if $f$ satisfies the a priori bounds \eqref{eq:aPriori_Bound} and \eqref{eq:aPriori_Bound-Final}, the quantity $(1+f)^{-1}\mathcal{B}(f,f)$ is well defined in $L^{\infty-}(\Omega;\LLLs{1})$. Hence, it becomes feasible to search for solutions satisfying \eqref{eq:general-renorm-stochastic-Boltz} in the sense of distributions for a suitable class of renormalizations.  Towards this end, we make the following definition:
\begin{definition}
 Define the set of renormalizations $\mathcal{R}$ to consist of $C^{1}(\R_{+})$ functions $\Renorm: \R_{+} \to \R$ such that the mapping $x \mapsto (1+x)\,|\Renorm^\prime(x)|$ belongs to $L^{\infty}(\R_{+})$.
\end{definition}
It is important to keep in mind that this class of renormalizations excludes the possibility of choosing $\Gamma(f) = f$ or $\Gamma(f) = f\log{f}$ and therefore extra care must be taken to obtain the a priori estimates (\ref{eq:aPriori_Bound}) and (\ref{eq:aPriori_Bound-Final}) above.  

We note that for analytical purposes, relating to martingale techniques, it is often more convenient to work with \eqref{eq:general-renorm-stochastic-Boltz} in It\^{o} form. Thus, we introduce the matrix 
\begin{equation}
a(x,v)=\frac{1}{2}\sum_{k \in \N} \sigma_{k}(x,v)\tensor \sigma_{k}(x,v),
\end{equation}
and define the operator 
\begin{equation}
\mathcal{L}_{\sigma}\varphi := \Div_{v}(a\nabla_{v} \varphi).
\end{equation} 
Using the divergence free assumption for each $\sigma_{k}$, the random transport term in \eqref{eq:general-renorm-stochastic-Boltz} can be converted to It\^{o} form via the relation
\begin{equation}
\Div_v (\Renorm(f)\sigma_{k}\strat \dot{\beta}_{k})=-\mathcal{L}_{\sigma}\Renorm(f) + \Div_v (\Renorm(f)\sigma_{k}\dot{\beta}_{k}).
\end{equation}

We are now ready to define our notion of solution for (\ref{eq:general-renorm-stochastic-Boltz}). \begin{definition} \label{def:Renorm-Martingale-BM}
A density $f$ is defined to be a renormalized martingale solution to \eqref{eq:general-stochastic-Boltz} provided there exists a stochastic basis $(\Omega, \mathcal{F}, \p, (\mathcal{F}_{t})_{t=0}^{T}, \{ \beta_{k}\}_{k \in \N})$ such that the following hold:
\begin{enumerate}
\item For all $(t,\omega) \in [0,T] \times \Omega$, the quantity $f(t,\omega)$ is a non-negative element of $L^{1}_{x,v}$.  
\item The mapping $f: [0,T] \times \Omega \to L^{1}_{x,v}$ defines an $(\mathcal{F}_{t})_{t=0}^{T}$ adapted process with continuous sample paths.
\item For all renormalizations $\Renorm \in \mathcal{R}$, test functions $\varphi \in C\,^{\infty}_{c}(\R^{2d})$, and times $t \in [0,T]$; the following equality holds $\p$ almost surely:
\begin{equation}\label{eq:Renorm-Weak-Form}
\begin{split}
&\iint_{\R^{2d}}\Renorm(f)(t)\varphi \dx \dv = \iint_{\R^{2d}}\Renorm(f_{0})\varphi \dx \dv\\ 
&\hspace{.2in}+ \int_{0}^{t}\iint_{\R^{2d}}[\Renorm(f)v \cdot \nabla_{x}\varphi+\Renorm'(f)\mathcal{B}(f,f)\varphi]\dx\dv\ds\\
&\hspace{.2in} + \frac{1}{2}\int_{0}^{t}\iint_{\R^{2d}}\Renorm(f)\LStrat_\sigma\varphi\,\dx\dv\ds\ +\sum_{k\in\N}\int_{0}^{t}\iint_{\R^{2d}}\Renorm(f)\sigma_{k}\cdot \nabla_{v}\varphi\, \dx\dv \dee\beta_{k}(s).
\end{split}
\end{equation}
\item For all $p \in [1,\infty)$ there exists a positive constant $C_{p}$ such that:
\begin{equation}\label{eq:Moment-ent-dis-estimates-thm}
\E\|(1+|x|^{2}+|v|^{2}+|\log f|)f\|_{\LLL{\infty}{1}}^{p}, \quad \E \|\mathcal{D}(f)\|^{p}_{\LLt{1}} \leq C_{p}.
\end{equation}
\end{enumerate}
\end{definition}

\begin{remark}
  In light of the estimate (\ref{eq:renorm-coll-bound}), the estimates in condition 4 of Definition \ref{def:Renorm-Martingale-BM} ensure that the weak form (\ref{eq:Renorm-Weak-Form}) is well defined and the stochastic integral is a continuous-time martingale.
\end{remark}

At present, we require a further technical hypothesis on  $\sigma$ and $\sigma \cdot \nabla_{v}\sigma$. This is related to the regularity needed on $\sigma$ to renormalize a linear, stochastic kinetic transport equation, a crucial procedure in our analysis. This is discussed in more detail in Section \ref{subsec:OverviewOfArticle} below.
\begin{hyp}\label{hyp:Noise-Coefficients-Derivatives}
There exists an $\epsilon>0$ such that for each compact set $K \subseteq \R^{2d}$
\begin{align}
\tag{H3}\label{eq:Noise-Assumption-3} &\|\sigma\|_{\ell^2(\N;W^{1,2+\epsilon}(K))} = \Big(\sum_{k\in\N} \|\sigma_{k}\|_{W^{1,2+\epsilon}(K)}^2\Big)^{1/2} < \infty\\
\tag{H4}\label{eq:Noise-Assumption-4} & \|\sigma \cdot \nabla_v\sigma\|_{\ell^1(\N;W^{1,1+\epsilon}(K))} = \sum_{k\in\N} \|\sigma_k\cdot\nabla_v\sigma_k\|_{W^{1,1+\epsilon}(K)} < \infty.
\end{align}
\end{hyp}
The main result of this article is the following global existence theorem:
\begin{thm} \label{thm:MainResult}
  Let $\{ \sigma_{k}\}_{k \in \N}$ be a collection of noise coefficients satisfying Hypotheses \ref{hyp:Noise-Coefficients} and \ref{hyp:Noise-Coefficients-Derivatives} and assume that the collision kernel $b$ satisfies Hypothesis \ref{hyp:Collision-Kernel}. For any initial data $f_{0}: \R^{2d} \to \R_{+}$ satisfying
  \begin{equation}
    (1+|x|^2 + |v|^2 + |\log f_0|)f_0 \in L^1_{x,v},
  \end{equation}
  there exists a renormalized martingale solution to \eqref{eq:general-stochastic-Boltz}, starting from $f_{0}$ with noise coefficients $\{\sigma_k\}_{k\in\N}$.

  Moreover $f$ satisfies
  \begin{itemize}
  \item almost sure local conservation of mass
    \begin{equation}\label{eq:local-mass-conv}
       \partial_t\int_{\Rd} f \dv + \Div_x\int_{\Rd} v f\dx = 0,
     \end{equation}
   \item average global balance of momentum
     \begin{equation}\label{eq:global-momentum-bal}
       \E\iint_{\R^{2d}} v f(t)\dv\dx = \frac{1}{2}\sum_{k}\E\int_0^t\iint_{\R^{2d}}\sigma_k\cdot\nabla_v\sigma_k f(s)\dv\dx\ds + \iint_{\R^{2d}}v f_0\dv\dx,
     \end{equation}
   \item average global energy inequality
     \begin{equation}\label{eq:global-energy-ineq}
       \begin{aligned}
         &\E\iint_{\R^{2d}} \frac{1}{2}|v|^2f(t)\dv\dx \leq \sum_{k}\E\int_0^t\iint_{\R^{2d}}(v\cdot(\sigma_k\cdot\nabla_v\sigma_k) + |\sigma_k|^2)f(s)\dv\dx\ds\\
         &\hspace{1in}+ \iint_{\R^{2d}} \frac{1}{2}|v|^2f_0\dv\dx,
       \end{aligned}
     \end{equation}
   \item almost sure global entropy inequality
     \begin{equation}\label{eq:global-entropy-ineq}
       \iint_{\R^{2d}} f(t)\log{f(t)}\dv\dx + \int_0^t\int_{\R^{d}} \SDis(f)(s)\dx\ds \leq \iint_{\R^{2d}} f_0\log{f_0}\dv\dx.
     \end{equation}
   \end{itemize}
   The almost sure local conservation of mass holds $\P$ almost surely in distribution, the average global momentum and energy balances hold for every $t\in [0,T]$, and the global entropy inequality holds $\P$ almost surely for every $t\in [0,T]$.
\end{thm}

\subsection{Outline of the deterministic theory}
In order to motivate many of the steps of the proof of theorem \ref{thm:MainResult}, and make more explicit the similarities and differences to the deterministic $\sigma = 0$ theory of renormalized solutions for the Boltzmann equation. We will outline a rough sketch of the strategy of proof in the the determinisitic setting and emphasize several key steps. For a more detailed exposition of the deterministic theory, we direct the reader to the original work of DiPerna/Lions \cite{diperna1989cauchy} as well as some later improvements \cite{Lions1994-jf} with regard to dealing with passing limits in the collision operator. A detailed exposition of modern theory can be found in the notes \cite{Golse2005-zh} and a very nice sketch of the deterministic theory can be found in the book by Saint-Raymond \cite{MR2683475}.

To illustrate the existence theory for renormalized solutions, it suffices to prove a {\it stability result}. That is, suppose we have a sequence $\{f_n\}_{n\in\N}$ of renormalized solutions to the deterministic Boltzmann equation
\begin{equation}
    \partial_tf_n + v\cdot\nabla_xf_n = \mathcal{B}(f_n,f_n),
  \end{equation}
  with compact initial data and satisfying the uniform apriori estimates
  \begin{equation}\label{eq:deterministic-apriori}
    \sup_n\|(1+|x|^2 + |v|^2+|\log{f}_n|)f_n\|_{L^\infty_t(L^1_{x,v})} <\infty,\quad \sup_n\|\SDis(f_n)\|_{L^1_{t,x}} <\infty.
  \end{equation}
  Then we would like to show that any limit point $f$ of the sequence $\{f_n\}_{n\in\N}$ is also a renormalized solution to the Boltzmann equation and satisfies the same apriori estimates. Indeed, this was the approach initially take by DiPerna/Lions \cite{diperna1989cauchy}.

  The proof of the deterministic stability result can conceptually be broken into three main steps: a {\it weak compactness step} using the apriori estimates, a {\it strong compactness} step using the equation and velocity averaging Lemmas, and a {\it limit passage step} using both the weak and strong compactness to pass the limit in the renormalized equation.

  The first conceptual step is to obtain several weak compactness results using the apriori estimates (\ref{eq:deterministic-apriori}). Using the classical Dunford Pettis criterion, one can easily obtain from the first uniform estimate in (\ref{eq:deterministic-apriori}) that the sequence $\{f_n\}_{n\in\N}$ is weakly relatively compact in $[L^1_{t,x,v}]_{\loc}$. Furthermore, using the uniform entropy dissipation bound in (\ref{eq:deterministic-apriori}) and a bound on the gain part of the collision operator, one can show that the sequence
  \begin{equation}
    \left\{\frac{\BCol(f_n,f_n)}{1 + f_n}\right\}_{n\in\N}\quad \text{is weakly relatively compact in } L^1_{t,x,v}.
  \end{equation}

  The next step is to use the equation to deduce various strong compactness results using the equation. Here one chooses a sequence of renormalizations
  \begin{equation}
    \beta_\delta(z) =\frac{1}{\delta}\log(1+ \delta z)
  \end{equation}
  which approximate the identity $\beta_\delta(z) \to z$ as $\delta\to 0$. Since $\beta_{\delta}(z)$ is a suitable renormalizer then $\beta_\delta(f_n)$ solves
  \begin{equation}\label{eq:betadel-renorm-deterministic}
    \partial_t\beta_\delta(f_n) + v\cdot\nabla_x\beta_\delta(f_n) = \frac{\BCol(f_n,f_n)}{1+\delta f_n}.
  \end{equation}
  Using the weak compactness of $\left\{\frac{\BCol(f_n,f_n)}{1+\delta f_n}\right\}_{n\in\N}$ deduced from the apriori estimates, we may deduce that $\{\beta_\delta(f_ n)\}_{n\in\N}$ is relatively compact in $C_t([L^2_{x,v}]_{\w})$. Furthermore, we are now in a position to use the subtle regularizing effects present in the kinetic equation (\ref{eq:betadel-renorm-deterministic}) known as velocity averaging. Here, one can deduce that for every test function $\varphi(v)$ the velocity averages
  \begin{equation}
    \left\{\langle\beta_\delta(f_n)\,\varphi\rangle\right\}_{n\in\N}\quad\text{are strongly relatively compact in } L^1_{t,x},
  \end{equation}
  where the angle brackets denote integration in velocity. Using the weak relative compactness of $\{f_n\}_{n\in\N}$ one can show that $\beta_\delta(f_n)$ approximates $f_n$ as $\delta\to 0$ uniformly in $n$ and strongly in $L^1$, namely
  \begin{equation}\label{eq:uniform-approx}
    \lim_{\delta \to 0}\sup_n\|\beta_\delta(f_n) - f_n\|_{L^\infty_t(L^1_{x,v})} = 0.
  \end{equation}
  This allows one to deduce compactness results on $\{f_n\}_{n\in\N}$ from $\{\beta_\delta(f_n)\}_{n\in\N}$ concluding that $\{f_n\}_{n\in\N}$ is compact in $C_t([L^1_{x,v}]_{\w})$ and the velocity averages $\{\langle f_n\,\varphi\rangle\}_{n\in\N}$ are relatively compact in $L^1_{t,x}$.

  At this stage one has enough compactness to pass certain limits in the collision operators, suitably renormalized. Using the product limit Lemma \ref{lem:product_lemma}, one can deduce the following that, up to a subsequence, for each $\varphi \in C_c(\R^d)$,
  \begin{equation}\label{eq:Bcol-cont-est}
    \left\langle \frac{\BCol^{\pm}(f_n,f_n)}{1+\langle f_n\rangle},\varphi\right\rangle \to \left\langle \frac{\BCol^{\pm}(f,f)}{1+\langle f\rangle},\varphi\right\rangle \quad\text{in}\quad L^1_{t,x}.
  \end{equation}
The next step is to choose a bounded sequence of renormalizers,
  \begin{equation}
    \Gamma_m(z) = \frac{z}{m^{-1} + z},
  \end{equation}
  and using the weak compactness already deduced, conclude that
  \begin{equation}
    \Gamma_m(f_n) \to \overline{\Gamma_m(f)} \quad\text{in}\quad C_t([L^1_{x,v}]_\w)
  \end{equation}
  and
  \begin{equation}
    \Gamma^\prime_m(f_n)\BCol^{\pm}(f_n,f_n)\to \BCol_m^{\pm}\quad\text{in}\quad [L^1_{t,x,v}]_\w.
  \end{equation}
  Again, using the uniform estimate (\ref{eq:uniform-approx}), with $\Gamma_m(z)$ replacing $\beta_\delta(z)$, one can show the following {\it strong} convergence
  \begin{equation}\label{eq:strong-conv}
    \overline{\Gamma_m(f)} \to f \quad\text{in}\quad L^\infty_t(L^1_{x,v}).
  \end{equation}
  The idea is then to pass the limit in the equation
  \begin{equation}\label{eq:double-renorm-eq}
    \partial_t\log(1+\overline{\Gamma_m(f)}) + v\cdot\nabla_x\log(1+\overline{\Gamma_m(f)}) = \frac{\BCol_m^+}{1+\overline{\Gamma_m(f)}} - \frac{\BCol_m^-}{1+\overline{\Gamma_m(f)}}.
  \end{equation}
  The strong convergence (\ref{eq:strong-conv}) is enough to pass the limit on the left side in the sense of distribution. It remains to pass the limit on the right-hand side as $m\to \infty$. This is the most technical part of the analysis and requires limit (\ref{eq:Bcol-cont-est}) as well as all of the compactness and apriori estimates deduced on $\{f_n\}_{n\in\N}$. The main result is that
  \begin{equation}\label{eq:renomr-coll-anal}
    \frac{\BCol_m^{\pm}}{1+\overline{\Gamma_m(f)}} \to \frac{\BCol^{\pm}(f,f)}{1+f}\quad\text{in}\quad [L^1_{t,x,v}]_{\w},
  \end{equation}
  which allows one to pass the limit on the right-hand side of (\ref{eq:double-renorm-eq}) and conclude
  \begin{equation}\label{eq:log-renorm}
    \partial_t\log(1+f) + v\cdot\nabla_x\log(1+f) = \frac{\BCol(f,f)}{1+f}.
  \end{equation}
  It remains to show the the the equation also holds for any admissible renormalization $\Gamma(z)$.
  \begin{equation}
    \partial_t\Gamma(f) + v\cdot\nabla_x\Gamma(f) = \Gamma^\prime(f) \BCol(f,f).
  \end{equation}
  This can be done by renormalizing equation (\ref{eq:log-renorm}) with a truncation of
  \begin{equation}
    \beta(z) = \Gamma(e^{z}-1)
  \end{equation}
  and using the bound on $(1+f)^{-1}\BCol(f,f)$ to pass the limit in the truncation.


\subsection{Overview of the article}\label{subsec:OverviewOfArticle}

Now we proceed to an overview of the article, addressing the main difficulties overcome and the relation of our work to the existing literature on kinetic equations and stochastic PDE's.

Our analysis begins with formal a priori estimates which point to the natural functional framework for \eqref{eq:general-stochastic-Boltz}.  Namely, in Section \ref{sec:FormalBounds} we show that under the coloring Hypothesis (\ref{eq:Noise-Assumption-1}), solutions to \eqref{eq:general-stochastic-Boltz} formally satisfy
\begin{align}
  &\E\|(1+|x|^2 +|v|^2 + |\log{f}|)f\|_{ \LLL{\infty}{1}}^p \leq C_p, \label{apriori-estimates-overview-energy}\\
 &\E\|\SDis(f)\|_{\LLt{1}}^p \leq C_p.\label{apriori-estimates-overview-entropy}
\end{align}
With these formal a priori bounds at hand, the remainder of the paper splits roughly into two parts.  In Sections \ref{sec:Stoch-Kinetic-Renorm} and \ref{sec:Stoch-Velocity-Avg}, we analyze linear stochastic kinetic equations, while Sections $5-8$ are devoted to the proof of Theorem \ref{thm:MainResult}.   

In Sections \ref{sec:Stoch-Kinetic-Renorm} and \ref{sec:Stoch-Velocity-Avg} we move to a detailed discussion of stochastic kinetic equations of the form
\begin{equation} \label{eq:Intro_StochKin}
\begin{aligned}
&\partial_{t}f + v \cdot \nabla_{x}f + \Div_{v}(f\sigma_{k}\circ \dot{\beta}_{k})=g, \\
&f|_{t=0}=f_{0}.
\end{aligned}
\end{equation}
Here $f_{0} \in \LLs{1}$ is a deterministic initial density, while $g$ is a certain random variable with values in $\LLLs{1}$. We will focus on so-called weak martingale solutions to \eqref{eq:Intro_StochKin}. Roughly speaking (see Definition \ref{def:Martingale-Sol} of Section \ref{subsec:weak-martingale-sol} for the precise meaning), these are $L^{1}_{x,v}$ valued stochastic processes satisfying (\ref{eq:Intro_StochKin}) weakly in both the PDE and the probabilistic sense. In this context, probabilistically weak means that the filtered probability space $(\Omega,\mathcal{F},(\mathcal{F}_t)_{t=0}^T,\P)$ and the Brownian motions $\{\beta_{k}\}_{k \in \N}$ are not fixed in advance, but found as solutions to the problem, along with the process $f$ solving (\ref{eq:Intro_StochKin}) in the sense of distribution. 

For convenience we introduce the following language to refer to solutions of (\ref{eq:Intro_StochKin}), we say that: {\em $f$ is a solution to the stochastic kinetic equation {\it driven} by $g$ and starting from $f_{0}$, relative to the noise coefficients $\sigma$ and the stochastic basis $(\Omega, \mathcal{F}, \P, (\mathcal{F}_t)_{t=0}^T, \{\beta_k\}_{k\in\N})$}. In the case that the coefficients $\sigma$, the filtration $(\mathcal{F}_t)_{t=0}^{T}$, and the Brownian motions $\{\beta_k\}_{k\in\N}$ are implicitly known or irrelevant, we may omit them from the statement, saying instead:
{\em $f$ is a solution to the stochastic kinetic equation {\it driven} by $g$ and starting from $f_{0}$}.

A key workhorse for our analysis is a stability result (Proposition \ref{prop:stability-weak-martingale}) for weak martingale solutions to stochastic kinetic equations.  In the deterministic setting, this simply corresponds to the observation that the space of solutions to linear, kinetic equations is closed with respect to convergence in distribution. More precisely, if there are functions $\{f_n\}$, $\{f_0^n\}$ and $\{g_n\}$ satisfying  
\begin{equation} 
\begin{aligned}
&\partial_{t}f_{n} + v \cdot \nabla_{x}f_{n}=g_{n}, \\
&f|_{t=0}=f_{0}^{n},
\end{aligned}
\end{equation}
in the sense of distributions and $\{(f_{n},g_{n},f_{0}^{n}\}_{n \in \N}$ converges to $(f,g,f^{0})$ in the sense of distributions, then it readily follows from the linear structure of the equation that the limits $(f,f_0,g)$ also satisfy
\begin{equation} 
\begin{aligned}
&\partial_{t}f + v \cdot \nabla_{x}f=g,\\
&f|_{t=0}=f_{0},
\end{aligned}
\end{equation}
in the sense of distributions. In the stochastic framework, an additional subtlety arises.  Namely, one should distinguish between stability of stochastically strong solutions, where a stochastic basis has been fixed, and stability of stochastically weak solutions, where each solution comes equipped with its own stochastic basis. For a fixed stochastic basis $(\Omega, \mathcal{F},\p, (\mathcal{F}_t)_{t=0}^T, \{\beta_{k}\}_{k \in \N})$ and noise coefficients $\{\sigma_{k}\}_{k \in \N}$, one can use the linearity of $f \to \Div_{v}(f\sigma_{k}\circ \dot{\beta}_{k})$ together with a method of Pardoux \cite{pardoux1975equations} to make a direct passage to the limit on both sides of the equation. However, for stochastically weak solutions, the Brownian motions are not fixed, and the mapping  $(f, \beta_{k}) \mapsto \Div_{v}(f\sigma_{k}\circ \dot{\beta}_{k})$ is nonlinear, prohibiting the passage of weak limits. In this situation, a martingale method is used to overcome this difficulty and produce another weak martingale solution with a new stochastic basis. This result is detailed in Proposition \ref{prop:stability-weak-martingale}.

Section \ref{sec:Renormalization} is devoted to renormalizing weak martingale solutions to stochastic kinetic equations. The technique of renormalization of {\it deterministic} transport equations originates from the now classical results of Di'Perna and Lions \cite{diperna1989ordinary}, where they were able to show uniqueness to certain linear transport equations when the drift has lower regularity that the classical theory of characteristics would allow. Formally, the strategy is as follows: if $f$ satisfies (\ref{eq:Intro_StochKin}) and $\Renorm: \R \to \R$ is a smooth renormalization, then $\Renorm(f)$ satisfies
\begin{equation}\label{eq:Renormalized-Stochkin-Outline}
\begin{aligned}
  &\partial_t\Gamma(f) + v\cdot\nabla_x\Gamma(f) + \Div_v(\Gamma(f)\sigma_k\circ \dot{\beta}_k) = \Gamma^\prime(f)g,\\
  &\Gamma(f)|_{t=0} = \Gamma(f_0).
\end{aligned}
\end{equation}
If one can justify such a computation, then upon integrating both sides of the equation (\ref{eq:Renormalized-Stochkin-Outline}) for certain non-negative choices of $\Renorm(z)$ that vanish only at $z=0$, for instance $\Gamma(z) = z/(1+z)$, then one can get explicit bounds on $\Gamma(f)$ in terms of the initial data, which, by linearity, implies uniqueness. However, since we are working with analytically weak solutions to (\ref{eq:Intro_StochKin}), this formal calculation may fail if the individual $\sigma_{k}$ are too rough. In particular (to our knowledge), only requiring the $L^\infty$ coloring hypothesis (\ref{eq:Noise-Assumption-1}) is insufficient. The ability to renormalize stochastic kinetic transport equations will turn out to be a crucial property in the final stages of main existence proof. However, as in the case of the deterministic Boltzmann equation, it does not imply uniqueness of the equation, due to the nonlinear nature of the equation.


 Our strategy in Section \ref{sec:Renormalization} uses the method of DiPerna and Lions reduces the renormalizability of stochastic kinetic equations to the vanishing of certain commutators between smoothing operators and the differential action of the rough vector fields. Specifically, given a smooth renormalization $\Gamma(z)$ with bounded first and second derivatives, we begin by smoothing a solution $f$  to \eqref{eq:Intro_StochKin} in the $(x,v)$ variables with mollifier $\eta_{\ep}$. The regularity improvement allows us to renormalize the equation by $\Gamma$ at the expense of a remainder $R_\ep(f)$ comprised of commutators and double commutators of $\sigma_k\cdot\nabla_v$ and convolution by $\eta_\ep$,
\begin{equation}
   [\eta_\ep, \sigma_k\cdot\nabla_v](f), \quad \big[[\eta_\ep,\sigma_k\cdot\nabla_v], \sigma_k\cdot\nabla_v\big](f).
 \end{equation}
 As is well known from the classical theory of renormalization by \cite{diperna1989ordinary} that the single commutator
 \begin{equation}
   [\eta_\ep, \sigma_k\cdot\nabla_v](f)\xrightarrow[\ep\to 0]{} 0 \quad \text{in}\quad L^r_{x,v}
 \end{equation}
 as long as $\sigma \in W^{1,q}_{x,v}$ and $f\in L^p$ with $1/r = 1/p + 1/q$. As it turns out, the double commutator also vanishes
 \begin{equation}
   \big[[\eta_\ep,\sigma_k\cdot\nabla_v], \sigma_k\cdot\nabla_v\big](f) \xrightarrow[\ep\to 0]{} 0 \quad \text{in}\quad L^1_{x,v}
 \end{equation}
 provided that $\sigma_k \in W^{1,\frac{2p}{p-1}}_{x,v}$ and $\sigma_{k} \cdot \nabla_{v}\sigma_{k} \in W^{1,\frac{p}{p-1}}_{x,v}$. However one of the primary differences between the deterministic and stochastic theory is an interesting consequence of It\^{o}'s formula. Specifically the remainder $R_\ep(f)$ involves the square of the single commutator $[\eta_\ep, \sigma_k\cdot\nabla_v](f)$. Due to the limited integrability and regularity of $f$, this imposes that $p \geq 2$ and $\sigma_k \in W^{1,\frac{2p}{p-2}}_{x,v}$ for this contribution to vanish in $L^1$ (see Proposition \ref{prop:Weak_Is_Renormalized} for more details on this). Based on this method of proof, we are presently unable to treat the case $p \in [1,2)$. The main result of this section (Proposition \ref{prop:Weak_Is_Renormalized}) shows show that a weak martingale solution $f \in L^p(\Omega\times[0,T]\times\R^{2d})$, $p\geq 2$ to \eqref{eq:Intro_StochKin} is renormalizable provided we have the following regularity conditions on $\sigma$,
\begin{equation}\label{eq:renorm-sigma-conditions}
  \sigma \in \ell^2(\N; W^{1,\frac{2p}{p-2}}_{x,v})\quad\text{and}\quad \sigma\cdot\nabla_v\sigma \in \ell^1(\N; W^{1,\frac{p}{p-1}}_{x,v}).
\end{equation}
We believe these results are consistent with the work of Lions/Le-Bris \cite{bris2008existence} on deterministic parabolic equations with rough diffusion coefficients. There should also be a connection with the more recent work of Bailleul/Gubinelli \cite{bailleul2015unbounded}. In the case that $f\in L^{\infty-}(\Omega\times[0,T]\times\R^{2d})$, the conditions (\ref{eq:renorm-sigma-conditions}) become precisely the assumptions (\ref{eq:Noise-Assumption-3}) and (\ref{eq:Noise-Assumption-4}) on the noise coefficients.

Section \ref{sec:Stoch-Velocity-Avg} concerns the subtle regularizing effects for stochastic kinetic equations. These are captured by studying the velocity averages of the solution, and have a long history in the deterministic literature \cite{Bouchut1999-nh,golse1988regularity,golse2002velocity,jabin2004real} as well as several more recent results in the SPDE literature \cite{debussche2015invariant,gess2014long,Lions2011-2012}. Since equation (\ref{eq:Intro_StochKin}) is of transport type, without more information on $g$, one does not expect to obtain any further regularity on the solution $f$ than is present in the initial data $f_0$. However, in view of the deterministic theory it is natural to expect a small gain in the regularity of velocity averages
\begin{equation}
\langle f\,\phi \rangle := \int_{\Rd} f \phi \dv,
\end{equation}
where $\phi \in C^{\infty}_{c}(\Rd_v)$ is a test function in velocity only. Using a method of Bouchut/Desvillette \cite{Bouchut1999-nh} based on the Fourier transform, we prove that if $f$ is a weak martingale solution to \eqref{eq:Intro_StochKin} and $f,g\in L^2(\Omega\times[0,T]\times\R^{2d})$, then $\langle f, \phi\rangle$ enjoys the following regularity estimate,
\begin{equation}\label{eq:L2_Regularity_Est}
 \E\|\langle f\, \phi \rangle\|_{L^2_t(H^{1/6}_x)}^2\leq C_{\phi,\sigma}\big( \|f_0\|^2_{L^2_{x,v}} + \E\|f\|_{L^2_{t,x,v}}^2 + \E\|g\|_{L^2_{t,x,v}}^2\big).
\end{equation}
Combining this with a standard control on oscillations in time, one expects to obtain a form of strong compactness on the velocity averages. To formulate this directly in terms of $f$ rather than its velocity averages, we introduce a topological vector space $\LLsMw{p}$ consisting of the space of $L^p_{t,x}$ functions taking values in the space of Radon measures $\mathcal{M}_v^*$ on $\Rd_v$ endowed with it's weak-$\star$ topology.  The topology is designed so that sequential convergence in $\LLsMw{p}$ corresponds exactly to  
strong $L^p_{t,x}$ convergence of each sequence of velocity averages.
We prove a characterization of compact sets in $\LLsMw{p}$ in the appendix. Using the regularity gain in $L^2$, we exhibit a sufficient criterion for a sequence $\{f_n\}_{n\in\N}$ of weak martingale solutions to a stochastic kinetic equation driven by $\{g_n\}_{n\in\N}$ to induce tight laws on $[\LLsMw{2}]_{\loc}$. However, for applications to Boltzmann, one is mostly interested in the case where $\{g_{n}\}_{n \in \N}$ is only known to be uniformly bounded in $L^{1}(\Omega \times [0,T] \times \R^{2d})$, due to the very limited control provided by the a priori bounds on the renormalized collision operator $f \to \Renorm'(f)\mathcal{B}(f,f)$. The criteria for tightness in $\LLsMw{1}$ is the main result of Section \ref{sec:Stoch-Velocity-Avg}. 
As in the deterministic setting (see \cite{golse1988regularity,golse2002velocity}), there is no easily quantifiable regularity gain for $f, g \in L^1(\Omega\times[0,T]\times\R^{2d})$, making the analysis more involved.  At present, we can only treat well-prepared sequences of approximations for which the solution $f_{n}$ and the source $g_{n}$ are somewhat better behaved for fixed $n \in \N$.  This is captured by Hypothesis \ref{hyp:Stoch-Vel-Average-Fixed-n_hyp}.

At this point in the article, we have completed our analysis of the linear problem and proceed to apply our results from Section $3-4$ in the context of \eqref{eq:general-stochastic-Boltz}.  This begins in Section $5$ with a construction of a sequence of approximations $\{\tilde{f}_{n}\}_{n \in \N}$ satisfying a stochastic transport equation driven by a truncated collision operator
\begin{equation}
  \mathcal{B}_{n}(f,f) = \frac{\widehat{\BCol}_n(f,f)}{(1 + n^{-1}\langle f\rangle)}.
\end{equation}
This truncation was introduced in \cite{diperna1989cauchy} to make $\BCol_n(f,f)$ Lipschitz in $L^1_{t,x,v}$ while preserving it's conservation properties. After smoothing the noise coefficients and activating only finitely many Brownian motions, we obtain existence by way of the stochastic flow representation of Kunita \cite{kunita1997stochastic}, in combination with a fixed point argument.  The main subtleties in comparison to the deterministic theory are due to the fact that the flow map is not explicit.  To obtain the a priori bounds
\begin{equation}
  \sup_n\E\|(1+|x|^2 + |v|^2 + |\log{\tilde{f}_n}|)\tilde{f}_n\|^p_{\LLL{\infty}{1}} \leq C_p,\quad \sup_n\E\|\SDis_n(\tilde{f}_n)\|^p_{\LLLs{1}} \leq C_p,
\end{equation}
we require asymptotic growth estimates for the stochastic flow and a stopping time argument.  A similar difficulty arises in the work of Hofmanova \cite{hofmanova2015bhatnagar}.  An additional difference with the deterministic theory is that we do not prove that our approximations are of Schwartz class in position and velocity.  Instead, we use our renormalization lemma to establish the moment and entropy identities used in Section \ref{sec:FormalBounds}. 

Let us now discuss the main features of the existence proof for Theorem \ref{thm:MainResult} and some of the main difficulties. The main goal in sections $6-8$ is to extract an appropriate limit point $f$ on a well prepared stochastic basis $(\Omega, \mathcal{F},\p,(\mathcal {F}_t)_{t=0}^T,\{\beta_k\}_{k\in\N})$ and verify that $f$ is indeed a renormalized martingale solution to \eqref{eq:general-stochastic-Boltz}. This requires a somewhat involved combination of the renormalization and stochastic velocity averaging lemmas together with the general line of arguments introduced by DiPerna and Lions \cite{diperna1989cauchy} and a later work of Lions \cite{Lions1994-jf}. The argument requires a careful interpretation in the stochastic framework. We study the laws of the sequence $\{\tilde{f}_{n}\}_{n \in \N}$ and use the velocity averaging and renormalization lemmas to show they are tight on $\LLsMw{1}\cap \CLw{1}$. A generalization of the Skorohod theorem due to Jakubowski \cite{jakubowski1997non} and Vaart/Wellner \cite{van1996weak} gives a candidate limit $f$, which we endeavor to show is a renormalized martingale solution to \eqref{eq:general-stochastic-Boltz}.  The Skorohod theorem allows one to gain compactness of the nonlinear drift terms at the expense of the noise terms.  Indeed, additional oscillations are introduced in the noise terms after switching probability spaces as $\Div_{v}(\sigma_{k}^{n}f_{n}\dot{\beta}_{k})$ is replaced by $\Div_{v}(\sigma_{k}^{n}\tilde{f}_{n}\dot{\tilde{\beta}}^{n}_{k})$, at which point we are setup to apply our weak stability result.  However, this is done in a somewhat indirect way.

The procedure of identifying $f$ with a solution of (\ref{eq:general-renorm-stochastic-Boltz}) requires two conceptually different steps. First, in Section $6$ we fix a bounded renormalization $\Renorm_{m}$ which converges to the identity as $m \to \infty$. With $m$ fixed, we check the criterion necessary to apply our weak stability result to the sequence $\{\Renorm_{m}(f_{n})\}_{n \in \N}$. This sequence is also shown to induce tight laws on $\LLsMw{1}\cap \CLw{1}$.  The stability result implies its limit point $\overline{\Renorm}_{m}$ is a solution to a stochastic kinetic equation with a driver $\mathcal{B}_{m}$. To show this requires analysis of the laws induced by the sequence of renormalized collision operators $\{\Renorm_{m}'(f_{n})\mathcal{B}_{n}(f_{n},f_{n})\}_{n \in \N}$.

At this stage, we do not yet have any sort of closed evolution equation for $\overline{\Renorm_m(f)}$. Indeed, it is unclear the relation between $\overline{\Renorm}_{m}$ and $\mathcal{B}_{m}$. Hence, our next step is to pass $m \to \infty$ and hope to obtain a closed evolution equation in the limit. As a result of the initial renormalization procedure $\overline{\Gamma_m(f)}$ converges strongly to $f$ in $\LLL{\infty}{1}$, $\P$ almost surely. Unfortunately, as $m \to \infty$ one does not have any good control on $\{\mathcal{B}_{m}\}_{m \in \N}$ in any space of distributions (only in the topology of measurable functions, which does not play well with the weak form). On the other hand, we do have control of $\{(1 + \overline{\Renorm_{m}(f)})^{-1}\mathcal{B}_{m}\}_{m \in \N}$. Hence, the strategy is to renormalize again, this time with $\log(1+z)$, and apply again our stability result in the limit $m \to \infty$.

Section \ref{sec:Renorm-Collision-Analysis} is dedicated to analysis of the renormalized collision operator $\mathcal{B}_m$.  As in the deterministic setting, we are able to obtain a pointwise (in $\Omega$) continuity result
\begin{equation}
  \frac{\BCol_n(f_{n},f_{n})}{1+\langle f_{n}\rangle} \to \frac{\BCol(f,f)}{1+\langle f\rangle} \quad \text{in} \quad L^{1}_{t,x}(\mathcal{M}_{v}^{*}),
\end{equation}
as a consequence of the velocity averaging lemmas. Following the strategy in \cite{Lions1994-jf} and \cite{Golse2005-zh}, we are able to conclude that
\begin{equation}
  \frac{\mathcal{B}_{m}}{1+\overline{\Renorm_{m}(f)}}\, \to  \, \, \frac{\mathcal{B}(f,f)}{1+f}\quad \text{in} \quad L^2(\Omega; [\LLLs{1}]_w),
\end{equation}
allowing us to apply again the stability result.

We are then able to deduce that $\log(1+f)$ is a solution to a stochastic kinetic equation driven by $(1+f)^{-1}\BCol(f,f)$.  Roughly speaking, the final step is verify the renormalized form of \eqref{eq:general-stochastic-Boltz} with an arbitrary renormalization.   Since $\log(1+f) \in L^{\infty-}(\Omega\times[0,T]\times\R^{2d})$, the conditions on the noise coefficients (\ref{eq:Noise-Assumption-3}) and (\ref{eq:Noise-Assumption-4}) are exactly such that the renormalization Lemma \ref{prop:Weak_Is_Renormalized} applies whereby, one renormalized by a truncation of $\beta(z) = \Gamma(e^{z}-1)$ and passes the limit in the truncation.


\section{Preliminaries}\label{sec:FormalBounds}

\subsection{Basic properties of the collision operator}
In this section, we recall some basic properties of the collision operator $f \to \mathcal{B}(f,f)$ (defined in \eqref{eq:Boltzmann-Coll-Def}) which will be used throughout the article.  A more in-depth discussion can be found in \cite{diperna1989cauchy}.  To begin, we note that the collision operator  may be split into gain and loss terms
\begin{equation}
  \BCol(f,f) = \BCol^+(f,f) - \BCol^-(f,f),
\end{equation}
with
\begin{equation}
  \BCol^+(f,f) = \iint_{\R^d\times\S^{d-1}} f^\prime f^\prime_*b(v-v_*,\theta)\dee\theta\dee v_*,\quad \BCol^-(f,f) = f(\overline{b}* f),
\end{equation}
and $\overline{b}$ defined by
\begin{equation}
  \overline{b}(z) = \int_{\S^{d-1}} b(z,\theta)\dee\theta.
\end{equation}
The following inequality due to Arkeryd \cite{arkeryd1984loeb} relates the positive and negative parts of the collision operator through the entropy dissipation. Namely, for $K >1$ and $f\in L^1_v$, it holds
\begin{equation}\label{eq:AckerydBd}
\BCol^+(f,f)(v) \leq  K \BCol^-(f,f)(v) + \frac{1}{\log{K}}\SDis^0(f)(v),  
\end{equation}
where $\SDis^0(f)$ is defined by
\begin{equation}
  \SDis^0(f)  = \frac{1}{4} \iint_{\Rd\times \S^{d-1}} d(f)b(v-v_*,\theta)\dee \theta \dee v_* .
\end{equation}
Note that the quantity $\SDis^0(f)$ is {\it not} the entropy dissipation $\SDis(f)$ as defined in (\ref{eq:Entopy-Dissipation-Def}), but is instead related to $\SDis(f)$ by an integration in $v$, 
\begin{equation}
  \SDis(f) = \int_{\Rd} \SDis^0(f) \dee v.
\end{equation}




\subsection{Formal a priori estimates}

In this section, we will derive formal a priori estimates on the stochastic Boltzmann equation (\ref{eq:general-stochastic-Boltz}) with $\{\sigma_k\}_{k\in \N}$ satisfying (\ref{eq:Noise-Assumption-1}) and initial data $f_0$ satisfying
 \begin{equation}
\|(1+|x|^2 + |v|^2 + |\log{f_0}|)f_0\|_{\LLs{1}}^p < \infty.
\end{equation}
Specifically we will see that under these assumptions, there exists a positive constant $C \equiv C_{p,\sigma,T,f_0}$, depending on $p$, $\{\sigma_k\}_{k\in\N}$, $T$, and $f_0$ such that
\begin{equation}\label{eq:formal-apriori-estimate}
\E \|(1+|x|^{2}+|v|^{2}+|\log f|)f\|_{\LLL{\infty}{1}}^{p} \leq C.
\end{equation}
In addition the entropy dissipation $\SDis(f)$ satisfies 
\begin{equation}\label{eq:formal-dissipation-bound}
  \E\|\SDis(f)\|_{\LLt{1}}^p \leq C.
\end{equation}

These a priori estimates are completely natural in the context of the deterministic Boltzmann equation and correspond to the physical assumptions of finite mass, momentum, energy, entropy, and entropy production (see for instance \cite{Cercignani2013-vz} or \cite{Golse2005-zh}).

Throughout the argument $C$ will denote a positive, finite constant that depends on $p$, $\{\sigma_k\}_{k\in\N}$, $T$ and $f_0$. It may change from line to line, and even within a line. 

\subsubsection{Moment Bound}\label{subsec:MomentBound}

We begin by showing that
\begin{equation}\label{eq:v^2-Estimate}
  \E \|(1+ |x|^2 + |v|^2)f\|_{\LLL{\infty}{1}}^p \leq C,
\end{equation}
for $p >2$. To this end, we multiply the Boltzmann equation by $(1+|x|^2 +|v|^2)$ in It\^{o} form and integrate over $[0,t]\times\R^d_x\times\Rd_v$ to obtain
\begin{equation}\label{eq:Global-v^2-boltz}
\begin{split}
&\frac{1}{2}\iint_{\R^{2d}} (1+|x|^2 + |v|^2)f_t\,\dv\dx  = \frac{1}{2}\iint_{\R^{2d}}(1+|x|^2 + |v|^2)f_0\,\dv\dx\\
&\hspace{.5in} + \int_0^t \iint_{\R^{2d}}\sum_{k\in\N}|\sigma_k|^2 f_s\,\dx\dv\ds\\
&\hspace{.5in} + \int_0^t\iint_{\R^{2d}}\Big(\sum_{k\in\N}(\sigma_k\cdot\nabla_v \sigma_k)+x\Big)\cdot v f_s\,\dx\dv\ds\\
&\hspace{.5in} + \sum_{k\in\N}\int_0^t\Big(\int_{\R^{2d}}v\cdot\sigma_k\,f_s\,\dx\dv\Big)\dee \beta_k(s).
\end{split}
\end{equation}
Applying Cauchy-Schwartz to the time integral the following estimate readily follows,
\begin{equation}\label{eq:v^2-est-1}
\begin{aligned}
  &\Big|\int_0^t\iint_{\R^{2d}}\Big(\sum_{k\in\N}(\sigma_k\cdot\nabla_v \sigma_k)+x\Big)\cdot v f_s\,\dx\dv\ds\Big|^p\\
&\hspace{1in}\leq C \|\sigma\cdot\nabla_v\sigma\|_{\ell^1(\N;\LLs{\infty})}^p\int_0^t\|(1+|x|^2+|v|^2)f_s\|_{\LLs{1}}^p\ds,
\end{aligned}
\end{equation}
and similarly 
\begin{equation}\label{eq:v^2-est-2}
  \begin{aligned}
    \Big|\int_0^t\int_{\R^{2d}} \sum_{k\in\N} |\sigma_k|^2 f_s\,\dx\dv\ds\Big|^p \leq C \int_0^t \|(1+|x|^2+|v|^2)f_s\|_{\LLs{1}}^p\,\ds.
  \end{aligned}
\end{equation}
For the stochastic integral term in (\ref{eq:Global-v^2-boltz}), the BDG (Burkholder-Davis-Gundy) inequality yields
\begin{equation}
  \E\bigg|\sup_{r\in[0,t]}\sum_{k\in\N}\int_0^r\Big(\int_{\R^{2d}} v\cdot\sigma_k\, f_s\,\dv\dx\Big)\dee \beta_k(s)\bigg|^p \leq \E\bigg(\int_0^t\sum_{k\in\N}\Big(\int_{\R^{2d}} \sigma_k\cdot v f_s\,\dv\dx\Big)^2\ds\bigg)^{p/2}.
\end{equation}
Therefore, after another application of Cauchy-Schwartz to the time integral, we conclude
\begin{equation}\label{eq:v^2-est-3}
\begin{aligned}
    &\E\bigg|\sup_{r\in[0,\,t]}\sum_{k\in\N}\int_0^r\Big(\int_{\R^{2d}} v\cdot\sigma_k\, f_s\,\dv\dx\Big)\dee \beta_k(s)\bigg|^p\\
&\hspace{1in} \leq C\|\sigma\|_{\ell^2(\N;\LLs{\infty})}^p\int_{0}^t\E\|(1+|x|^2 + |v|^2)f_s\|_{\LLs{1}}^p\ds
\end{aligned}
\end{equation}
We may now combine estimates (\ref{eq:v^2-est-1}), (\ref{eq:v^2-est-2}) and (\ref{eq:v^2-est-3}) with~(\ref{eq:Global-v^2-boltz}) to obtain
\begin{equation}
  \E\bigg( \sup_{r \in [0,t]} \|(1+ |x|^2+ |v|^2)f_r\|_{\LLs{1}}\bigg)^p \leq C + C\int_{0}^t\E\bigg(\sup_{r\in[0,\,s]}\|(1+|x|^2+|v|^2)f_r\|_{\LLs{1}}\bigg)^p\ds. 
\end{equation}
Whereby Gr\"{o}nwall's Lemma gives (\ref{eq:v^2-Estimate}).

\subsection{Entropy Bound}\label{subsec:Entropy-Bound}

Next, we show that
\begin{equation}\label{eq:log-Estimate}
  \E\|f\log{f}\,\|_{\LLL{\infty}{1}}^p \leq C. 
\end{equation}
This estimate, as in the deterministic case, is comprised of two parts, control of the entropy $f\log{f}$ from above by the entropy dissipation (\ref{eq:entropy-dis}) and control of $f\log{f}$ from below using estimates (\ref{eq:v^2-Estimate}) and a Maxwellian. Specifically, integrating the entropy dissipation law (\ref{eq:entropy-dis}) in $[0,t]\times\Rd_x$ gives the $\P$ almost sure identity, for each $t\in[0,T]$,
\begin{equation}\label{eq:entropy-identity}
  \int_{R^{2d}} f_t\log{f_t}\dv\dx = \int_{R^{2d}} f_0\log{f_0}\dv\dx - \int_{0}^t\int_{\Rd} \SDis(f_s)\dx\ds,
\end{equation}
and since $\SDis(f) \geq 0$, this yields the classical entropy inequality,
\begin{equation}\label{eq:entropy-intequality}
  \int_{\R^{2d}} f_t\log{f_t}\dv\dx\leq   \int_{\R^{2d}} f_0\log{f_0}\,\dv\dx. 
\end{equation}
Using this and standard estimates from kinetic theory (see \cite{Cercignani2013-vz}), we obtain $\P$ almost surely
\begin{equation}
\begin{aligned}
  \int_{\R^{2d}}f_t|\log{f_t}|\dv\dx &\leq \int_{R^{2d}} f_t\log{f_t}\dx\dv + 2\int_{\R^{2d}} (|x|^2 + |v|^2)f_t\dv\dx\\
&\hspace{.5in} + 2\frac{\log{e}}{e}\int_{\R^{2d}}e^{-\frac{1}{2}( |x|^2 + |v|^2)}\,\dv\dx\\
&\leq \|f_0\log{f_0}\|_{\LLs{1}} + C\|(1 +|x|^2 + |v|^2)f_t\|_{\LLs{1}} + C.
\end{aligned}
\end{equation}
Applying the previous estimate on $(1+ |x|^2 + |v|^2)f$ to the above inequality gives the desired estimate of $f\log{f}$.

\subsubsection{Dissipation Bound}\label{subsec:Dissipation-Bound}

Finally with regard to the entropy dissipation estimate (\ref{eq:formal-dissipation-bound}), observe that equation (\ref{eq:entropy-identity}) also implies the $\P$ almost sure bound
\begin{equation}
  \|\SDis(f)\|_{\LLt{1}}\leq \|f\log{f}\|_{\LLLs{1}} + \|f_0\log{f_0}\|_{\LLs{1}},
\end{equation}
from which the estimate (\ref{eq:formal-dissipation-bound}) clearly follows.


\section{Stochastic Kinetic Transport Equations}\label{sec:Stoch-Kinetic-Renorm}
In this section, we assume that a probability space $(\Omega,\mathcal{F},\p)$ is given,  together with a deterministic initial condition $f_{0} \in \LLs{1}$ and a random variable $g\in L^{1}(\Omega;\LLLs{1})$.  Moreover, we have a collection of noise coefficients $\{\sigma_k\}_{k\in\N}$ satisfying the coloring Hypothesis \ref{hyp:Noise-Coefficients}. We analyze properties of solutions to stochastic kinetic equations of the type
\begin{equation} \label{eq:StochKin}
\begin{aligned}
&\partial_{t}f + v \cdot \nabla_{x}f + \Div_{v}(f\sigma_{k}\circ \dot{\beta}_{k})=g \\
&f \mid_{t=0}=f_{0},
\end{aligned}
\end{equation}
where solutions are understood in the weak martingale sense, given precisely in Definition \ref{def:Martingale-Sol} below. 

\subsection{Weak martingale solutions}\label{subsec:weak-martingale-sol}

\begin{definition}[Weak Martingale Solution] \label{def:Martingale-Sol} A process $f:[0,T]\times \Omega \to \LLs{1}$ is a weak martingale solution of the stochastic kinetic transport equation driven by $g$ with initial data $f_0$, provided the following is true:
\begin{enumerate}
\item \label{item:continuity} For all $\varphi \in C^{\infty}_{c}(\R^{2d})$, the process $\langle f, \varphi \rangle : \Omega \times [0,T] \to \R$ admits $\p$ a.s. continuous sample paths.  Moreover, $f$ belongs to $L^{2}(\Omega; L^\infty_t(\LLs{1}))$. 
\item \label{item:stoch-basis} There exists a collection of Brownian motions $\{ \beta_{k}\}_{k \in \N}$ and a filtration $(\mathcal{F}_t)_{t=0}^T$ so that the $[L^{1}_{x,v}]_w$ valued processes $(f_{t})_{t=0}^{T}$, $(\int_{0}^{t}g_{s}ds)_{t=0}^{T}$ and each Brownian motion $(\beta_{k}(t))_{t=0}^{T}$ are adapted to $(\mathcal{F}_t)_{t=0}^T$.
\item \label{item:martingaleCond} For all test functions $\varphi\in C^\infty_c(\R^{2d})$, the process $(M_{t}(\varphi))_{t=0}^{T}$ defined by
\begin{equation}\label{eq:Martingale-To-Check}
M_{t}(\varphi)=\iint_{\R^{2d}}f_t\varphi \dx \dv - \iint_{\R^{2d}}f_{0}\varphi \dx \dv 
- \int_{0}^{t}\iint_{\R^{2d}}f(v \cdot \nabla_{x}\varphi+\mathcal{L}_{\sigma}\varphi)+g\varphi \, \dx\dv\ds
\end{equation}
is an $(\mathcal{F}_{t})_{t=0}^{T}$ martingale. Moreover, its quadratic variation and cross variation with respect to each $\beta_{k}$ are given by:
\begin{align}
\dblbrak{M(\varphi)}{M(\varphi)}_{t}&= \sum_{k\in\N}\int_{0}^{t}\Big(\iint_{\R^{2d}}f_s \sigma_{k} \cdot \nabla_{v}\varphi \dx \dv\Big)^{2}\ds.\\
 \dblbrak{M(\varphi)}{\beta_{k}}_{t}&=\int_{0}^{t}\iint_{\R^{2d}}f_s \sigma_{k} \cdot \nabla_{v}\varphi \dx \dv \ds.
\end{align} 
\end{enumerate}
\end{definition}

\begin{remark} \label{rem:stoch-integral-form-of-eq}
Note that if $f$ is a martingale solution to a stochastic kinetic equation driven by $g$ and starting from $f_{0}$ relative to the stochastic basis $(\Omega, \mathcal{F}, \P, (\mathcal{F}_t)_{t=0}^T,\{ \beta_{k}\}_{k \in \N})$, then for all $t \in [0,T]$ the following identity holds $\p$ almost surely
  \begin{equation}\label{eq:weak-martingale-stochastic-integral}
  \begin{split}
    &\iint_{\R^{2d}}f_{t}\varphi \dx \dv = \iint_{\R^{2d}}f_{0}\varphi \dx \dv 
+ \int_{0}^{t}\iint_{\R^{2d}}[f_{s}(v \cdot \nabla_{x} + \LStrat_\sigma)\varphi+ g_{s}\varphi]\dx\dv\ds \\
&\hspace{1in}+\sum_{k\in\N}\int_{0}^{t}\iint_{\R^{2d}}f_{s}\sigma_{k}\cdot \nabla_{v}\varphi \dx\dv\dee\beta_{k}(s).
  \end{split}
\end{equation}
This is guaranteed by Lemma \ref{Lem:Appendix:Three_Martingales_Lemma} of the appendix.
\end{remark}

\begin{remark}
  The definition of weak martingale solution is a subtle one and deserves some discussion. A martingale solution to (\ref{eq:StochKin}) involves finiding a process $f_t$ along with a stochastic basis $(\Omega,\mathcal{F},(\mathcal{F}_t)_{t=0}^T,\{\beta_k\}_{k\in\N},\P)$. This is in contrast to {\it stochastically strong} solutions, which involve finding solutions for a given cannonical stochastic basis.


  The reason for considering the more general weak martingale solutions is that the renormalized solutions to the Boltzmann equation obtained in this paper are weak martingale solutions to (\ref{eq:StochKin}) in the sense that $\Gamma(f)$ solves the stochasatic kinetic equation with $g = \Gamma^\prime(f)\BCol(f,f)$, and the stability result Proposition \ref{prop:stability-weak-martingale} of this section will need to be applied in a setting where each solution is a weak martingale solution to (\ref{eq:StochKin}).
\end{remark}
The following existence result may be proved with a small modification to the arguments given in \cite{flandoli2010well} (which use a strategy developed already in the Ph.D thesis of E. Pardoux \cite{pardoux1975equations}).
\begin{thm}[Existence]\label{thm:existence-Stoch-Kin}
Let $\{\beta_k\}_{k\in\N}$ be a given collection of $(\mathcal{F}_t)_{t=0}^T$ Brownian motions. Assume that $\{\sigma_k\}_{k\in\N}$ satisfies hypothesis (\ref{eq:Noise-Assumption-1}) and that $g\in L^{\infty}(\Omega;\LLLs{1}\cap \LLLs{\infty})$, with $(\int_0^t g_s \ds)_{t=0}^T$ an $(\mathcal{F}_t)_{t=0}^T$ adapted process. If $f_0 \in L^p_{x,v}$ for $p\in[1,\infty]$, then there exists a weak martingale solution $f$ (relative to the given stochastic basis) to the stochastic kinetic equation driven by $g$ with initial data $f_0$.  Moreover, we have the following estimate for every $p \in [1,\infty)$,
  \begin{equation}
    \E\|f\|_{\LLLs{p}}^p \leqc \|f_0\|_{\LLs{p}}^p + \E\|g\|_{\LLLs{p}}^p.
  \end{equation}
\end{thm}
The next result is a time regularity estimate.
\begin{lem}\label{lem:time-regularity}
Let $q \in (2,\infty]$ and assume $f\in L^{\infty-}(\Omega;L^q_t(\LLs{1}))$ is a weak martingale solution to the stochastic kinetic transport equation driven by $g\in L^{\infty-}(\Omega;L^q_t(\LLs{1}))$ with with initial data $f_0\in \LLs{1}$. Then for any test function $\varphi \in C^\infty_c(\R^{2d})$ and $p \in (\frac{2q}{q-2},\infty)$, we have the following estimate
\begin{equation}
\E \| \langle f, \varphi \rangle \|_{W^{\gamma,p}_{t}}^{p} \leq C_{\varphi,\sigma}\Big (\E\| f\|_{\LLL{q}{1}}^{p}+\E\|g\|_{\LLL{q}{1}}^{p} \Big),
\end{equation}
where $\gamma = \frac{1}{2} - \frac{1}{p} - \frac{1}{q}$.
\end{lem}
\begin{proof} Consider two times $t,s\in \R_+$, $t\neq s$. Writing (\ref{eq:Renorm-Weak-Form}) in It\^{o} form, we can conclude that the difference $\langle f_t - f_s,\varphi \rangle$ satisfies
  \begin{equation}
    \begin{split}
    \langle f_t - f_s, \varphi\rangle &= \int_s^t\iint_{\R^{2d}}(v\cdot\nabla \varphi + \LStrat_\sigma\varphi)f\,dx\dv\dr + \int_s^t\iint_{\R^{2d}}\varphi g\,\dx\dv\ds\\
&\hspace{.5in}+ \sum_{k\in\N} \int_s^t\Big(\iint_{\R^{2d}} f\sigma_k\cdot\nabla_v\varphi\,\dx\dv\Big)\dee\beta_k(r).
    \end{split}
  \end{equation}
We would like to estimate $\E|\langle f_t - f_s ,\varphi\rangle|^p$.  To this end, since $v\cdot \nabla\varphi + \LStrat_{\sigma}\varphi \in \LLs{\infty}$, we have the estimate
\begin{equation}\label{eq:time-reg-est-1}
  \Big|\int_s^t\iint_{\R^{2d}}(v\cdot\nabla \varphi + \LStrat_\sigma\varphi)f\,dx\dv\dr\Big|^p \leq C_{\varphi,\sigma} |t-s|^{p(1-\frac{1}{q})}\|f\|_{\LLL{q}{1}}^p,
\end{equation}
and similarly
\begin{equation}\label{eq:time-reg-est-2}
  \Big|\int_s^t\varphi g\,\dx\dv\ds\Big|^p \leq C_{\varphi}|t-s|^{p(1-\frac{1}{q})}\|g\|_{\LLL{q}{1}}^p.
\end{equation}
By the BDG inequality we may estimate the martingale term by
\begin{equation}\label{eq:time-reg-est-3-BDG}
  \begin{aligned}
  \E\bigg|\sum_{k\in\N} \int_s^t\Big(\iint_{\R^{2d}} f\sigma_k\cdot\nabla\varphi \dx\dv\Big)\dee\beta_k(r)\bigg|^p &\leq \E\bigg(\int_s^t \sum_{k\in\N}\Big (\iint_{\R^{2d}} f\sigma_k \cdot\nabla \varphi \dx\dv\Big)^2\dr\bigg)^{p/2}\\
&\leq C_{\varphi,\sigma}|t-s|^{\frac{p}{2}(1-\frac{2}{q})}\E\|f\|_{\LLL{q}{1}}^p.
  \end{aligned}
\end{equation}
Combining these estimates gives
\begin{equation}
  \E|\langle f_t-f_s,\varphi\rangle|^p \leq C_{\varphi,\sigma}|t-s|^{p(\frac{1}{2}-\frac{1}{q})}\Big(|t-s|^{\frac{p}{2}}\big(\E\|f\|_{\LLL{q}{1}}^p +\E\|g\|_{\LLL{q}{1}}^p\big) + \E\|f\|_{\LLL{q}{1}}^p\Big).
\end{equation}
We now estimate the regularity of $\langle f,\varphi\rangle$ via the Sobolev-Slobodeckij semi-norm $[\cdot]_{W^{\gamma,p}_t}$. For  \\ $\gamma p + 1=p(\frac{1}{2}-\frac{1}{q})$ we find 
\begin{equation}
  \E[\langle f,\varphi\rangle]_{W^{\gamma,p}_t}^p = \int_0^T\int_0^T \frac{\E|\langle f_t - f_s,\varphi\rangle|^p}{|t-s|^{\gamma p + 1}}\ds\dt \leq C_{\varphi,\sigma} \Big(\E\|f\|_{\LLL{\infty}{1}}^p + \|g\|_{\LLL{\infty}{1}}^p\Big).
\end{equation}
\end{proof}

\subsection{Stability of weak martingale solutions}

In this section, we establish our main stability result for sequences of weak martingale solutions to stochastic kinetic equations.  The result below will be used repeatedly throughout the article.
\begin{prop}\label{prop:stability-weak-martingale}
Let $f:\Omega \times [0,T] \to L^{1}_{x,v}$ be a stochastic process and $\{\beta_{k}\}_{k \in \N}$ be a collection of Brownian motions.  Assume there exists a sequence of processes $\{f_{n}\}_{n \in \N}$ with the following properties.  
\begin{enumerate}
\item For each $n \in \N$ there exist $g_{n},f_{n}^{0},$ and $\sigma^{n} = \{\sigma^n_k\}_{k\in\N}$ such that $f_{n}$ is a weak martingale solution to a stochastic kinetic equation driven by $g_{n}$ with initial data $f_{n}^{0}$, relative to the noise coefficients $\sigma^{n}$ and the stochastic basis $(\Omega, \mathcal{F}, \P, (\mathcal{F}^n_t)_{t=0}^T,\{\beta_{k}^{n}\}_{k \in \N})$.
\item The sequences $\{f_{n}\}_{n \in \N}$ and $\{g_{n}\}_{n \in \N}$ are bounded in $L^{2+}(\Omega ; \LLL{\infty}{1})$ and $L^{2}(\Omega ; \LLLs{1})$ respectively. Assume that $\p$ almost surely,
 \begin{equation} \label{eq:conv-at-fixed-time}
f_{n} \to f \quad \text{in} \quad C_{t}([L^{1}_{x,v}]_{w}).
\end{equation}
Moreover, for each $\varphi \in C^{\infty}_{c}(\R^{d})$ and each $t \in [0,T]$, 
 \begin{equation} \label{eq:conv-time-ints}
\Big \langle \int_{0}^{t}g_{n}(s)ds, \varphi \Big \rangle \to  \Big \langle \int_{0}^{t}g(s)ds, \varphi \Big \rangle \quad \text{in} \quad L^{2} ( \Omega ).
\end{equation}
\item As $n \to \infty$, the following convergences hold:
  for each $k\in\N$,
  \begin{equation}
\begin{aligned}
\beta^{n}_{k} \to \beta_{k} \quad &\text{in} \quad L^{2} \left ( \Omega ; C_{t} \right ), \\
f_{n}^{0} \to f^{0} \quad & \text{in} \quad \LLs{1}. \\
\end{aligned}
\end{equation}
\item The sequences $\{\sigma^{n}\}_{n \in \N}$ and $\{\nabla_{v}\sigma^{n}\}_{n \in \N}$ are uniformly bounded in $\ell^{2}(\N; L^{\infty}_{x,v})$, converge pointwise a.e. on $\R^{2d}$ to $\sigma$ and $\nabla\sigma$ and 
\begin{equation}\label{c5}
\lim_{N \to \infty}\sup_{n \in \N}\sum_{k=N}^{\infty}\|\sigma^{n}_{k}\|_{L^{\infty}_{x,v}}^{2}=0
\end{equation}
\end{enumerate}
Under these hypotheses, we may deduce that $f$ is a weak martingale solution driven by $g$ and starting from $f_{0}$, relative to the noise coefficients $\sigma$ and the Brownian motions $\{\beta_{k}\}_{k \in \N}$.  

Moreover, if $(\Omega, \mathcal{F}, \P, (\mathcal{F}_t^{n})_{t=0}^T,\{\beta_{k}^{n}\}_{k \in \N})$ is independent of $n \in \N$, then $f$ can be built with respect to the same stochastic basis.
\end{prop}
\begin{proof}
Define a collection of topological spaces $(E_{t})_{t=0}^{T}$ by $E_{t}=C \big ( [0,t]; [\LLs{1}]_{w}^{2} \big) \times C[0,t]^{\infty}$.  Let $r_{t}:E_{T} \to E_{t}$ be the corresponding restriction operators.  Next define the $\LLs{1}$ valued processes $(G_{t})_{t=0}^{T}$ and $(G^{n}_{t})_{t=0}^{T}$ to be the running time integrals (starting from $0$) of $g$ and $g_{n}$, respectively.  Use these to define the $E_{T}$ valued random variables $X=\left (f, G, \{\beta_{k}\}_{k \in \N} \right)$ and $X_{n}=\left (f_{n}, G_{n}, \{\beta_{k}^{n}\}_{k \in \N} \right)$.

 We will verify that $f$ is a weak martingale solution relative to the filtration $(\mathcal{F}_{t})_{t=0}^{T}$ given by $\mathcal{F}_{t}=\sigma(r_{t}X)$.  With this filtration, Part 1 of Definition \ref{def:Martingale-Sol} certainly holds.  Part 2 is true by assumption.  Hence, if suffices to verify Part $3$.  Let $\varphi \in C^{\infty}_{c}(\R^{2d})$ and define the continuous process $(M_{t}(\varphi))_{t=0}^{T}$ by \eqref{eq:Martingale-To-Check}.  Let $s<t$ be two times and suppose that $\gamma \in C_{b}(E_{s};\R)$.  It suffices to show
\begin{align}
&\E \Big ( \gamma(r_{s}X) \big ( M_{t}(\varphi)-M_{s}(\varphi) \big )  \Big)
=0, \label{eq:Mart}\\
&\E \Big (\gamma(r_{s}X)  \big (  M_{t}(\varphi)^{2}-M_{s}(\varphi)^{2}  \big ) \Big)=\sum_{k\in\N}\E \Big ( \gamma(r_{s}X)\int_{s}^{t}\Big ( \iint_{\R^{2d}}f \sigma_{k} \cdot \nabla_{v}\varphi \dx \dv \Big )^{2}\dee r \Big), \label{eq:QV}\\
&\E \Big ( \gamma(r_{s}X) \big ( M_{t}(\varphi)\beta_{k}(t)-M_{s}(\varphi)\beta_{k}(s) \big ) \Big) = \E \Big ( \gamma(r_{s}X) \int_{s}^{t}\iint_{\R^{2d}}\sigma_{k} \cdot \nabla_{v}\varphi f \dx \dv \dee r \Big). \label{eq:CV}
\end{align}   
Begin by defining the filtration $(\mathcal{F}_{t}^{n})_{t=0}^{T}$ by the relation $\mathcal{F}_{t}^{n}=\sigma(r_{t}X_{n})$.  Next define the following $(\mathcal{F}_{t}^{n})_{t=0}^{T}$ continuous martingale $( M_{t}^{n}(\varphi))_{t=0}^{T}$ via
\begin{equation}\label{c1}
M_{t}^{n}(\varphi)=\iint_{\R^{2d}}f_{n}(t)\varphi \dx \dv - \iint_{\R^{2d}}f_{n}^{0}\varphi \dx \dv 
- \int_{0}^{t}\iint_{\R^{2d}}f_{n}(v \cdot \nabla_{x}\varphi+\mathcal{L}_{\sigma^{n}}\varphi)+g_{n}\varphi \, \dx\dv\ds.
\end{equation}
By the first assumption of the Proposition and Definition \ref{def:Martingale-Sol}, we find that
\begin{equation}\label{c2}
\E \Big ( \gamma(r_{t}X_{n})\big (  M_{t}^{n}(\varphi)-M_{s}^{n}(\varphi) \big ) \Big)
=0
\end{equation}
We now claim that for each $t \in [0,T]$, the random variables $\{ M_{t}^{n}(\varphi) \}_{n \in \N}$ converge to $M_{t}(\varphi)$ in $L^{2}(\Omega)$.  Indeed, this hinges on the following facts.  First, the sequences  $\{ \langle f_{n}(t),\varphi\rangle \}_{n \in \N}$ and $\{ \langle G_{n}(t),\varphi\rangle \}_{n \in \N}$ converge to  $\langle f(t),\varphi\rangle $ and  $\langle G(t),\varphi\rangle $ in $L^{2}(\Omega)$ as a result of the assumptions in Part 2 of the Proposition.  Also, we use that $\{f_{n}^{0}\}_{n \in \N}$ converges in $L^{1}_{x,v}$ to $f_{0}$ by the assumptions in Part 3 of the Proposition.   Second, the assumptions in Part 4 of the Proposition ensure that $\{\mathcal{L}_{\sigma^{n}}\varphi\}_{n \in \N}$ converges pointwise $\mathcal{L}_{\sigma}\varphi$, while remaining bounded in $L^{\infty}_{x,v}$.  Since $f_{n} \to f$ pointwise almost surely in $C_{t}([L^{1}]_{x,v})$, the product limit Lemma \ref{lem:product_lemma} ensures the almost sure converence
\begin{equation}
\lim_{n \to \infty}\int_{0}^{t}\iint_{\R^{2d}}f_{n}\mathcal{L}_{\sigma^{n}}\varphi\dx\dv\ds= \int_{0}^{t}\int_{\R^{2d}}f\mathcal{L}_{\sigma}\varphi\dx\dv\ds.
\end{equation}
Combinining this with the uniform bound on $\{f_{n}\}_{n \in \N}$ in $L^{2+}(\Omega;L^{\infty}(L^{1}_{x,v}))$, we may upgrade to convergence in $L^{2}(\Omega)$.  Moreover, it holds that for each $t \in [0,T]$, the random variables $\{ \gamma(r_{t}X_{n})\}_{n \in \N}$ converge to $\gamma(r_{t}X)$ with probability one, while remaining bounded $L^{\infty}(\Omega)$. To treat the sequence $\{G^{n}\}_{n \in \N}$, we use the fact that if a sequence of continuous functions converges pointwise to a continuous limit, then the convergence is also uniform.  With these remarks, we may pass $n \to \infty$ in \eqref{c2} and deduce \eqref{eq:Mart}.  Moreover, using again the product limit Lemma (in $\Omega$) we obtain
\begin{equation}
\lim_{n \to \infty}\E \Big (\gamma(r_{s}X_{n})  \big (  M_{t}^{n}(\varphi)^{2}-M_{s}^{n}(\varphi)^{2}  \big ) \Big)=\E \Big (\gamma(r_{s}X)  \big (  M_{t}(\varphi)^{2}-M_{s}(\varphi)^{2}  \big ) \Big)
\end{equation}
together with
\begin{equation}\label{c6}
\lim_{n \to \infty}\E \Big ( \gamma(r_{s}X_{n}) \big ( M_{t}^{n}(\varphi)\beta_{k}^{n}(t)-M_{s}^{n}(\varphi)\beta_{k}^{n}(s) \big ) \Big )=\E \Big ( \gamma(r_{s}X) \big ( M_{t}(\varphi)\beta_{k}(t)-M_{s}(\varphi)\beta_{k}(s)\big ) \Big ). 
\end{equation}
Next we observe that:
\begin{align}
\E \Big (\gamma(r_{s}X_{n})  \big (  M_{t}^{n}(\varphi)^{2}-M_{s}^{n}(\varphi)^{2}  \big ) \Big)&=\sum_{k\in\N}\E \Big ( \gamma(r_{s}X_{n})\int_{s}^{t}\Big ( \iint_{\R^{2d}}f_{n}\sigma_{k}^{n} \cdot \nabla_{v}\varphi \dx \dv \Big )^{2}\dee r \Big).\label{eq:QVKnown}
\end{align}
To pass the limit on the right hand side of \eqref{eq:QVKnown}, first fix a $k \in \N$.  We claim that  
\begin{equation}
\begin{aligned}
&\lim_{n \to \infty}\E \Big ( \gamma(r_{s}X_{n})\int_{s}^{t}\Big(\iint_{\R^{2d}}f_{n} \sigma_{k}^{n} \cdot \nabla_{v}\varphi \dx \dv\Big)^{2}\dee r \Big)\\
&\hspace{.5in}=\E \Big ( \gamma(r_{s}X)\int_{s}^{t}\Big(\iint_{\R^{2d}}f \sigma_{k} \cdot \nabla_{v}\varphi \dx \dv\Big)^{2}\dee r \Big).
\end{aligned}
\end{equation}
To show this, use again $f_{n} \to f$ almost surely in $C_{t}([L^{1}_{x,v}]_{w})$ together with the pointwise convergence of $\sigma_{k}^{n}$.  Combine this with the uniform bounds on $\{\sigma_{k}^{n}\}_{n \in \N}$ in $L^{\infty}_{x,v}$ and $\{f_{n}\}_{n \in \N}$ in $L^{2+}(\Omega; L^{\infty}_{t}(L^{1}_{x,v}))$.  Using also that the random variables $\{ \gamma(r_{t}X_{n})\}_{n \in \N}$ converge to $\gamma(r_{t}X)$ with probability one, while remaining bounded $L^{\infty}(\Omega)$, the claim follows from the product limit Lemma \ref{lem:product_lemma}.  

Moreover, we have the inequality
\begin{equation}
\E \Big ( \gamma(r_{s}X_{n})\int_{s}^{t}\Big(\iint_{\R^{2d}}f_{n} \sigma_{k}^{n} \cdot \nabla_{v}\varphi \dx \dv\Big)^{2}\dee r \Big) \leq \|\gamma\|_{C_{b}(E_{s};\R)}\|\nabla_{v}\varphi\|_{\LLs{\infty}}\E \|f_{n}\|_{\LLL{2}{1}}^{2} \|\sigma_{k}^{n}\|_{\LLs{\infty}}^{2}.
\end{equation} 
Recalling that $\{\sigma^{n}\}_{n \in \N}$ has property \eqref{c5} and using that $\{f_{n}\}_{n \in \N}$ is uniformly bounded in $L^{2}(\Omega ; \LLL{2}{1})$, we may split the series into finitely many terms plus a uniformly controlled remainder to obtain
\begin{equation}
\begin{split}
&\lim_{n \to \infty}\sum_{k\in\N}\E \Big ( \gamma(r_{s}X_{n})\int_{s}^{t}\Big(\iint_{\R^{2d}}f_{n} \sigma_{k}^{n} \cdot \nabla_{v}\varphi \dx \dv\Big)^2\dr\Big)\\
&\hspace{.5in}=\sum_{k\in\N}\E \Big ( \gamma(r_{s}X)\int_{s}^{t}\Big(\iint_{\R^{2d}}f \sigma_{k} \cdot \nabla_{v}\varphi \dx \dv\Big)^{2}\dee r \Big). 
\end{split}
\end{equation}
We may now pass $n \to \infty$ on both sides of \eqref{eq:QVKnown} to obtain \eqref{eq:QV}.  A similar (and in fact easier) argument yields
\begin{equation}
 \lim_{n \to \infty}\E \Big ( \gamma(r_{s}X_{n}) \int_{s}^{t}\iint_{\R^{2d}}\sigma_{k}^{n} \cdot \nabla_{v}\varphi f_{n} \dx \dv \dee r \Big)=\E \Big ( \gamma(r_{s}X) \int_{s}^{t}\iint_{\R^{2d}}\sigma_{k} \cdot \nabla_{v}\varphi f \dx \dv \dee r \Big),
\end{equation}
which can be combined with \eqref{c6} to obtain \eqref{eq:CV}.  This completes the proof.
\end{proof}

\subsection{Renormalization}
\label{sec:Renormalization}
Formally, given a regular solution $f$ to \eqref{eq:StochKin} and a smooth $\Gamma: \R \to \R$, Ito's formula implies that $\Renorm(f)$ satisfies
\begin{equation}\label{eq:general-renorm-kinetic}
\begin{aligned}
 &\partial_t \Renorm(f) + v\cdot \nabla_x\Renorm(f) + \Div_v (\Renorm(f)\sigma_{k}\strat \dot{\beta}_{k}) = \Renorm'(f)g,\\
 &\Renorm(f)|_{t=0} = \Renorm(f_0).
\end{aligned}
\end{equation}
However, if we only impose Hypothesis \ref{hyp:Noise-Coefficients} on the noise coefficients, it is not clear whether \eqref{eq:general-renorm-kinetic} can be justified when $f$ is only a weak martingale solution to \eqref{eq:StochKin}.  In this section, we show that if $f$ has increased local integrability in $x,v$ and $\sigma$ has sufficient Sobolev regularity, then \eqref{eq:general-renorm-kinetic} holds relative to a large class of renormalizations $\Renorm$.  Towards this end, we introduce the notion of renormalized martingale solution to \eqref{eq:StochKin}. 

\begin{definition}[Renormalized Martingale Solution] \label{def:Renorm-Martingale}
  Suppose that $(f_t)_{t=0}^T$ is a weak martingale solution to the stochastic kinetic equation driven by $g$ with initial data $f_0$ and with with respect to the stochastic basis $(\Omega, \mathcal{F}, \P, (\mathcal{F}_t)_{t=0}^T,\{\beta_k\}_{k\in\N})$. We say that $(f_t)_{t=0}^T$ is a renormalized weak martingale solution provided that for all renormalizations $\Renorm \in C^2(\R)$ with $\sup_{z\in\R}(|\Gamma^\prime(z)| + |\Gamma^{\prime\prime}(z)|) <\infty$ and $\Gamma(0) = 0$, the process $(\Gamma(f)_t)_{t=0}^T$ is weak martingale solution to the stochastic kinetic equation driven by $\Gamma^\prime(f)g$ with initial data $\Gamma(f_0)$, and with respect to the same stochastic basis $(\Omega, \mathcal{F}, \P, (\mathcal{F}_t)_{t=0}^T,\{\beta_k\}_{k\in\N})$.
\end{definition}

\begin{remark}
  It is important to note the assumptions on $\Gamma$ ensure that a renormalized martingale solution is consistent with the notion of weak martingale solution given in Definition \ref{def:Martingale-Sol}. Specifically, the assumptions $\sup_{z\in\R}|\Gamma^\prime(z)| <\infty$ and $\Gamma(0) = 0$ given in definition \ref{def:Renorm-Martingale} imply that $\Gamma(z) \leq C|z|$. This means that when $f\in L^2(\Omega; \LLL{\infty}{1})$, so is $\Gamma(f)$. Likewise we see that $\Gamma(f_0)\in \LLs{1}$ when $f_0$ is and $\Gamma^\prime(f)g\in L^1(\Omega; \LLLs{1})$ when $g$ is.
\end{remark}

\begin{prop}\label{prop:Weak_Is_Renormalized}
  Let $f$ be a weak martingale solution to the stochastic kinetic equation driven by $g$ with initial data $f_0$.  Assume that $f \in L^p(\Omega\times[0,T]\times\R^{2d})$ for some $p \in [2,\infty)$.  If the noise coefficients satisfy for each compact set $K\subseteq \R^{2d}$,
  \begin{equation}
\sum_{k\in\N} \|\sigma_k\|_{W^{1,\frac{2p}{p-2}}(K)}^2 <\infty,\quad \sum_{k\in\N}\|\sigma_k\cdot\nabla_v\sigma_k\|_{W^{1,\frac{p}{p-1}}(K)} < \infty,
\end{equation}
then $f$ is also a renormalized weak martingale solution.
\end{prop}
\begin{proof}Let $\Renorm$ satisfy the assumptions of definition \ref{def:Renorm-Martingale} , then our goal is to establish that $\Gamma(f)$ is a weak martingale solution driven by $\Gamma^\prime(f)g$ starting from $\Gamma(f_0)$. Towards this end, let $\eta$ be a standard symmetric mollifier with support contained in the unit ball on $\R^{d}_{x} \times \R^{d}_{v}$ with $\int_{\R^{2d}}\eta(x,v)\dx\dv = 1$. Set $\eta_\ep(x,v) = \ep^{-2d}\eta(\ep^{-1}x,\ep^{-1}v)$ and denote by $f_{t,\epsilon}=\eta_{\epsilon}*\,f_t = (f_t)_{\epsilon}$ the mollified process.

Let $\varphi \in C^{\infty}_{c}(\R^{2d})$. The main step in this proof will be to establish that for all $t \in [0,T]$, the following identity holds $\p$ almost surely:
\begin{equation}\label{eq:Regularized-Identity}
\begin{split}
&\iint_{\R^{2d}}\Renorm(f_{t,\epsilon})\varphi \dx \dv = \iint_{\R^{2d}}\Renorm(f_{0,\epsilon})\varphi \dx \dv\\
&\hspace{1in}+ \int_{0}^{t}\iint_{\R^{2d}}[\Renorm(f_{s,\epsilon})(v \cdot \nabla_{x} + \LStrat_\sigma)\varphi+\varphi \Renorm'(f_{s,\epsilon})g_{s,\ep}]\dx\dv\ds \\
&\hspace{1in}+\sum_{k=1}^{\infty}\int_{0}^{t}\iint_{\R^{2d}}\Renorm(f_{s,\epsilon})\sigma_{k}\cdot \nabla_{v}\varphi \dx\dv\dee\beta_{k}(s)
+ R^\varphi_\ep(t),
\end{split}
\end{equation}
for a process $(R^\varphi_\epsilon(t))_{t=0}^T$ such that for each $t\in [0,T]$,
\begin{equation}\label{eq:remainer-convergence}
  R^\varphi_\ep(t) \to 0 \quad \text{in probability as}\quad \ep \to 0.
\end{equation}
Assuming we can verify (\ref{eq:Regularized-Identity}) and (\ref{eq:remainer-convergence}), let us complete the proof. Using standard properties of mollifiers, for almost every $(\omega,t,x,v) \in \Omega\times[0,T]\times\R^{2d}$ one has
 \begin{equation}
   \begin{aligned}
   \Gamma(f_{\ep}) &\to \Gamma(f)\\
   \Gamma(f_{0,\ep}) &\to \Gamma(f_0),
   \end{aligned}
 \end{equation}
and furthermore, using the boundedness of $\Gamma(z)$ and $\Gamma^\prime(z)$, for each compact set $K \subseteq \R^{2d}$ one has
\begin{equation}
\begin{aligned}
  \Gamma(f_{\ep}) &\to \Gamma(f) \quad \text{in}\quad L^2(\Omega\times[0,T]\times K),\\
  \Gamma^\prime(f_{\ep})g_{\ep} &\to \Gamma^\prime(f)g \quad \text{in}\quad L^1(\Omega\times[0,T]\times K).
\end{aligned}
\end{equation}
Using the convergence properties above along with the It\^{o} isometry and the convergence of $R^\varphi_\ep$ to $0$, we may pass the $\ep \to 0$ limit in each term of (\ref{eq:Regularized-Identity}), where the convergence holds in $L^1(\Omega\times [0,T])$. We conclude that $\Gamma(f)$ solves
\begin{equation}
  \begin{split}
    &\iint_{\R^{2d}}\Renorm(f_{t})\varphi \dx \dv = \iint_{\R^{2d}}\Renorm(f_{0})\varphi \dx \dv\\
   &\hspace{.5in}+ \int_{0}^{t}\iint_{\R^{2d}}[\Renorm(f_{s})(v \cdot \nabla_{x} + \LStrat_\sigma)\varphi+\varphi \Renorm'(f_{s})g_{s}]\dx\dv\ds \\
&\hspace{.5in}+\sum_{k=1}^{\infty}\int_{0}^{t}\iint_{\R^{2d}}\Renorm(f_{s})\sigma_{k}\cdot \nabla_{v}\varphi \dx\dv\dee\beta_{k}(s),
  \end{split}
\end{equation}
thereby completing the proof.

It now remains to verify identity (\ref{eq:Regularized-Identity}) along with the vanishing of the remainder (\ref{eq:remainer-convergence}). We begin by considering the equation (\ref{eq:weak-martingale-stochastic-integral}). We fix $z = (x,v)\in \R^{2d}$ and choose $\varphi(w) = \eta_\ep(z - w)$. This is equivalent to mollifying both sides of equation, giving
\begin{equation}\label{eq:Mollified-eq}
\begin{aligned}
    &f_{t,\ep}(z) = f_{0,\ep}(z) + \int_0^t[(-v\cdot\nabla_x f_s)_\ep(z) + (\LStrat_\sigma f)_\ep(z) + g_{s,\ep}(z)]\ds\\
&\hspace{1in} - \sum_{k\in\N}\int_0^t (\sigma_k\cdot\nabla_v f_s)_{\ep}(z)\dee \beta_{k}(s).
\end{aligned}
\end{equation}
For each $z\in \R^{2d}$, we may renormalize by $\Gamma$ by applying It\^{o}'s formula,
\begin{equation}
  \begin{split}
    &\Gamma(f_{t,\ep}(z)) = \Gamma(f_{0,\ep}(z)) + \int_0^t\Gamma^\prime(f_{s,\ep}(z))[(-v\cdot\nabla_x f_s)_\ep(z) + (\LStrat_\sigma f)_\ep(z) + g_{s,\ep}(z)]\ds\\
&\hspace{.5in} + \frac{1}{2}\sum_{k\in\N} \int_0^t \Gamma^{\prime\prime}(f_{s,\ep}(z))(\sigma_k\cdot\nabla_v f_s)_\ep^2(z)\ds\\
&\hspace{.5in} - \sum_{k\in\N}\int_0^t\Gamma^\prime(f_{s,\ep}(z)) (\sigma_k\cdot\nabla_v f_s)_{\ep}(z)\dee \beta_{k}(s).
  \end{split}
\end{equation}
Naturally we can force the form of (\ref{eq:Regularized-Identity}) into view by the use of the commutators,
\begin{equation}
  \begin{aligned}
    &[\eta_\ep,v\cdot\nabla_x](f) = (v\cdot\nabla_xf)_\ep - v\cdot\nabla_xf_\ep\\
    &[\eta_\ep,\LStrat_\sigma ](f) = (\LStrat_\sigma f)_{\ep} - \LStrat_\sigma f_\ep\\
    &[\eta_\ep,\sigma_k\cdot\nabla_v](f) = (\sigma_k\cdot\nabla_v f)_\ep  - \sigma_k\cdot\nabla_v f_\ep.
  \end{aligned}
\end{equation}
 Specifically, using the fact that $\LStrat_\sigma  \Gamma(f) = \Gamma^\prime(f)\LStrat_\sigma f + \tfrac{1}{2}(\sigma_k\cdot\nabla_v f)^2\Gamma^{\prime\prime}(f)$, we find
\begin{equation}\label{eq:commutator-remainder-form}
  \begin{split}
  &\Gamma(f_{t,\ep}) = \Gamma(f_{0,\ep}) + \int_0^t[(-v\cdot\nabla + \LStrat_\sigma)\Gamma(f_{s,\ep}) + \Gamma^\prime(f_{s,\ep})g_{s,\ep}]\ds\\
&\hspace{1in} - \sum_{k\in\N}\int_{0}^t\sigma_k\cdot\nabla_v \Gamma(f_{s,\ep})\dee\beta_{k}(s) + R_{t,\ep}
  \end{split}
\end{equation}
where $R_{t,\ep}$ is a given by
\begin{equation}
\begin{split}
  &R_{t,\ep} = \int_0^t\Gamma^\prime(f_{s,\ep})\big(-[\eta_\ep,v\cdot\nabla_x](f_s) +[\eta_\ep,\LStrat_\sigma ](f_s)\big)\ds\\
&\hspace{1in} + \sum_{k\in\N}\frac{1}{2}\int_0^t\Gamma^{\prime\prime}(f_{s,\ep})[(\sigma_k\cdot\nabla_v f_s)_\ep^2 -  (\sigma_k\cdot\nabla_vf_{s,\ep})^2]\ds\\
&\hspace{1in}-  \sum_{k\in\N}\int_0^t\Gamma^\prime(f_{s,\ep}) [\eta_\ep,\sigma_k\cdot\nabla_v](f_s)\dee\beta_{k}(s).
\end{split}
\end{equation}
Integrating both sides of (\ref{eq:commutator-remainder-form}) against $\varphi$, we obtain (\ref{eq:Regularized-Identity}). 

It remains to show that for each $t\in[0,T]$,
\begin{equation}
R^\varphi_\ep(t) := \iint_{\R^{2d}} \varphi R_{t,\ep}\dx\dv  \to 0\quad \text{in probability as} \quad \ep \to 0.
\end{equation}
This will be proved with the aid of standard commutator lemmas taken from \cite{diperna1989ordinary}. Specifically, we use that $f\in L^p(\Omega\times[0,T]\times \R^{2d}) \cap L^{p}(\Omega \times [0,T]; L^{1}_{x,v})$ and for each $k \in \N$, we have $\sigma_k \in [W^{1,\frac{2p}{p-2}}_{x,v}]_{\loc}$.  It follows that for almost every $(\omega, t) \in \Omega\times [0,T]$ we have
\begin{align}
  [\eta_\ep, v\cdot\nabla_x](f_t) \to 0 \quad&\text{in}\quad [\LLs{2}]_{\loc},\label{eq:v-commutator-conv}\\
  [\eta_\ep, \sigma_k\cdot\nabla_v](f_t) \to 0 \quad &\text{in}\quad [\LLs{2}]_{\loc},\label{eq:sigma-commutator-conv}
\end{align}
as well as the bound,
\begin{equation}\label{eq:sigma-commutator-bound}
  \|[\eta_\ep, \sigma_k\cdot\nabla_v](f_t)\|_{L^2(K)} \leq \|\sigma_k\|_{W^{1,\frac{2p}{p-2}}(K)}\,\|f_t\|_{\LLs{p}}.
\end{equation}
In order to use the commutator results (\ref{eq:v-commutator-conv}) and (\ref{eq:sigma-commutator-conv}) to our advantage, we will need to manipulate $R_{t,\ep}$. First we write the commutator $[\eta_\ep, \LStrat_\sigma](f)$ in terms of $[\eta_\ep,\sigma_k\cdot\nabla_v]$ as follows:
\begin{equation}
\begin{split}
  [ \eta_\ep , \LStrat_\sigma ](f) &= \frac{1}{2}\sum_{k\in\N}\Big((\sigma_k\cdot\nabla_v(\sigma_k\cdot\nabla_vf))_\ep - \sigma_k\cdot\nabla_v(\sigma_k\cdot\nabla_vf_\ep)\Big)\\
 &=\frac{1}{2}\sum_{k\in\N} \Big([\eta_{\ep},\sigma_k\cdot\nabla_v](\sigma_k\cdot\nabla_v f) + \sigma_k\cdot\nabla_v[\eta_\ep,\sigma_k\cdot\nabla_v](f)\Big).
\end{split}
\end{equation}
The second observation is the following equalities
\begin{equation}
  \begin{split}
  &\frac{1}{2}\Gamma^{\prime\prime}(f_\ep)[(\sigma_k\cdot\nabla_v f)_\ep^2 -  (\sigma_k\cdot\nabla_vf_{\ep})^2]\\
  &\hspace{.5in}= \frac{1}{2}\Gamma^{\prime\prime}(f_\ep)[\eta_\ep,\sigma_k\cdot\nabla_v](f)\big((\sigma_k\cdot\nabla_vf)_\ep + \sigma_k\cdot\nabla_v f_\ep\big)\\
 &\hspace{.5in}= \frac{1}{2}\Gamma^{\prime\prime}(f_\ep)\big([\eta_\ep, \sigma_k\cdot\nabla_v](f)\big)^2 + \Gamma^{\prime\prime}(f_{\ep})[\eta_{\ep}, \sigma_k \cdot \nabla_v](f)\sigma_k\cdot\nabla_v f_\ep\\
&\hspace{.5in}= \frac{1}{2}\Gamma^{\prime\prime}(f_\ep)\big([\eta_\ep, \sigma_k\cdot\nabla_v](f)\big)^2 - \Gamma^\prime(f_\ep) \sigma_k\cdot\nabla [\eta_\ep,\sigma_k\cdot\nabla_v](f)\\
&\hspace{1in} + \sigma_k\cdot\nabla_v\left(\Gamma^\prime(f_\ep)[\eta_\ep,\sigma_k\cdot\nabla_v](f)\right).
  \end{split}
\end{equation}
Adding the two identities above and introducing the double commutator $\big[[\eta_\ep, \sigma_k\cdot\nabla_v], \sigma_k\cdot\nabla_v\big]$ defined by
\begin{equation}
  \big[[\eta_\ep, \sigma_k\cdot\nabla_v], \sigma_k\cdot\nabla_v\big](f) = [\eta_\ep, \sigma_k\cdot\nabla_v](\sigma_k\cdot\nabla_vf) - \sigma_k\cdot\nabla_v[\eta_\ep,\sigma_k\cdot\nabla_v](f),
\end{equation}
we conclude that
\begin{equation}
  \begin{split}
   & \Gamma^\prime(f_\ep)[\eta_\ep,\LStrat_\sigma](f)  + \sum_{k\in\N}\frac{1}{2}\Gamma^{\prime\prime}(f_\ep)[(\sigma_k\cdot\nabla_v f)_\ep^2 -  (\sigma_k\cdot\nabla_vf_{\ep})^2]\\
&\hspace{.5in}= \sum_{k\in\N}\Big( \sigma_k\cdot\nabla_v\left(\Gamma^\prime(f_\ep)[\eta_\ep,\sigma_k\cdot\nabla_v](f)\right) + \frac{1}{2} \Gamma^{\prime\prime}(f_\ep)\big([\eta_\ep, \sigma_k\cdot\nabla_v](f)\big)^2\\
&\hspace{1in}+ \frac{1}{2}\Gamma^\prime(f_\ep)\big[[\eta_\ep, \sigma_k\cdot\nabla_v], \sigma_k\cdot\nabla_v\big](f)\Big).
  \end{split}
\end{equation}
The process $R_{t,\ep}$ is therefore given by
\begin{equation}\label{eq:Remainder-term-final}
  \begin{split}
    &R_{t,\ep} = -\int_0^t\Gamma^\prime(f_{s,\ep})[\eta_\ep,v\cdot\nabla_x](f_s)\ds + \sum_{k\in\N}\int_0^t\sigma_k\cdot\nabla_v\left(\Gamma^\prime(f_\ep)[\eta_\ep,\sigma_k\cdot\nabla_v](f_s)\right)\ds\\
&\hspace{.5in} + \frac{1}{2}\sum_{k\in\N}\int_0^t \Gamma^{\prime\prime}(f_{s,\ep})\big([\eta_\ep, \sigma_k\cdot\nabla_v](f_s)\big)^2\ds\\
&\hspace{.5in}+ \frac{1}{2}\sum_{k\in\N}\int_0^t\Gamma^\prime(f_{s,\ep})\big[[\eta_\ep, \sigma_k\cdot\nabla_v], \sigma_k\cdot\nabla_v\big](f_s)\ds\\
&\hspace{.5in} + \sum_{k\in\N}\int_0^t\Gamma^\prime(f_{s,\ep}) [\eta_\ep,\sigma_k\cdot\nabla_v](f_s)\dee\beta_{k}(s).
  \end{split}
\end{equation}
Integrating $R_{t,\ep}$ against $\varphi$ to obtain $R^\varphi_\ep(t)$, it is now possible to use the convergences (\ref{eq:v-commutator-conv}), (\ref{eq:sigma-commutator-conv}), the uniform bound (\ref{eq:sigma-commutator-bound}), and our assumptions on the noise coefficients to show that each term in $R^\varphi_\ep(t)$ involving the single commutators, $[\eta_\ep,v\cdot\nabla_x](f)$ and $[\eta_{\ep}, \sigma_k\cdot\nabla_v](f)$, converges to $0$ in probability for each $t\in[0,T]$. 

It remains to estimate the double commutator term
\begin{equation}\label{eq:double-commutator-remainder}
  I_{t,\ep} = \frac{1}{2}\sum_{k\in\N}\int_{0}^t\iint_{\R^{2d}}\varphi\Gamma^{\prime}(f_{s,\ep})\big[[\eta_\ep, \sigma_k\cdot\nabla_v], \sigma_k\cdot\nabla_v\big](f_s)\dx\dv\ds.
\end{equation}
We will prove that for each $t\in[0,T]$, $I_{t,\ep} \to 0$ in probability.

In what follows, to simplify notation, we will denote both $z = (x,v)$ and $w =(y,u)$ the phase space (position-velocity) coordinates in $\R^{2d}$ wherever possible, and define the translation operator
\begin{equation}
\delta_{w}f(z):=f(z+w)-f(z).
\end{equation}
We will need to evaluate the double-commutator explicitly.  This will be done piece by piece. For the first piece, since $\Div_v\sigma_k =0 $, integrating by parts gives
\begin{equation}
    \begin{split}
 &\left[\eta_{\ep}\,,\,\sigma_k\cdot\nabla_v\right] (\sigma_k\cdot\nabla_v f_t)(z) = \int_{\R^{2d}} \nabla^2_v\eta_{\ep}(z-w):\sigma_k(w)\tensor[\sigma_k(w)-\sigma_k(z)]\,f_t(w)\,\dw\\
&\hspace{1in} - \int_{\R^{2d}} \nabla_v\eta_{\ep}(z-w)\cdot (\sigma_k(w)\cdot\nabla_v\sigma_k(w))\,f_t(w)\,\dw,
    \end{split}
\end{equation}
and similarly, for the second piece, we have
\begin{equation}
  \begin{split}
    &\sigma_k\cdot\nabla_v\left[\eta_\ep\,,\,\sigma_k\cdot\nabla_v\right](f_t)(z) = \int_{\R^{2d}}\nabla^2_v\eta_\epsilon(z-w):[\sigma_k(w) - \sigma_k(z)]\tensor\sigma_k(z)\,f_t(w)\,\dw\\
&\hspace{1in} - \int_{\R^{2d}}\nabla_v\eta_\epsilon(z-w)\cdot(\sigma_k(z)\cdot\nabla_v\sigma_k(z))\,f_t(w)\,\dw.
  \end{split}
\end{equation}
Note that the operation $f \to \big[[\eta_\ep, \sigma_k\cdot\nabla_v], \sigma_k\cdot\nabla_v\big](f)$ vanishes on constant functions.  Hence, in both identities above we may freely replace $f(w)$ by $f(w)-f(z)$. Therefore, using the symmetry of $\nabla_v^2\eta_\ep$, and changing variables $w\to  w+z$, we conclude that the double commutator can be written in the following form
\begin{equation}
  \begin{split}
  &\big[[\eta_\ep, \sigma_k\cdot\nabla_v], \sigma_k\cdot\nabla_v\big](f_t)(z) =  \int_{\R^{2d}}\nabla^2_v\eta_\epsilon(w):(\delta_w\sigma_k(z)\tensor\delta_w\sigma_k(z))\delta_wf_t(z)\,\dw\\
&\hspace{1in} + \int_{\R^{2d}}\nabla_v\eta_\epsilon(w)\cdot \delta_w(\sigma_k\cdot\nabla_v\sigma_k)(z)\delta_wf_t(z)\,\dw.
\end{split}
\end{equation}

Next we use the fact that for any $g \in W^{1,r}_{x,v}$, the following inequality holds pointwise in $w \in \R^{2d}$
\begin{equation}\label{eq:EstOnTranslations}
|\delta_{w}g|_{L^{r}_{x,v}}\leq |w||\nabla g|_{L^{r}_{x,v}}.
\end{equation}
Using Holder's inequality, the estimate \eqref{eq:EstOnTranslations}, and the fact that $|\nabla^2_v\eta_\ep(w)|\,|w|^2$ and $|\nabla_v \eta_\ep(w)|\,|w|$ are uniformly bounded in $L^1_w$, we may estimate $I_{t,\ep}$ for each $t\in[0,T]$ and $\omega \in \Omega$,
\begin{equation}
    |I_{t,\ep}| \leq C_\varphi\Big(\|\sigma\|_{\ell^2(\N; W^{1,\frac{2p}{p-2}}(K_\varphi))}^2 + \|\sigma\cdot\nabla\sigma\|_{\ell^1(\N; W^{1,\frac{p}{p-1}}(K_\varphi))}\Big)\|\Gamma^{\prime}(f_\ep)\|_{\LLLs{\infty}}\sup_{|w| <\ep} \|\delta_w f\|_{\LLLs{p}},
  \end{equation}
  where  $K_\varphi$ is the compact support of $\varphi$. Since $f\in L^p([0,T]\times\R^{2d})$ with probability one, 
\begin{equation}
\sup_{|w| <\ep} \|\delta_w f\|_{\LLLs{p}} \to 0, \quad \P \text{ almost surely}.  
\end{equation}
 The proof of the Proposition is now complete since this implies $I_{t,\ep} \to 0$ in probability for each $t\in[0,T]$.

\end{proof}

This section is now completed by checking that renormalized, weak martingale solutions to \eqref{eq:StochKin} with additional integrability have strongly continuous sample paths. The following lemma will be crucial for ultimately deducing strong continuity properties of the solution to the stochastic Boltzmann equation.
\begin{lem}[Strong Continuity]\label{lem:strong-cont}
Let $f$ be a renormalized weak martingale solution to the stochastic kinetic equation driven by $g$ with initial data $f_0$. If $f$ belongs to $\LLL{\infty}{p}$ with probability one for some $p \in (1,\infty)$, then $f \in \CLL{q}$ with probability one for any $q\in (1,p)$.
\end{lem}
\begin{proof}
We begin by remarking that $f \in \CLLw{p}$ with probability one. Indeed, let $\varphi \in C^{\infty}_{c}(\R^{2d})$.  It follows directly from inspection of the weak form and elementary properties of stochastic integrals that the process $t \to \langle f_{t},\varphi \rangle$ has continuous sample paths.  Moreover, since $f$ belongs to $\LLL{\infty}{p}$ with probability one, it follows that $f \in \CLLw{p}$ with probability one.  

The next step is to upgrade to continuity with values in $\LLs{q}$ with the strong topology.  Towards this end, let $\psi: \R \to \R$ be defined by $\psi(x)=|x|^{q}$.  We may choose a sequence of smooth, truncations of $\psi$, denoted $\{\psi_{k}\}_{k \in \N}$ that satisfy the conditions on the renormalizations in Definition \ref{def:Renorm-Martingale} such that $\psi_{k}$ converge pointwise in $\R$ to $\psi$ as $k \to \infty$. Moreover, these truncations can be chosen so that when $|x| < k $, $\psi_k(x) = \psi(x)$, and when $|x| > k$, $0\leq \psi_k(x) \leq \psi(x)$.  Applying Proposition \ref{prop:Weak_Is_Renormalized}, and using the fact that, with probability one, $\psi_k(f)$ is in $\LLL{\infty}{1}$ and $\psi^\prime_k(f)g$ is in $\LLLs{1}$, we find that, for all $t\in [0,T]$, we have the $\p$- a.s. identity,
\begin{equation}
\| \psi_{k}(f_{t}) \|_{\LLs{1}}=\| \psi_{k}(f_{0}) \|_{\LLs{1}}+\int_{0}^{t}\iint_{\R^{2d}}\psi_{k}'(f_{s})g_{s}\dx\dv \dee s.
\end{equation}
In particular, this implies that $t \mapsto \|\psi_k(f_t)\|_{\LLs{1}}$  has continuous sample paths with probability one. Since weak martingale solutions are in $\CLLw{1}$ with probability one, then by interpolation, $f$ is in $\CLLw{q}$ with probability one, and therefore for each $t\in[0,T]$, $\|\psi(f_t)\|_{\LLs{1}}$ is defined $\P$- a.s.

Next, we claim that, $\p$ almost surely,
\begin{equation}
\|\psi_k(f)\|_{\LLs{1}} \to \|\psi(f)\|_{\LLs{1}}\quad \text{in}\quad L^\infty([0,T]),
\end{equation}
whereby we may conclude that $t \to \|f_{t}\|_{\LLs{q}}$ has continuous sample paths with probability one. Indeed, we find
\begin{equation}
  \begin{aligned}
    \sup_{t\in[0,T]}\big|\|\psi(f)\|_{\LLs{1}} - \|\psi_k(f)\|_{\LLs{1}}\big| &\leq \|\psi(f) - \psi_k(f)\|_{\LLL{\infty}{1}} \leq \|\psi(f)\1_{f\geq k}\|_{\LLL{\infty}{1}}\\
 &\leq \|f\|_{\LLL{\infty}{p}}^q\Big(\sup_{t\in[0,T]} |\{ |f_t| \geq k\}|\Big)^{1-p/q}\\
 &\leq \frac{1}{k^{p-q}}\|f\|_{\LLL{\infty}{p}}^{p} \to 0 \quad \text{as} \quad k\to \infty.
  \end{aligned}
\end{equation}
Since $\LLs{q}$ is a uniformly convex space for $q>1$, the fact that $f$ is in $\CLLw{q}$ with probability one, combined with the fact $t\mapsto \|f_t\|_{\LLs{q}}$ has $\P$-a.s. continuous sample paths implies that $f\in \CLL{q}$ with probability one.
\end{proof}

\section{Stochastic Velocity Averaging}\label{sec:Stoch-Velocity-Avg}
In Section \ref{sec:Approximating-Scheme}, we will construct a sequence $\{f_{n}\}_{n \in \N}$ of approximations to the Boltzmann equation \eqref{eq:general-stochastic-Boltz} with stochastic transport.  These will satisfy the formal a priori bounds \eqref{eq:aPriori_Bound}, uniformly in $n \in \N$ enabling us to extract a weak limit $f$, which will be a candidate renormalized solution to \eqref{eq:general-stochastic-Boltz}.  However, we need a form of strong compactness to handle the stability of the non-linear collision operator.

In this section we investigate some subtle regularizing effects for stochastic kinetic equations, inspired by the classical work of Golse/ Lions/ Perthame/ Sentis \cite{golse1988regularity}.  These will be applied in Section 6 to obtain a form of strong compactness of $\{f_n\}_{n \in \N}$.  In fact, we allow for a nontrivial probability of oscillations in the velocity variable, so the strong compactness is only in space and time. The main idea, following \cite{golse1988regularity}, is that while $\{f_n\}_{n\in\N}$ may not be any more regular than it's initial data, {\it velocity averages} of $f_n$ may be more regular. Specifically a velocity average of $f$ is given by
\begin{equation}
  \langle f\,\phi \rangle := \int_{\Rd} f\phi\,\dv,
\end{equation}
where $\phi(v)$ is a test function of velocity only. Note that $\langle f\,\phi\rangle$ may depende on $(\omega,t,x)$ if $f$ does.

It turns out that the criteria for renormalization obtained in Section 3 plays an important role in the proof of our stochastic velocity averaging results.  As a consequence, we are only able to establish our compactness criterion for sequences of well-prepared approximations. Indeed for each $n\in \N$, suppose that $f_n$ is a weak martingale solution to the stochastic kinetic equation driven by $g_n$ and starting from $f_n^0$, relative to the noise coefficients $\sigma^n = \{\sigma_k^n\}_{k\in\N}$ and the stochastic basis $(\Omega_n,\mathcal{F}_n,(\mathcal{F}^n_t)_{t\in[0,T]},\{\beta_k^n\}_{k\in\N},\P_n)$. Then we make the following assumptions on $f_n$, $f^n_0$, $g_n$, and $\sigma_n$,
 \begin{hyp}\label{hyp:Stoch-Vel-Average-Fixed-n_hyp}
   \hspace{1in}
\begin{enumerate}
\item Both $f_n$ and $g_n$ belong to $L^{\infty-}(\Omega; \LLLs{1}\cap\LLLs{\infty})$.
\item $f^0_n$ is in $\LLs{1}\cap\LLs{\infty}$, and $\{f_n^0\}_{n\in \N}$ is uniformly integrable $\LLs{1}$
\item $\sigma^{n}$ satisfies Hypothesis \ref{hyp:Noise-Coefficients-Derivatives}, and $\{\sigma^n\}_{n\in\N}$ satisfies Hypothesis \ref{hyp:Noise-Coefficients} uniformly.
\end{enumerate}
\end{hyp}

Our main stochastic velocity averaging result can now be stated as follows:
\begin{lem} \label{lem:L1-Velocity-Averaging}
  Let $\{f_{n}\}_{n \in \N}$ be a sequence of weak martingale solutions to a stochastic kinetic equation satisfying Hypothesis \ref{hyp:Stoch-Vel-Average-Fixed-n_hyp} and suppose that $\{g_n\}_{n\in\N}$ is uniformly bounded in $L^{1}(\Omega \times [0,T] \times \R^{2d})$ and induces a tight family of laws on $[\LLLs{1}]_{\mathrm{w},\loc}$.
  \begin{enumerate}
  \item Then for each $\varphi \in C^\infty_c(\Rd)$,  $\{\langle f_{n}\,\varphi\rangle\}_{n \in \N}$ induces a tight family of laws on $[L^1_{t,x}]_{\loc}$.
  \item If in addition, for each $\eta >0$ the velocity averages $\{\langle f_{n}\,\varphi\rangle\}_{n \in \N}$ satisfy
  \begin{equation}
    \lim_{R\to\infty}\sup_n\P\left(\|\langle f_n\,\varphi\rangle \1_{|x|>R}\|_{L^1_{t,x}} >\eta \right) = 0,
    \end{equation}
    then for each $\varphi \in C^\infty_c(\Rd)$,  $\{\langle f_{n}\,\varphi\rangle\}_{n \in \N}$ induces a tight family of laws on $L^1_{t,x}$.
  \end{enumerate}
\end{lem}

Before we prove this, it will be useful to consider the $L^2$ case.
\subsection{ $L^{2}$ Velocity Averaging}
As is typical with velocity averaging lemmas in $L^1$ (see \cite{golse1988regularity}), we will find it useful first to prove an $L^2$ result.  Roughly speaking, the $L^{1}$ case is then reduced to showing that the part of the solution sequence violating the hypotheses of the $L^{2}$ lemma has a high probability of being small in $L^{1}$.
\begin{lem}\label{lem:L2-Velocity-Averaging}
Let $\{f_{n}\}_{n \in \N}$ be a sequence of martingale solutions to the stochastic kinetic equation satisfying Hypothesis \ref{hyp:Stoch-Vel-Average-Fixed-n_hyp}.  If $\{f_n^0\}_{n\in \N}$ is bounded in $\LLs{2}$ and $\{g_n\}_{n\in\N}$, $\{f_n\}_{n\in\N}$ are bounded in $L^{p}(\Omega \times [0,T] \times \R^{2d})$ for each $p\geq 1$, then for each $\phi \in C^\infty_c(\Rd_v)$, the velocity averages $\{\langle f_{n}\,\phi\rangle\}_{n \in \N}$ induce tight laws on $[L^2_{t,x}]_\loc$.
\end{lem}
In the $L^{2}$ setting, Fourier methods yielding explicit regularity estimates on the velocity averages can be obtained. Using an extension of the method outlined in \cite{Bouchut1999-nh}, the following spatial regularity estimate on $\langle f\, \phi \rangle$ can be established.
\begin{lem}\label{lem:L2_velocity_averaging_x}
Let $f$ be a weak martingale solution to the stochastic kinetic equation driven by $g$, with initial data $f_0$ relative to noise coefficients $\sigma$ satisfying Hypothesis \ref{hyp:Noise-Coefficients}. If $f,g \in L^2(\Omega\times[0,T]\times\R^{2d})$ and $f_0\in\LLs{2}$, then for any $\phi \in C^\infty_c(\Rd_v)$,
\begin{equation}\label{eq:L2_Regularity_Est}
 \E\|\langle f\, \phi \rangle\|_{L^2_t(H^{1/6}_x)}^2\leq C_{\phi,\sigma}\big( \|f_0\|^2_{L^2_{x,v}} + \E\|f\|_{L^2_{t,x,v}}^2 + \E\|g\|_{L^2_{t,x,v}}^2\big).
\end{equation}
\end{lem}
The proof is technical and left to Appendix \ref{sec:Vel-Aver-Appendix}. We are now equipped to prove Lemma \ref{lem:L2-Velocity-Averaging}:
\begin{proof}[Proof of Lemma \ref{lem:L2-Velocity-Averaging}] Let $\phi \in C_{c}(\R^{d})$ be arbitrary. We proceed by explicitly constructing sets $(K_\ell)_{\ell > 0}$ which are compact in $[\LLt{2}]_{\loc}$ such that
\begin{equation}
\lim_{\ell\to \infty} \sup_n \P\{ \langle f_n\, \phi \rangle \notin K_\ell\} = 0.
\end{equation}
  Let $\{\varphi_j\}_{j=1}^\infty$ be a dense subset of $L^2_x$ and $\{N_{j}\}_{j\in\N}$ be a positive, real-valued sequence to be selected later. Define the sets
  \begin{equation}
    \begin{aligned}
    E_\ell &= \big\{ \rho \in\LLt{2}\,:\, \|\rho\|_{L^2_t(H^{1/6}_x)} \leq \ell \big\},\\
    F_\ell &= \bigcap_{j=1}^\infty\big\{ \rho \in\LLt{2}\,:\, \|\langle \rho, \varphi_j\rangle \|_{W^{\gamma,p}_t} \leq (\ell\,N_{j})^{\frac{1}{p}}\big\},
    \end{aligned}
  \end{equation}
where $p>4$ and $\gamma = \frac{1}{4} - \frac{1}{p}$. Let $K_{\ell}=E_\ell\cap F_\ell$ and observe this is a compact set in $[\LLt{2}]_{\mathrm{loc}}$.

Applying the Chebyshev inequality followed by Lemma \ref{lem:L2_velocity_averaging_x},
\begin{equation}
  \P\big\{ \langle f_n, \phi \rangle \notin E_\ell \big\} \leq \frac{1}{\ell}\E\|\langle f_n\, \phi \rangle \|_{L^2_t(H^{1/6}_x)} \leq \frac{C_{\phi}}{\ell},
\end{equation}
where $C_{\phi}$ depends on the uniform bounds for $\{f_n\}_{n=1}^\infty$, $\{g_n\}_{n=1}^\infty$, $\{f_{n}^{0}\}_{n \in \N}$, and $\{\sigma^{n}\}_{n \in \N}$.  Similarly, for each $j \in \N$ we may appeal to Lemma \ref{lem:time-regularity} to find a constant $C_{\varphi_{j}}$ (depending on the same uniform bounds) such that
\begin{equation}
    \P\{\langle f_n, \phi \rangle \notin F_\ell\} \leq \sum_{j=1}^\infty \P\Big \{ \big \| \big \langle \langle f_n\, \phi \rangle , \varphi_j \big \rangle \big \|_{W^{\gamma,p}_t} < \ell N_{j} \Big\} \leq \sum_{j=1}^\infty\frac{C_{\varphi_j}}{\ell N_{j}}.
\end{equation}
Choosing $N_{j} = 2^{j}C_{\varphi_j}$, we conclude that
\begin{equation}
  \sup_n \P\{ \langle f_n\, \phi \rangle \notin K_{\ell}\} \leq \frac{1}{\ell}\sum_{j=1}^\infty 2^{-j} =\frac{1}{\ell} .
\end{equation}
Taking $\ell \to \infty$ gives the result.
\end{proof}


\subsection{Proof of Main lemma}\label{subsec:Stoch-Vel-Avg-L1-Proof}
In this section, we give the proof of the main result of the section, Lemma \ref{lem:L1-Velocity-Averaging}.
\begin{proof}[Proof of Lemma \ref{lem:L1-Velocity-Averaging}]
Let $\{(\Omega_n,\mathcal{F}_n,(\mathcal{F}^n_t)_{t\in[0,T]},\{\beta_k^n\}_{k\in\N},\P_n)\}_{n\in\N}$ be the sequence of stochastic bases corresponding to $\{f_n\}_{n\in\N}$. Fix $\epsilon >0$ and for each $n\in\N$, we begin by decomposing $f_n$ as
\begin{equation}
f_n = f_{n}^{\leq L} + f_{n}^{> L},
\end{equation}
such that $f_{n}^{\leq L}$ solves
\begin{equation}
  \partial_t f_n^{\leq L} + v\cdot\nabla_x f_n^{\leq } + \sigma_k^n\cdot\nabla_v f_n^{\leq L}\strat \dot{\beta}_k^n = g_n\1_{|g_n|\leq L}, \quad f_n^{\leq L}|_{t=0} = f_n^0\1_{|f^0_n|\leq L}.
\end{equation}
and $f_{n}^{> L}$ solves
\begin{equation}
  \partial_t f_n^{> L} + v\cdot\nabla_x f_n^{> L} + \sigma_k^n\cdot\nabla_v f_n^{> L}\strat \dot{\beta}_k^n = g_n\1_{|g_n|> L}, \quad f_n^{> L}|_{t=0} = f_n^0\1_{|f^0_n|>L}
\end{equation}
on the filtered probability space $(\Omega_n,\mathcal{F}_n,(\mathcal{F}^n_t)_{t\in[0,T]},\P_n)$. Since $g_n\1_{|g_n|\leq L}$ belongs to the space $L^{\infty-}(\Omega; \LLLs{1}\cap\LLLs{\infty})$ by Hypothesis \ref{hyp:Stoch-Vel-Average-Fixed-n_hyp}, we can build the above decomposition in the following way.  First apply the existence result, Theorem \ref{thm:existence-Stoch-Kin} to obtain $f_{n}^{\leq L}$ as a solution to the equation above. Then, by linearity, the process $f_{n}^{> L} := f_n - f_{n}^{\leq L}$ must solve it's corresponding equation above. Moreover, since $f_n$ and $f_{n}^{\leq L}$ are both in $L^{\infty-}(\Omega\times[0,T]\times\R^{2d})$, so is $f_{n}^{>L}$.  In view of our assumptions on the noise coefficients made in Hypothesis \ref{hyp:Stoch-Vel-Average-Fixed-n_hyp} we may apply Proposition \ref{prop:Weak_Is_Renormalized} to deduce that $f_{n}^{>L}$ is in fact a renormalized solution.

The strategy of the proof will be to show that the process $\langle f_{n}^{\leq L}\,\phi\rangle$ is tight in $n$ using the $L^2$ velocity averaging Lemma \ref{lem:L2-Velocity-Averaging} and that the remaining processes, $f_{n}^{\geq L}$, can be made uniformly small in $n$ by taking $L$ sufficiently large and therefore appealing to Lemma \ref{lem:tightness-decomp}.

First we apply our $L^{2}$ velocity averaging lemma to $\{f_{n}^{\leq L}\}_{n \in \N}$. Note that $\{f^0_n\1_{|f^0_n| \leq L}\}_{n\in\N}$ is bounded in $\LLs{2}$ (by interpolation) and $\{g_n\1_{|g_n|\leq L}\}_{n\in\N}$ is bounded in $L^{\infty-}(\Omega \times [0,T] \times \R^{2d})$. Therefore, by the estimate given in Theorem \ref{thm:existence-Stoch-Kin}, $\{f_n^{\leq L}\}_{n\in\N}$ is also bounded in $L^{\infty-}(\Omega \times [0,T] \times \R^{2d})$. Hence we have enough to apply Lemma \ref{lem:L2-Velocity-Averaging} and conclude that $\langle f_{n}^{\leq L}\,\phi\rangle$ induced tight laws on $[L^2_{t,x}]_{\loc}$.

Our next step is prove tightness of $\{\langle f_{n}\,\phi\rangle\}_{n\in\N}$ on $[L^1_{t,x}]_{\loc}$ by estimating the sequence $\{\langle f_{n}^{>L}\,\phi\rangle\}_{n \in \N}$. Indeed, since
\begin{equation}
  \|\langle f_n^{>L}\,\phi\rangle\|_{L^1_{t,x}} \leq \|f_n^{>L}\|_{L^1_{t,x,v}} \|\phi\|_{L^\infty_v},
\end{equation}
we only need to estimate $\{f_n^{>L}\}$ in $L^1_{t,x,v}$. Therefore, by Lemma \ref{lem:tightness-decomp}, it suffices to show that for any $\eta >0$,
\begin{equation}
  \lim_{L\to\infty} \sup_n\P\left( \|f_n^{>L}\|_{L^1_{t,x}} > \eta\right) = 0.
\end{equation}
Since $f_{n}^{>L}$ is renormalized, the following inequality holds $\p$ almost surely:
\begin{equation}\label{eq:f-geq-est}
  \|f_n^{> L}\|_{\LLLs{1}} \leq \|f_n^0\1_{|f_0^n|>L}\|_{\LLs{1}} + \|g_n\1_{|g_n|>L}\|_{\LLLs{1}}
\end{equation}
Since Hypothesis \ref{hyp:Stoch-Vel-Average-Fixed-n_hyp} gives uniform integrability of $\{f_n^0\}_{n\in \N}$, we may choose an $L_0>0$ such that for $L > L_0$,
\begin{equation} \label{eq:id-control}
\sup_{n \in \N} \|f_n^0\1_{|f_0^n|>L}\|_{\LLs{1}} \leq \eta/2.
\end{equation}
Therefore by the inequality (\ref{eq:f-geq-est}),
\begin{equation}\label{eq:f>L-meas-ineq}
  \P\left(\|f_n^{> L}\|_{\LLLs{1}} > \eta\right) \leq \P\left(\|g_n\1_{|g_n|>L}\|_{\LLLs{1}} > \eta/2\right).
\end{equation}
Since $\{g_n\}_{n\in\N}$ induces a tight family of laws on $[\LLLs{1}]_{w,\loc}$, it follows from the tightness criterion on $[L^1_{t,x,v}]_{w,\loc}$ given in Lemma \ref{lem:L1-tightness-crit} the right-hand side of inequality (\ref{eq:f>L-meas-ineq}) vanishes as $L\to \infty$, thereby proving tightness of the laws of $\{\langle f_{n}\,\phi\rangle\}_{n\in\N}$ on $[L^1_{t,x}]_{\loc}$.

Next we show that if in addition, for every $\eta > 0$ and $\phi\in C^\infty_c(\Rd_v)$ we have
\begin{equation}
    \lim_{R\to\infty}\sup_n\P\left(\|\langle f_n\,\phi\rangle\1_{|x|>R}\|_{L^1_{t,x}} >\eta \right) = 0,
  \end{equation}
  then $\{\langle f_{n}\,\phi\rangle\}_{n\in\N}$ has tight laws on $L^1_{t,x}$. To this end fix $\ep >0$ and $\phi \in C^\infty_c(\Rd_v)$ and use what we have just proved to produce a compact set $K\subseteq [L^1_{t,x}]_\loc$ such that
  \begin{equation}
    \P(\langle f_{n}\,\phi\rangle \notin K) < \ep.
  \end{equation}
  Next for each $k\in\N$, $k \geq 1$ choose $R_k$ such that
  \begin{equation}
    \sup_n\P\left(\|\langle f_n\,\phi\rangle\1_{|x|>R_k}\|_{L^1_{t,x}} >1/k \right) <\ep 2^{-k},
  \end{equation}
  and define the closed set $A_k$
  \begin{equation}
    A_k = \left\{f \in L^1_{t,x} :\|\langle f\,\phi\rangle\1_{|x|>R_k}\|_{L^1_{t,x}} \leq 1/k\right\}.
  \end{equation}
  It it straight forward to conclude that
  \begin{equation}
    \hat{K} = \bigcap_{k=1}^\infty K\cap A_k
  \end{equation}
  is tight (in the sense of functions in $[L^1_{t,x}]_\loc$) and therefore $\hat{K}$ is compact in $L^1_{t,x}$. It follows that
  \begin{equation}
    \P\left(\langle f_{n}\,\phi\rangle \notin \hat{K}\right) \leq \P\left(\langle f_{n}\,\phi\rangle \notin K\right) + \sum_{k=1}^\infty\P\left(\langle f_{n}\,\phi\rangle \notin A_k\right) < 2\ep
  \end{equation}

\end{proof}



\section{Approximating Scheme}\label{sec:Approximating-Scheme}
There are two main goals in this section.  First, for each $n \in \N$ fixed we will construct a renormalized weak martingale solution to the SPDE
\begin{equation} \label{eq:Approx-Boltzmann}
\begin{cases}
&\partial_{t}f_{n}+v \cdot \nabla_{x}f_{n}+\Div_{v}(f_{n}\sigma_{k}^{n}\circ \dot{\beta}_{k})=\mathcal{B}_{n}(f_n,f_n)\\
&f_{n}\mid_{t=0}=f_{n}^{0},
\end{cases}
\end{equation}
where the initial datum $f_{n}^{0}$ and the noise coefficients $\sigma^{n}$ are sufficiently regular, and $\mathcal{B}_{n}$ is an approximation to $\mathcal{B}$ involving a truncation and a regularized collision kernel $b_{n}$.  The second goal is to rigorously establish the uniform bounds on $\{f_{n}\}_{n \in \N}$ obtained formally in Section 2.  Towards this end, our regularizations are chosen to satisfy the following hypotheses.
\begin{hyp}[Initial Data]\label{hyp:regularized-idata}
\hspace{1in}
\begin{enumerate}
\item For each $n \in \N$, $f_{n}^{0}$ is smooth, non-negative and bounded from above. 
\item There exists a constant $C_{n}$ such that for all $(x,v) \in \R^{2d}$, $f_n^0$ has the lower bound 
\begin{equation}
f_{n}^{0}(x,v) \geq C_{n}e^{-|x|^{2}-|v|^{2}}.
\end{equation}
\item For all $j \in \N$, $(1+|x|^{2}+|v|^{2})^{j}f_{n}^{0} \in \LLs{1}$,
\item The sequence $\{ (1+|x|^{2}+|v|^{2}+|\log f_{n}^{0}|)f_{n}^{0}) \}_{n \in \N}$ is uniformly bounded in $\LLs{1}$ and $\{f_{n}^{0}\}_{n \in \N}$ converges to $f_{0}$ strongly in $\LLs{1}$.
\end{enumerate}
\end{hyp}
\begin{hyp}[Noise Coefficients]\label{hyp:regularized-noise-coeff}
\hspace{1in}
\begin{enumerate}
\item For each $k,n \in \N$, the noise coefficient $\sigma^{n}_{k} \in C^{\infty}(\R^{2d}; \R^{d})$ and $\Div_{v}\sigma^{n}_{k}=0$.
\item For $k>n$, the noise coefficient $\sigma^{n}_{k}$ vanishes identically.  
\item The sequences $\{\sigma^{n}\}_{n \in \N}$ and $\{\sigma^{n} \cdot \nabla_{v}\sigma^{n}\}_{n \in \N}$ converge pointwise to $\sigma$ and $\sigma \cdot \nabla_{v}\sigma$, are uniformly bounded in the spaces $\ell^{2}(\N ; L^{\infty}_{x,v})$ and $\ell^{1}(\N ; L^{\infty}_{x,v})$. Furthermore we have
  \begin{equation}
    \lim_{M\to\infty}\sup_n\ \sum_{k=M}^\infty\|\sigma^n_k\|_{L^\infty_{x,v}} = 0,\quad \lim_{M\to\infty}\sup_n\sum_{k=M}^\infty\|\sigma^n_k\cdot\nabla_v\sigma_k^n\|_{L^\infty_{x,v}} =0.
  \end{equation}
\end{enumerate}
\end{hyp}
\begin{remark}
  Note that hypothesis \ref{hyp:regularized-noise-coeff} can be acheived by, for instance by a standard mollification procedure and Hypothesis \ref{hyp:Noise-Coefficients}.
\end{remark}
\begin{hyp}[Collision Kernel]\label{hyp:regularized-collision}
\hspace{1in}
\begin{enumerate}
\item For each $n \in \N$, $b_{n}$ is smooth and compactly supported in $\R^{d} \times \S^{d-1}$.
\item The sequence $\{b_{n}\}_{n \in \N}$ is bounded in $L^{\infty}(\R^{d} \times \S^{d-1})$ and converges strongly to $b$ in $L^{1}(\R^{d} \times \S^{d-1})$.
\end{enumerate}
\end{hyp}

Following DiPerna/Lions \cite{diperna1989cauchy}, the truncated collision operator $\mathcal{B}_{n}$ is defined for $f \in L^{1}(\R^{d}_{v})$ by
\begin{equation}
\mathcal{B}_{n}(f,f)=\frac{1}{1+n^{-1}\int_{\Rd}f\dv}\iint_{\R^{d} \times \S^{d-1}}(f'f_{*}'-ff_{*})b_{n}(v-v_{*},\theta)\dv_{*}\dee\theta.
\end{equation}
The following lemma provides the necessary boundedness and continuity properties of the operator $\mathcal{B}_{n}$.  The method of proof is classical, see \cite{diperna1989cauchy} or \cite{Cercignani2013-vz} for most of the ideas.
\begin{lem}\label{lem:Trunc-Collision-Props}
For each $n \in \N$, there exists a constant $C_{n}$ such that
\begin{enumerate}
\item For all $f,g \in \LLs{1}$ it holds:
\begin{equation} \label{eq:Lip-L1}
\|\mathcal{B}_{n}(f,f)-\mathcal{B}_{n}(g,g)\|_{\LLs{1}} \leq C_{n} \|f-g\|_{\LLs{1}}.
\end{equation}
\item For all $f$ such that $(1+|x|^{2}+|v|^{2})^{k}f \in \LLs{1}$ and $k \in \N$, it holds
\begin{equation} \label{eq:Lip-Weighted-Space }
\|(1+|x|^{2}+|v|^{2})^{k}\mathcal{B}_{n}(f,f)\|_{\LLs{1}} \leq C_{n} \|(1+|x|^{2}+|v|^{2})^{k}f\|_{\LLs{1}}.
\end{equation}
\item For all $f \in \LLs{\infty}$ it holds:
\begin{equation} \label{eq:Lip-Linf}
\|\mathcal{B}_{n}(f,f)\|_{\LLs{\infty}} \leq C_{n} \|f\|_{\LLs{\infty}}.
\end{equation}
\end{enumerate}
\end{lem}
The strategy for solving the SPDE \eqref{eq:Approx-Boltzmann} involves a sequence of successive approximations based on mild formulation of \eqref{eq:Approx-Boltzmann} in terms of stochastic flows.  Namely, we fix a probability space $(\Omega, \mathcal{F},\p)$ and a collection of independent, one dimensional Brownian motions $\{\beta_{k}\}_{k \in \N}$.  The filtration generated by the Brownian motions is denoted $(\mathcal{F}_{t})_{t=0}^{T}$.  For each $n \in \N$, the smoothing regularizations present in Hypothesis \ref{hyp:regularized-noise-coeff}, in particular the $L^\infty$ bounds on $\sigma^n$ and $\sigma^n\cdot\nabla_v\sigma^n$ allow us to apply the results of Kunita \cite{kunita1997stochastic} to obtain a collection of stochastic flows of volume preserving homeomorphisms $\{\Phi_{s,t}^{n}\}_{n\in\N}$, $0\leq s \leq t \leq T$, $\Phi^n_{s,s}(x,v) = (x,v)$, associated to the Stratonovich SDE
\begin{equation}\label{eq:regularized-Strat-SDE}
\dee X_{t}^{n}=V_{t}^{n}\dt, \quad \dee V_{t}^{n}=\sum_{j=1}^{n}\sigma^{n}_{j}(X_{t}^{n}, V_{t}^{n})\circ \dee\beta_{j}.
\end{equation}
The corresponding inverse (in $(x,v)$) stochastic flows will be denoted $\{\Psi_{s,t}^{n}\}_{n\in\N}$. These objects have been studied at length by Kunita \cite{kunita1997stochastic}, so we will mostly defer to this reference for proofs of their properties. The main fact needed for our purposes concerns the following $\p$ almost sure growth estimates for the flow, which can be found as exercises (Exercises 4.5.9 and 4.5.10) in Kunita \cite{kunita1997stochastic}, Chapter 4, Section 5.

\begin{lem} \label{lem:Stoch-Flow-Growth}
Let $\epsilon \in (0,1)$. For each $n \in \N$, the following limits holds $\p$ almost surely:
\begin{align}\label{eq:forward-flow-bound}
&\lim_{(x,v) \to \infty}\sup_{\{s,t \in [0,T], s\leq t\}}\frac{|\Phi_{s,t}^{n}(x,v)|}{(1+|x|+|v|)^{1+\epsilon}}=0, \\
&\lim_{(x,v) \to \infty}\sup_{\{s,t \in [0,T], s\leq t\}}\frac{(1+|x|+|v|)^\ep}{(1+|\Phi_{s,t}^{n}(x,v)|)}=0. 
\end{align}
\end{lem}

Our next step is to apply Lemmas \ref{lem:Trunc-Collision-Props} and \ref{lem:Stoch-Flow-Growth} to establish the following existence result.
\begin{prop}\label{prop:exist-approx-scheme}
  Fix a stochastic basis $(\Omega, \mathcal{F}, (\mathcal{F}_t)_{t\in[0,T]}, \{\beta_k\}_{k\in\N}, \P)$. For each $n \in \N$ there exists an analytically weak, stochastically strong solution to the truncated Boltzmann equation
  \begin{equation}
    \begin{aligned}
      &\partial_t f_n + v\cdot\nabla_x f_n + \sigma_k^n\cdot\nabla_v f_n\strat \dot{\beta}_k = \BCol_n(f_n,f_n)\\
      &f_{n}|_{t=0} = f_n^0.
    \end{aligned}
  \end{equation}
such that $f_{n}$ has the following properties:
\begin{enumerate}
\item $f_n: \Omega \times [0,T] \to L^1_{x,v}$ is a $\mathcal{F}_t$ progressively measurable process. 
\item $f_n$ belongs to $L^2(\Omega; C_t(\LLs{1}))\cap L^\infty(\Omega\times[0,T]\times\R^{2d})$.
\item There exists a constant $\overline{C}_{n}$ such that for each $t \in [0,T]$, $\p$ almost surely
\begin{equation}\label{eq:exp-lower-bound}
f_{n}(t) \geq e^{-\overline{C}_{n}t}f_{n}^{0}\circ \Psi_{0,t}^{n}.
\end{equation}
\item For all $j \in \N$, $(1+|x|^{2}+|v|^{2})^{j}f_{n}$ is in $L^{\infty-}(\Omega ;\LLL{\infty}{1})$.
\item The sequence $\{ (1+|x|^{2}+|v|^{2})f_{n} \}_{n \in \N}$ is uniformly bounded in $L^{p}\big (\Omega;\LLL{\infty}{1} \big )$ for each $p\in [1,\infty)$.
\end{enumerate}
\end{prop}

\begin{proof}
Begin by constructing a sequence of successive approximations $\{f_{n,k}\}_{k \in \N}$.
For each $k \in \N$, define $\{f_{n,k}\}_{k \in \N}$ over $[0,T]$ by the relation
\begin{equation}\label{eq:Successive_Approx}
f_{n,k}(t)=f_{n}^{0} \circ \Psi_{0,t}^{n}+\int_{0}^{t}\mathcal{B}_{n}\big(f_{n,k-1}(s),f_{n,k-1}(s) \big) \circ \Psi_{s,t}^{n}\,\ds,\quad f_{n,0} = 0.
\end{equation}
Applying classical results of Kunita \cite{kunita1997stochastic}, it follows that $f_{n,k}$ is a stochastically strong, classical solution to
\begin{equation}\label{eq:successive_approx-Eulerian}
  \begin{aligned}
    &\partial_tf_{n,k} + v\cdot\nabla_xf_{n,k} + \sigma^n_j\cdot\nabla_vf_{n,k}\strat\dot{\beta}_j = \BCol_n(f_{n,k},f_{n,k}),\\
    &f_{n,k}|_{t=0} = f_n^0.
  \end{aligned}
\end{equation}

Let $X_{T}$ be the Banach space of $(\mathcal{F}_{t})_{t=0}^{T}$ progressively measurable processes $f: [0,T] \times \Omega \to \LLs{1}$ endowed with the $L^{2}(\Omega; \CLL{1})$ norm.  Let $C_{n}$ be the constant corresponding to the continuity estimates for $\mathcal{B}_{n}$ from Lemma \ref{lem:Trunc-Collision-Props}.  In addition, observe that the Hypothesis $\Div_{v}\sigma_{k}^{n}=0$ implies that for every $s,t \in [0,T]$, $s<t$, the flow map $\Phi_{s,t}$ is almost surely volume preserving (see Kunita \cite{kunita1997stochastic} Theorem 4.3.2 for more details).  Taking $\LLs{1}$ norms on both sides of \eqref{eq:Successive_Approx}, maximizing over $[0,T]$, and using the Lipschitz continuity of $\mathcal{B}_{n}$ in $\LLs{1}$ obtained in Lemma \ref{lem:Trunc-Collision-Props}, we find 
\begin{equation}
\|f_{n,k+1}-f_{n,k}\|_{X_{T}} \leq (C_{n}T)^k\|f^0_n\circ \Psi^n_{0,t}\|_{X_T} = (C_{n}T)^k\|f^0_n\|_{L^1_{x,v}},
\end{equation}
for each $k\in\N$. Choosing $T$ small enough, the sequence $\{f_{n,k}\}_{k \in \N}$ is Cauchy in $X_{T}$. Applying this argument a finite number of times, we may remove the constraint on $T$. Therefore, for each $n \in \N$, there exists an $f_{n} \in X_{T}$ such that $\{f_{n,k}\}_{k \in \N}$ converges to $f_{n}$ in $L^{2}(\Omega ; C_{t}(L^{1}_{x,v}))$. In view of Lemma \ref{lem:Trunc-Collision-Props}, $\mathcal{B}_{n}$ is continuous on $L^{2}(\Omega ; C_{t}(L^{1}_{x,v}))$. Therefore we have more then enough to pass the limit weakly in each term of equation (\ref{eq:successive_approx-Eulerian}) 


Our next step is to verify the lower bound \eqref{eq:exp-lower-bound}.  Let $\overline{C}_{n}$ be a deterministic constant to be selected.  In view of \eqref{eq:Successive_Approx} and the fact that $\Psi_{s,t}^{n} \circ \Phi_{0,t}^{n}=\Phi_{0,s}^{n}$ for $s<t$, the following inequalities hold $\p$ almost surely:
\begin{equation}
\begin{aligned}
e^{\overline{C}_{n}t}f_{n,k+1}(t) \circ \Phi_{0,t}^{n} &= f_{0}^{n}+\int_{0}^{t}e^{\overline{C}_{n}s}\mathcal{B}_{n}\big(f_{n,k}(s),f_{n,k}(s) \big) \circ \Phi_{0,s}^{n}\ds+\overline{C}_{n}\int_{0}^{t}e^{\overline{C}_{n}s}f_{n,k}(s)\circ \Phi_{0,s}^{n}\ds \\
&\geq f_{0}^{n}-\int_{0}^{t}e^{\overline{C}_{n}s}\mathcal{B}_{n}^{-}\big(f_{n,k}(s),f_{n,k}(s) \big) \circ \Phi_{0,s}^{n}\ds+\overline{C}_{n}\int_{0}^{t}e^{\overline{C}_{n}s}f_{n,k}(s)\circ \Phi_{0,s}^{n}\ds\\
&\geq f_{0}^{n}+[\overline{C}_{n}-n|\overline{b}_{n}|_{L^{\infty}_{v}}]\int_{0}^{t}e^{\overline{C}_{n}s}f_{n,k}(s) \circ \Phi_{0,s}^{n}\ds.
\end{aligned}
\end{equation}
In the last line, we used the explicit definition of the operator $\mathcal{B}_{n}^{-}$ together with Young's inequality and the fact that the flow map is volume preserving.  Choose $\overline{C}_{n}>n|\overline{b}_{n}|_{L^{\infty}_{v}}$ and apply the inequality above inductively to obtain the non-negativity of $f_{n,k}(t) \circ \Phi_{0,t}^{n}$, which consequently yields the more precise bound $e^{\overline{C}_{n}t}f_{n,k+1}(t) \circ \Phi_{0,t}^{n} \geq f_{0}^{n}$.
Passing $k \to \infty$ and using the $L^{2}(\Omega; C_{t}(\LLs{1}))$ convergence of $\{f_{n,k}\}_{k \in \N}$ towards $f_{n}$, we find that $e^{\overline{C}_{n}t}f_{n}(t) \circ \Phi_{0,t}^{n} \geq f_{0}^{n}$ for all $t \in [0,T]$ with probability one.  Composing with $\Psi_{0,t}^{n}$ on both sides gives the desired lower bound \eqref{eq:exp-lower-bound}. 

Our next step is prove that $f_n$ is in $L^\infty(\Omega\times[0,T]\times\R^{2d})$. We will do this be first checking that the sequence $\{f_{n,k}\}_{k \in \N}$ is uniformly (in $k$ only) bounded in $L^{\infty}(\Omega \times [0,T] \times \R^{2d})$. By Hypothesis \ref{hyp:regularized-idata}, $f_{n}^{0}$ is bounded. Taking $\LLs{\infty}$ norms on both sides of \eqref{eq:Successive_Approx}, then maximizing over $t \in [0,T]$ yields $\p$ almost surely
\begin{equation}
\|f_{n,k+1}\|_{L^{\infty}_{t,x,v}} \leq \|f_{n}^{0}\|_{\LLs{\infty}}+C_{n}T \|f_{n,k}\|_{L^{\infty}_{t,x,v}},
\end{equation}
where $C_{n}$ is the constant from Lemma \ref{lem:Trunc-Collision-Props}.  
Iterating, and summing the geometric series, we find that if $T<C_{n}^{-1}$, 
\begin{equation}
\|f_{n,k}\|_{L^{\infty}_{t,x,v}} \leq (1-C_{n}T)^{-1}\|f_{n}^{0}\|_{\LLs{\infty}}.
\end{equation}
Of course we may repeat this argument a finite number of times to remove the restriction on $T$. Taking $L^{\infty}(\Omega)$ norms on both sides of the above inequality yields the uniform bound. By weak-* $L^\infty$ sequential compactness of $L^\infty(\Omega\times[0,T]\times\R^{2d})$,  $f_n$ belongs to $L^\infty(\Omega\times[0,T]\times\R^{2d})$.

Our next goal is to establish the following uniform estimate: for all $p \in (1, \infty)$
\begin{equation}
\sup_{k,n \in \N} \E \|(1+|x|^{2}+|v|^{2})f_{n,k}\|_{\LLL{\infty}{1}}^{p} \leq C_{p},
\end{equation}
where $C_{p}$ depends only on $f_{0}$ and $\sigma$. If the process $(1+|x|^{2}+|v|^{2})f_{n,k}$ was known a priori to belong to $L^{\infty-}(\Omega; C_{t}(\LLs{1}))$, we could argue exactly as in the formal estimates Section \ref{subsec:MomentBound}. Since this is a priori unknown, we proceed by a stopping time argument based on the characteristics. Define for each $R\geq 0$, the stopping time
\begin{equation}
\tau^n_{R}= \inf \bigg \{t \in [0,T] \mid \sup_{s\in [0,t], (x,v)\in \R^{2n}}\frac{|\Phi^n_{s,t}(x,v)|}{(1+|x|+|v|)^{2}}\geq R \bigg \}.
\end{equation}
To see that this stopping time is well defined it suffices to show that the process
\begin{equation}\label{eq:cont-arguement-process}
t\mapsto \sup_{s\in [0,t], (x,v)\in \R^{2n}}\frac{|\Phi^n_{s,t}(x,v)|}{(1+|x|+|v|)^{2}}
\end{equation}
is adapted to $(\mathcal{F}_t)_{t\geq 0}$ and has continuous sample paths. Indeed, Lemma 4.5.6 of \cite{kunita1997stochastic} implies that $\Phi^n_{s,t}(x,v)$ is jointly continuous in $(s,t,x,v)$ and therefore the suprema in (\ref{eq:cont-arguement-process}) can be taken over a countable dense subset of $[0,t]\times\R^{2d}$, implying adaptedness. Furthermore, the decay estimate presented in Lemma \ref{lem:Stoch-Flow-Growth} allows the supremum in $(x,v)$ to be taken over a compact set in $\R^{2d}$. Continuity of the process in (\ref{eq:cont-arguement-process}) follows from the fact that for any jointly continuous function $f(x,y)$, $f:X\times Y \to \R$, where $X$ and $Y$ are two compact metric spaces, the function $g(x) = \sup_{y\in Y} g(x,y)$ is continuous. 

For each $t\in [0,T]$ we now define the stopped process $f_{n,k}^{R}(t)=f_{n,k}(t \wedge \tau^n_{R})$. We will verify that for each $k, n \in \N$ and $R>0$, the process $(1+|x|^{2}+|v|^{2})^{j}f_{n,k}^{R}$ belongs to the space $L^{\infty-}(\Omega; \LLL{\infty}{1})$ for all $j\geq 1$. The claim will be established by induction on $k \in \N$. Suppose the claim is true for step $k-1$. To check $k$, note that
\begin{equation}
\begin{aligned}
&\|(1+|x|^{2}+|v|^{2})^{j}f_{n,k}^{R}(t)\|_{\LLs{1}} \leq \|(1+|x|^{2}+|v|^{2})^{j}f_{n}^{0} \circ \Psi^n_{0,t \wedge \tau^n_{R}}\|_{\LLs{1}}\\
&\hspace{.5in} +\int_{0}^{T}\big \|\1_{s\in[0,t\wedge \tau_R^n]}(1+|x|^{2}+|v|^{2})^{j}\mathcal{B}_{n}\big(f_{n,k-1}(s),f_{n,k-1}(s) \big) \circ \Psi^n_{s,t \wedge \tau^n_{R}} \big \|_{\LLs{1}}\ds
\end{aligned}
\end{equation}
Using the volume preserving property of the stochastic flow, the right-hand side above is equal to
\begin{equation}
\begin{aligned}
&\|(1+|\Phi^n_{0,t\wedge\tau^n_R}(x,v)|^{2})^{j}f_{n}^{0}\|_{\LLs{1}}\\
&\hspace{.5in}+\int_{0}^{T}\big \|\1_{s\in[0,t\wedge \tau_R^n]}(1+|\Phi^n_{s,t\wedge \tau^n_R}(x,v)|^{2})^{j}\mathcal{B}_{n}\big(f_{n,k-1}(s),f_{n,k-1}(s) \big) \big \|_{\LLs{1}}\ds
\end{aligned}
\end{equation}
Using the definition of the stopping time to bound the flow and the $L^1$ bound on $\mathcal{B}_n$ in Lemma \ref{lem:Trunc-Collision-Props}, we obtain
\begin{equation}
\begin{aligned}
\|(1+|x|^{2}+|v|^{2})^{j}f_{n,k}^{R}(t)\|_{\LLs{1}} &\leqc R^{2j}\|(1+|x|^{2}+|v|^{2})^{2j}f_{n}^{0}\|_{\LLs{1}} \\
&\hspace{-1in}+R^{2j}\int_{0}^{T}\big \|\1_{s\in[0,t\wedge \tau^n_R]}(1+|x|^{2}+|v|^{2})^{2j}\mathcal{B}_{n}\big(f_{n,k-1}(s),f_{n,k-1}(s) \big)\big \|_{\LLs{1}}\ds \\
&\leqc (1+T)R^{2j}\|(1+|x|^{2}+|v|^{2})^{2j}f_{n,k-1}^{R}\|_{\LLL{\infty}{1}}.
\end{aligned}
\end{equation}
Taking the supremum in time, and the $L^{p}(\Omega)$ norm on both sides, we may use the inductive hypothesis to complete the inductive step. The base case is established in the same way. Therefore $(1+|x|^{2}+|v|^{2})^{j}f_{n,k}^{R}$ belongs to the space $L^{\infty-}(\Omega; \LLL{\infty}{1})$ for all $j \geq 1$.

Now, if one follows the argument in the a priori moment bounds section~\ref{subsec:MomentBound}, specifically multiplying the truncated Boltzmann equation for $f_{n,k}^R$ by $(1 + |x|^2 + |v|^2)$ and integrating in $(x,v)$ so as to kill the the collision operator, one may close the estimates on $(1+|x|^2 + |v|^2)f_{n,k}^R$ uniformly in $k$ using the BDG inequality, Gr\"{o}nwall's lemma and the uniform hypothesis \ref{hyp:regularized-idata} and \ref{hyp:regularized-noise-coeff} on the initial data and noise coefficients to find for all $R>0$
\begin{equation} \label{eq:unif-weighted-bounds}
\E\|(1+|x|^{2}+|v|^{2})f_{n,k}\1_{t\in[0,T\wedge\tau_R^n]}\|^p_{\LLL{\infty}{1}}\leq C_{p,T}
\end{equation}
It is important to note that the constant $C_p$ above does not depend on $R$, $n$ or $k$. The independence of $C_{p,T}$ from $R$ can be readily seen from the fact that the constant obtained in Section \ref{subsec:MomentBound} depends only in an increasing way on the final time $T$. 

Now we wish to send $R \to \infty$ on both sides of this inequality. To achieve this, note that Lemma \ref{lem:Stoch-Flow-Growth} implies that $\p$ almost surely,
\begin{equation}\label{eq:uniform-in-k-moment-bound}
\bigg \|\sup_{s,t \in [0,T], s<t}\frac{|\Phi^n_{s,t}(x,v)|}{(1+|x|+|v|)^{2}} \bigg\|_{\LLs{\infty}} < \infty.
\end{equation}
Hence,
\begin{equation}
\lim_{R \to \infty}\p(\tau^n_{R} \leq T) =\lim_{R \to \infty} \p\bigg ( \bigg \|\sup_{s,t \in [0,T], s<t}\frac{|\Phi^n_{s,t}(x,v)|}{(1+|x|+|v|)^{2}} \bigg \|_{\LLs{\infty}} \geq R \bigg )=0.
\end{equation}
Therefore, it follows that $\tau^n_{R} \wedge T$ converges in probability to $T$, and by the monotone convergence theorem we deduce that for any $p\in [1,\infty)$,
\begin{equation}
\E\|(1+|x|^{2}+|v|^{2})f_{n,k}\|_{\LLL{\infty}{1}}^{p}\leq C_{p}.
\end{equation}

Next we claim that the sequence $\{(1+|x|^{2}+|v|^{2})^{j}f_{n,k}\}_{k \in \N}$ is uniformly bounded (in $k$) in $L^{\infty-}(\Omega; \LLL{\infty}{1})$. We can estimate $(1+|x|^{2}+|v|^{2})^{j}f_{n,k}^{R}$ in a similar way to $(1+|x|^2 + |v|^2)f_{n,k}^R$, by multiplying the truncated Boltzmann equation for $f_{n,k}^R$ by $(1 + |x|^2 + |v|^2)^j$ and using estimate 2 in Lemma \ref{lem:Trunc-Collision-Props} to bound the collision operator. Using the BDG inequality and Gr\"{o}nwall inequality one can obtain after some tedious, though straight forward, calculations and using the uniform hypothesis \ref{hyp:regularized-noise-coeff} on the noise coefficients,
\begin{equation}
\begin{aligned}
  &\E\|(1+|x|^{2}+|v|^{2})^jf_{n,k}\1_{t\in[0,T\wedge\tau_R^n]}\|^p_{\LLL{\infty}{1}} \leq C_{p,T,j}\|(1+|x|^2+|v|^2)^jf_{n}^0\|_{L^1_{x,v}}^p\\ 
&\hspace{1in}  + TC_{p,T,n,j}\E\|(1+|x|^{2}+|v|^{2})^jf_{n,k-1}\1_{t\in[0,T\wedge\tau_R^n]}\|^p_{\LLL{\infty}{1}},
\end{aligned}
\end{equation}
where the constants $C_{p,T,n,j}$ and $C_{p,T,j}$ are independent of $k$ and $R$ and depend on the final time in an increasing way. Since we have made explicit that there is a multiplicative factor in the second term $T$ above (coming from the time integral of the collision operator), we find that, independently of $k$ and the initial data we may choose $T$ small enough so that $TC_{p,T,n,j} < 1$. This means that we may iterate the bound above and sum the geometric series to conclude that for such $T$, to conclude
\begin{equation}
  \E\|(1+|x|^{2}+|v|^{2})^jf_{n,k}\1_{t\in[0,T\wedge\tau_R^n]}\|^p_{\LLL{\infty}{1}} \leq C_{p,T,n,j}.
\end{equation}
Again, sending $R \to \infty$ and using monotone convergence we conclude the uniform in $k$ estimate
\begin{equation}
  \E\|(1+|x|^{2}+|v|^{2})^jf_{n,k}\|^p_{\LLL{\infty}{1}} \leq C_{p,T,n,j}.
\end{equation}
The restriction on $T$ can be removed in the usual way by repeating the above argument a finite number of times.

What remains is to pass the limit in $k$ on these estimates to obtain the estimates on $f_n$ stated in the Lemma. It suffices to show that for each $j\geq 0$, and $p\in[1,\infty)$,
\begin{equation}\label{eq:passing-k-ineq-moments}
  \E\|(1+|x|^2 +|v|^2)^jf_n\|^p_{\LLL{\infty}{1}} \leq \sup_{k\in \N} \E\|(1+|x|^2 +|v|^2)^jf_{n,k}\|^p_{\LLL{\infty}{1}}.
\end{equation}
We do this by cutting off the moment function. Let $B_M$ be the ball of radius $M> 0$ in $\R^{2d}$. Since $f_{n,k} \to f_n$ in $L^2(\Omega;\LLL{\infty}{1}))$, upon choosing a further subsequence if necessary, we have that $\P$ almost surely,
\begin{equation}
  \|(1+|x|^2 +|v|^2)^jf_{n,k}\1_{B_M} \|^p_{\LLL{\infty}{1}} \to \|(1+|x|^2 +|v|^2)^jf_{n}\1_{B_M} \|^p_{\LLL{\infty}{1}}.
\end{equation}
Applying Fatou's Lemma, gives
\begin{equation}
  \E \|(1+|x|^2 +|v|^2)^jf_{n}\1_{B_M} \|^p_{\LLL{\infty}{1}} \leq \sup_{k\in \N} \E \|(1+|x|^2 +|v|^2)^jf_{n,k}\|^p_{\LLL{\infty}{1}}.
\end{equation}
The inequality (\ref{eq:passing-k-ineq-moments}) is then proved by passing the limit in $M$ on the left-hand side by monotone convergence.
\end{proof}
The final step in this section is to realize the a priori estimates obtained from the formal entropy dissipation inequality \eqref{eq:entropy-dis}.  Towards this end, define the approximate entropy dissipation $f \to \mathcal{D}_{n}(f)$ by the relation
\begin{equation} \label{eq:Def-Of-Approx-Entropy-Diss}
\begin{aligned}
\mathcal{D}_{n}(f) &\equiv \frac{1}{4}\big (1+n^{-1}\langle f, 1\rangle \big )^{-1}\iiint_{\R^{2d}\times\S^{d-1}}d(f)\,b_{n}(v-v_*,\theta)\,\dee\theta\dv_*\dv,\\
\end{aligned}
\end{equation}
where $d(f)$ is defined by \eqref{eq:Entopy-Dissipation-Def}. Similarly, define $\SDis^0_n(f)$ by
\begin{equation}
  \SDis_n^0(f) \equiv \frac{1}{4}(1+ n^{-1}\langle f, 1\rangle)^{-1}\iint_{\R^d\times \S^{d-1}} d(f)b_n(v-v_*,\theta)\dee\theta\dee v_*.
\end{equation}
\begin{lem}\label{lem:entropy-Bounds}
Let $\{f_{n}\}_{n \in \N}$ be the sequence constructed in Proposition \ref{prop:exist-approx-scheme}. For each $p \in (1,\infty)$, there exists a constant $C_{p}$ depending on $\sigma$ and $f_{0}$ such that
\begin{equation}
\sup_{n \in \N} \E  \| f_{n}\log f_{n} \|_{\LLL{\infty}{1}}^{p}\leq C_{p}, \quad \sup_{n \in \N} \E  \|\mathcal{D}_{n}(f_{n})\|_{\LLt{1}}^{p}\leq C_{p}.
\end{equation}
\end{lem}
\begin{proof}
Begin by fixing $n \in \N$.  Note that it suffices to verify identity \eqref{eq:entropy-identity} from the formal a priori bounds section.  For each $\epsilon>0$, we define the renormalization $\beta_{\epsilon}(x)=x\log(x+\epsilon)$.  Using Proposition \ref{prop:Weak_Is_Renormalized} and the fact that $f_{n}$ belongs to  $L^{\infty}(\Omega \times [0,T] \times \R^{2d})$ and $L^{2}(\Omega ; \CL{1})$, it can be checked with a truncation argument that $\beta_{\epsilon}(f_{n})$ is a weak solution to the stochastic kinetic equation driven by $\beta_{\epsilon}'(f_{n})\mathcal{B}_{n}(f_{n},f_{n})$, starting from $\beta_{\epsilon}(f_{0}^{n})$. In particular, using the $L^1$ bounds on $f^n$ and the fact that $\BCol(f_n,f_n) \in L^1_{t,x,x}$, we can obtain the $\p$ almost sure identity
\begin{equation}\label{eq:renorm-approx-entropy}
\iint_{\R^{2d}}\beta_{\epsilon}(f_{n}(t))\dx \dv = \iint_{\R^{2d}}\beta_{\epsilon}(f_{n}^{0})\dx \dv+\int_{0}^{t}\iint_{\R^{2d}}\beta_{\epsilon}'(f_{n})\mathcal{B}_{n}(f_{n},f_{n})\dx \dv \ds.
\end{equation}
Observe that almost everywhere in $\Omega \times [0,T] \times \R^{2d}$, as $\epsilon \to 0$ we have the convergence $\beta_{\epsilon}(f_{n}(t)) \to f_{n}\log f_{n}(t)$ and $\beta_{\epsilon}'(f_{n})\mathcal{B}_{n}(f_{n},f_{n}) \to [1+\log f_{n} ]\mathcal{B}_{n}(f_{n},f_{n})$.  Since $f_{n}$ is in $L^{\infty}(\Omega \times [0,T] \times \R^{2d})$ and in $L^{2}(\Omega ; \CLL{1})$, it follows that $\p$ almost surely, for each $t \in [0,T]$
\begin{equation} 
\iint_{\R^{2d}}\beta_{\epsilon}(f_{n}(t))\dx \dv \to \iint_{\R^{2d}}f_{n}(t)\log f_{n}(t)\dx \dv.
\end{equation}
The initial data are also handled similarly in view of Hypothesis \ref{hyp:regularized-idata}.  To pass the limit in the remaining integral on the RHS of \eqref{eq:renorm-approx-entropy}, note that $|\beta_{\epsilon}'(x)|\leq(2+|\log(x)|)$ for $\epsilon$ small. Hence, by the dominated convergence theorem, it suffices to show that $\log f_{n}\mathcal{B}_{n}(f_{n},f_{n})$ belongs to $\LLLs{1}$ with probability one. By Proposition \ref{prop:exist-approx-scheme} combined with Hypothesis \ref{hyp:regularized-idata} we have 
\begin{equation}\label{eq:upper-lower-bound-f_n}
C_{n}e^{-\overline{C}_{n}t}e^{-|\Psi_{0,t}^{n}|^{2}}\leq f_{n} \leq \|f_{n}\|_{L^{\infty}(\Omega \times [0,T] \times \R^{2d})}.
\end{equation}
The second estimate on $\Phi^n_{0,t}$ given in Lemma \ref{lem:Stoch-Flow-Growth}, implies that $\P$ almost surely we have the bound,
\begin{equation}
  \sup_{(t,x,v) \in [0,T]\times \R^{2d}} \frac{|\Psi^n_{0,t}(x,v)|}{(1+|x| + |v|)^2} < \infty
\end{equation}
Combining this with the bounds in (\ref{eq:upper-lower-bound-f_n}) it follows that $\p$ almost surely
\begin{equation}
\sup_{(t,x,v) \in [0,T]\times \R^{2d}} \frac{|\log f_{n}(t,x,v)|}{(1+|x|^{2}+|v|^{2})^{2}}<\infty.
\end{equation}
Using this, the $\P$ almost sure $L^1_{t,x,v}$ estimate on $\log f_{n}\mathcal{B}_{n}(f_{n},f_{n})$ now follows from property $3$ of Lemma \ref{lem:Trunc-Collision-Props} and the fact that $(1+|x|^{2}+|v|^{2})^{2}f_{n} \in \LLL{\infty}{1}$ with probability one.
\end{proof}

\section{Compactness and Preliminary Renormalization}\label{sec:Compactness-and-Prelim-Renorm}
Let $\{\tilde{f}_{n}\}_{n \in \N}$ be the sequence of renormalized weak martingale solutions to \eqref{eq:Approx-Boltzmann} constructed in Proposition \ref{prop:exist-approx-scheme}.  Denote the supporting stochastic basis by $(\tilde{\Omega}, \tilde{\mathcal{F}}, \tilde{\p}, (\tilde{\mathcal{F}}_{t})_{t=0}^{T}, \{ \tilde{\beta}_{k}\}_{k \in \N})$.  In view of Proposition \ref{prop:exist-approx-scheme} and Lemma \ref{lem:entropy-Bounds}, we have the uniform bounds 
\begin{equation}\label{eq:AveragedBoundsOnApprox}
\begin{split}
&\sup_{n \in \N} \tilde{\E} \|(1+|x|^{2}+|v|^{2}+|\log \tilde{f}_{n}|)\tilde{f}_{n}\|_{\LLL{\infty}{1}}^{p}< \infty \\
&\sup_{n \in \N}\tilde{\E} \|\mathcal{D}_{n}(\tilde{f}_{n})\|_{\LLt{1}}^{p}<\infty.
\end{split}
\end{equation}
In this section, we will deduce several key tightness results and apply our main stochastic velocity averaging Lemma \ref{lem:L1-Velocity-Averaging}. We will study the induced laws of the approximations $\{\tilde{f}_{n}\}_{n \in \N}$, the renormalized approximations $\{\Renorm(\tilde{f}_{n})\}_{n \in \N}$, and renormalized collision operators $\{\Renorm'(f_{n})\mathcal{B}_{n}(\tilde{f}_{n},\tilde{f}_{n})\}_{n \in \N}$. The precise results are stated in Lemmas \ref{lem:renorm-collision-tightness-bounds}-\ref{lem:renorm-approx-tightness-f}.  Combining our tightness result with a recent extension of the Skorohod Theorem \ref{Thm:Appendix:Jakubowksi_Skorohod} to non-metric spaces, we will obtain our main compactness result Proposition \ref{prop:Grand-Skorohod}.  

Towards this end, we introduce for each $m \in \N$ a truncation type renormalization $\Renorm_{m}$ defined by
\begin{align} \label{eq:TruncationDef}
\Renorm_{m}(z)=\frac{z}{1+m^{-1}z}.
\end{align}

\subsection{The space $\LLsMw{1}$}
In order the apply the velocity averaging results we will find it convenient to turn the tightness results on velocity averages of $f$ of proved in Section \ref{sec:Stoch-Velocity-Avg} into tightness results for $f$ on a particular space $\LLsMw{p}$ characterizing `convergence in the sense of velocity averages'. To be more precise, we introduce a topological vector space $\LLsMw{p}$ as follows. Let $\Meas_{v}$ denote the space of finite Radon measures on $\R^{d}_v$, which can be identified with the dual of the continuous functions $C_0(\Rd)$ that vanish at $\infty$, and let $\Meas_v^*$ be $\Meas_v$ equipped with it's weak star topology. Consider the collection of equivalence classes (up to Lebesgue $[0,T]\times\Rd_x$ null sets) of measurable maps $f: [0,T] \times \R^{d}_x \to \Meas_{v}^*$, where the Borel sigma algebra is taken on $\Meas_v^*$. For each equivalence class $f$, and $\phi\in C_0(\Rd)$ we let $\langle f,\phi\rangle$ denote the pair between $\Meas_v$ and $C_0(\Rd)$ and for each $\phi \in C_{0}(\R^{d})$, define a corresponding semi-norm $\nu_{\phi}$ via
\begin{equation}
  \nu_{\phi}(f) = \|\langle f,\phi \rangle\|_{L^{p}_{t,x}}.
\end{equation}
We then say that $f$ is in $\LLsMw{p}$ provided that for all $\phi \in C_{0}(\R^{d})$, $\nu_\phi(f) < \infty$. Convergence in the space $\LLsMw{p}$ can be thought of as {\it strong} in the variables $(t,x)$ and ${\it weak}$ in the velocity variable $v$. The space $\LLsMw{p}$ can be identified with $\mathcal{L}(C_0(\Rd),L^p_{t,x})$ the space of bounded linear operators from $C_0(\Rd)$ to $L^1_{t,x}$ under the topology of pointwise convergence (see Lemma \ref{lem:representation-pettis}).

We will also define the space $[\LLsMw{p}]_{\loc}$ of locally integrable functions which is the space of equivalence classes of measurable functions $f:[0,T]\times\Rd \to \mathcal{M}_v$ generated by the semi-norms,
\begin{equation}
\nu_{\phi,K}(f) = \|\langle f,\phi\rangle \1_{K}\|_{L^p_{t,x}}  
\end{equation}
for each $\phi \in C_0(\Rd)$ and each compact set $K\subseteq \Rd$. Again such a space has an identification with $\mathcal{L}(C_0(\Rd),[L^p_{t,x}]_\loc)$.

Our main tool for obtaining compactness in the space $\LLsMw{p}$ are Lemmas \ref{lem:compact-char-vel-avg} and \ref{lem:tightness-crit-vel-averaged-space}, which give necessary and sufficient conditions for compactness and tightness of measure on $\LLsMw{p}$.

\subsection{Statement of the main proposition}
The main result of this section is the following compactness result.
\begin{prop}\label{prop:Grand-Skorohod}
There exists a new probability space $(\Omega, \mathcal{F},\p)$ and a sequence of maps $\{\widetilde{T}_{n}\}_{n \in \N}$ from $\Omega$ to $\tilde{\Omega}$ with the following properties:
\begin{enumerate}
\item For each $n \in \N$, the map $\tilde{T}_{n}$ is measurable from $(\Omega, \mathcal{F})$ to $(\tilde{\Omega}, \tilde{\mathcal{F}})$ and $(\tilde{T}_{n})_{\#}\p=\tilde{\p}$.
\item The new sequence $\{f_{n}\}_{n \in \N}$ defined by $f_{n}=\tilde{f}_{n} \circ \widetilde{T}_{n}$ satisfies the uniform bounds \eqref{eq:AveragedBoundsOnApprox} with $\E$ replacing $\tilde{\E}$.  Moreover, for all $\omega \in \Omega$, there exists a constant $C(\omega)$ such that
\begin{equation}
\begin{split}
&\sup_{n \in \N}\|(1+|x|^{2}+|v|^{2}+|\log f_{n}(\omega)|)f_{n}(\omega)\|_{\LLL{\infty}{1}}\leq C(\omega). \\
&\sup_{n \in \N}\|\mathcal{D}_{n}(f_{n})(\omega)\|_{\LLt{1}}\leq C(\omega).
\end{split}
\end{equation}
\item The new sequence $\{\beta_{k}^{n}\}_{k \in \N}$ defined by $\beta_{k}^{n}=\tilde{\beta}_{k}^{n}\circ \widetilde{T}_{n}$ consists of one-dimensional Brownian motions on $(\Omega, \F, \p)$.
\item There exist random variables $f$ and $\{\beta_{k}\}_{k \in \N}$ with values in $\CLLw{1}$ and $[C_{t}]^{\infty}$ respectively, such that the following convergences hold pointwise on $\Omega$:
\begin{align}
f_{n} \to f \quad &\text{in} \quad \LLsMw{1} \cap \CLLw{1}.\\
\{\beta_{k}^{n}\}_{k \in \N} \to \{\beta_{k}\}_{k \in \N} \quad &\text{in} \quad [C_{t}]^{\infty}. 
\end{align}
\item For each $m \in \N$, there exist auxiliary random variables $\overline{{\Renorm}_{m}(f)}$ and $\overline{\gamma_{m}(f)}$ in $\CLLw{1}$ along with $\mathcal{B}_{m}^{-}$ and $\mathcal{B}_{m}^{+}$ in $\LLLs{1}$ and $\overline{\SDis^0(f)}$ in $\mathcal{M}_{t,x,v}$ such that the following convergences hold pointwise on $\Omega$:
\begin{align}
\Renorm_{m}(f_{n}) \to \overline{{\Renorm}_{m}(f)} \quad &\text{in} \quad \LLsMw{1}\cap\CLLw{1}.\\
\Renorm_{m}'(f_{n})f_{n} \to \overline{\gamma_{m}(f)} \quad &\text{in} \quad \LLsMw{1}\cap\CLLw{1}.\\
\Renorm_{m}'(f_{n})\mathcal{B}_{n}^{+}(f_{n},f_{n}) \to \mathcal{B}_{m}^{+} \quad &\text{in} \quad [\LLLs{1}]_{w}. \\
\Renorm_{m}'(f_{n})\mathcal{B}_{n}^{-}(f_{n},f_{n}) \to \mathcal{B}_{m}^{-} \quad &\text{in} \quad [\LLLs{1}]_{w}\\
\SDis^0_n(f_n) \to \overline{\SDis^0(f)} \quad &\text{in} \quad \mathcal{M}_{t,x,v}^*.
\end{align}
\end{enumerate}
\end{prop}
\begin{remark}
For all $n \in \N$, $f_{n}$ is a weak martingale solution to the stochastic kinetic equation driven by $\BCol_n(f_n,f_n)$, starting from $f_0$, with noise coefficients $\sigma^n$. The supporting stochastic basis is given by $(\Omega, \mathcal{F}, \P, (\F_t^n)_{t=0}^T, \{\beta_k^n\}_{k\in\N})$, where the Brownian motions are given by $\beta_{k}^{n}=\tilde{\beta}_{k}^{n}\circ \widetilde{T}_{n}$ and $\mathcal{F}_{t}^{n}=\tilde{T}^{-1}_{n}\circ \tilde{\mathcal{F}}_{t}$.  
\end{remark}

\subsection{Tightness of renormalized quantities} \label{sec:Renorm-Tightness}
In this section, we study the compactness properties of the sequences $\{ \Renorm(\tilde{f}_{n})\}_{n \in \N}$ and \\$\{\Renorm'(\tilde{f}_{n})\mathcal{B}_{n}^{+}(\tilde{f}_{n},\tilde{f}_{n})\}_{n \in \N}$, where $\Gamma$ is a renormalization of a particular type.

\begin{definition}
 Let $\mathcal{R}^\prime$ denote the class of renormalizations $\Renorm \in C^{2}(\R_{+})$, such that $\Renorm(0)=0$ and 
\begin{equation}
\sup_{x \in \R_{+}}\left( |\Renorm(x)|+(1+x)|\Renorm'(x)|+|\Renorm''(x)| \right )<\infty.
\end{equation}
\end{definition}

\begin{lem}\label{lem:renorm-collision-tightness-bounds}
For each $\Gamma \in \mathcal{R}^\prime$, the sequences $\{\Renorm'(\tilde{f}_{n})\mathcal{B}^{-}_{n}(\tilde{f}_{n},\tilde{f}_{n})\}_{n \in \N}$ and $\{\Renorm'(\tilde{f}_{n})\mathcal{B}^{+}_{n}(\tilde{f}_{n},\tilde{f}_{n})\}_{n \in \N}$ are uniformly bounded in $L^{\infty-}(\tilde{\Omega};\LLLs{1})$.
\end{lem}
\begin{proof}
Let us begin with an estimate for $\{\Renorm'(\tilde{f}_{n})\mathcal{B}^{-}_{n}(\tilde{f}_{n},\tilde{f}_{n})\}_{n \in \N}$.  Since $\Gamma \in \mathcal{R}^\prime$, the mapping \\ $x \to (1+x)|\Gamma'(x)|$ is bounded on $\R_{+}$  Therefore, the following inequalities hold on $\tilde{\Omega} \times [0,T] \times \R^{2d}$
\begin{equation}
\Renorm'(\tilde{f}_{n})\BCol_{n}^{-}(\tilde{f}_{n},\tilde{f}_{n}) \leqs \frac{\BCol_{n}^{-}(\tilde{f}_{n},\tilde{f}_{n})}{1+f_{n}} \leqs \tilde{f}_{n}* \overline{b}_{n},
\end{equation}
where the convolution is only in the variable $v$. Recall, by Hypothesis \ref{hyp:regularized-collision}, the sequence $\{\overline{b}_{n}\}_{n \in \N}$ is uniformly bounded in $L^{1}(\R^{d}_{v})$. Integrating over $\tilde{\Omega} \times [0,T] \times \R^{2d}$ and applying Young's inequality for convolutions yields for each $p\in [1,\infty)$
\begin{equation}\label{eq:NegPartBounds}
\sup_{n \in \N} \tilde{\E} \|\Renorm'(\tilde{f}_{n})\mathcal{B}^{-}_{n}(\tilde{f}_{n},\tilde{f}_{n})\|^p_{\LLLs{1}}\leqs \sup_{n \in \N}\tilde{\E} \|\tilde{f}_{n}\|^p_{\LLLs{1}}.
\end{equation}
Now we can estimate $\{\Renorm'(\tilde{f}_{n})\mathcal{B}^{+}_{n}(\tilde{f}_{n},\tilde{f}_{n})\}_{n \in \N}$ by applying the bound \eqref{eq:AckerydBd} pointwise in $\tilde{\Omega} \times [0,T] \times \R^{2d}$ (to the truncated collision operator $\BCol_n(\tilde{f}_n,\tilde{f}_n)$ instead of $\BCol(f,f)$), then integrating in all variables to find
\begin{equation}\label{eq:PosPartBounds}
\begin{aligned}
\sup_{n \in \N} \tilde{\E} \|\Renorm'(\tilde{f}_{n})\mathcal{B}^{+}_{n}(\tilde{f}_{n},\tilde{f}_{n})\|^p_{\LLLs{1}}
&\leqs \sup_{n \in \N} \tilde{\E} \|\Renorm'(\tilde{f}_{n})\mathcal{B}^{-}_{n}(\tilde{f}_{n},\tilde{f}_{n})\|^p_{\LLLs{1}} + \sup_{n \in \N} \tilde{\E}\|\mathcal{D}_{n}(\tilde{f}_{n}) \|^p_{L^{1}_{t,x}}\\
&\leqs \sup_{n \in \N}\tilde{\E} \|\tilde{f}_{n}\|^p_{\LLLs{1}} + \sup_{n \in \N} \tilde{\E}\|\mathcal{D}_{n}(\tilde{f}_{n}) \|^p_{L^1_{t,x}},
\end{aligned}
\end{equation}
where we used \eqref{eq:NegPartBounds} in the last line. In view of inequalities \eqref{eq:NegPartBounds} and \eqref{eq:PosPartBounds}, the Proposition now follows from the uniform bounds \eqref{eq:AveragedBoundsOnApprox}.
\end{proof}

\begin{lem}\label{lem:tightness:renorm-neg}
For each $\Renorm\in \mathcal{R}^\prime$, the sequence $\{\Renorm'(\tilde{f}_{n})\mathcal{B}^{-}_{n}(\tilde{f}_{n},\tilde{f}_{n})\}_{n \in \N}$ induces tight laws on $[\LLLs{1}]_{w}$.
\end{lem}
\begin{proof}
Effectively, we have to show that the renormalized collision sequence is bounded, uniformly integrable, and tight in $\LLLs{1}$, with uniformly high probability.  Towards this end, let $\Psi(t)=t|\log t|$. By well-known arguments (see Section 3 in \cite{diperna1989cauchy}), there exists a constant $C$ depending only on $\Gamma$ and $\|\overline{b}\|_{L^1_{v}}$ such that the following two inequalities hold.  Regarding uniform integrability,
\begin{equation} \label{eq:LogLBound}
\int_{0}^{T}\iint_{\R^{2d}}\Psi \left( \Renorm'(\tilde{f}_{n})\mathcal{B}^{-}_{n}(\tilde{f}_{n},\tilde{f}_{n}) \right )\dx \dv \ds \leq C \Big[ \|\tilde{f}_n\|_{\LLLs{1}}+ \int_{0}^{T}\iint_{\R^{2d}}\Psi(\tilde{f}_{n})\dx \dv \ds \Big].
\end{equation}
Moreover, regarding tightness (in $\LLLs{1}$), for all $R>0$
\begin{equation} \label{eq:tightnessEst}
\begin{split}
&\int_{0}^{T}\iint_{\R^{2d}}1_{\{|x|+|v|>R \}}\Renorm'(\tilde{f}_{n})\mathcal{B}^{-}_{n}(\tilde{f}_{n},\tilde{f}_{n}) \dx \dv \ds \\
&\hspace{.5in}\leq C \Big [\|\tilde{f}_n\|_{L^1_{t,x,v}}\int_{\R^{d}}1_{\{|v|>\frac{R}{2}\} }\overline{b}_n(v)\dv + R^{-2}\int_{0}^{T}\iint_{\R^{2d}}(|x|^{2}+|v|^{2})\tilde{f}_{n} \dx \dv \ds \Big].
\end{split}
\end{equation}
Define the function $\lambda: \R_+ \to \R_+$ by
\begin{equation}
  \lambda(R) = \max\left\{\sup_n\int_{\R^d}\1_{|v|>\frac{R}{2}}\bar{b}_n(v)\dv\,,\,\, R^{-2}\right\},
\end{equation}
and note that, by Hypothesis \ref{hyp:regularized-collision}, $\lambda(R) \to 0$ as $R \to \infty$. Combining \eqref{eq:LogLBound} and \eqref{eq:tightnessEst} with the uniform bounds on $\tilde{f}_n$,
\begin{align}
&\label{eq:RenormBmin-unif-int}\sup_{n \in \N} \tilde{\E} \Big ( \big \| \psi \left( \Renorm'(\tilde{f}_{n})\mathcal{B}^{-}_{n}(\tilde{f}_{n},\tilde{f}_{n}) \right ) \big \|_{\LLLs{1}} \Big )<\infty. \\
&\label{eq:RenormBmin-tight} \sup_{n \in \N} \tilde{\E} \bigg ( \sup_{R>0} \big [ \lambda(R)^{-1} \big \|1_{\{|x|+|v|>R \}}\Renorm'(\tilde{f}_{n})\mathcal{B}^{-}_{n}(\tilde{f}_{n},\tilde{f}_{n}) \big \|_{\LLLs{1}} \big ] \bigg ) < \infty.
\end{align}
To construct our compact sets, note that for all $M>0$, the set
\begin{equation}
\Big \{f \in \LLLs{1} \mid \|f\|_{\LLLs{1}}+\|\psi(f)\|_{\LLLs{1}}+\sup_{R>0}\big [ \lambda(R)^{-1} \|1_{\{|x|+|v|>R \}}f\|_{\LLLs{1}} \big ] \leq M \Big\}
\end{equation}
is weakly compact in $\LLLs{1}$.  Indeed, every sequence in this set is bounded, uniformly integrable, and tight in $\LLLs{1}$.  By Chebyshev, the uniform bounds (\ref{eq:RenormBmin-unif-int}), (\ref{eq:RenormBmin-tight}) and our previous Lemma \ref{lem:tightness:renorm-neg}, it follows that $\{\Renorm'(\tilde{f}_{n})\mathcal{B}^{-}_{n}(\tilde{f}_{n},\tilde{f}_{n})\}_{n \in \N}$ induces tight laws on $[\LLLs{1}]_{\mathrm{w}}$.
\end{proof}

\begin{lem}\label{lem:tightness:renorm-pos}
For each $\Renorm$, the sequence $\{\Renorm'(\tilde{f}_{n})\mathcal{B}^{+}_{n}(\tilde{f}_{n},\tilde{f}_{n})\}_{n \in \N}$ induces tight laws on $[\LLLs{1}]_{w}$.
\end{lem}
\begin{proof} The main ingredient in the proof is a version of inequality (\ref{eq:AckerydBd}), which we state again in the precise form required. Specifically, for each $j > 1$ the following inequality holds pointwise a.e in $\Omega\times[0,T]\times\R^{2d}$,
  \begin{equation}\label{eq:B_+-B_--dis-bound}
    \Renorm'(\tilde{f}_{n})\mathcal{B}^{+}_{n}(\tilde{f}_{n},\tilde{f}_{n}) \leq j \Renorm'(f_{n})\mathcal{B}^{-}_{n}(\tilde{f}_{n},\tilde{f}_{n}) + \frac{1}{\log{j}}\SDis_n^0(\tilde{f}_n),
  \end{equation}
  where we recall that
  \begin{equation}
    \SDis_n^0(\tilde{f}_n) = \frac{1}{ 1+ n^{-1}\int_{\Rd} \tilde{f_n}\dv}\int_{\Rd} d_n(\tilde{f}_n)\dee v_*.
  \end{equation}
Let $\epsilon>0$.  By Lemma \ref{lem:tightness:renorm-neg}, we may select a weakly compact set $K_{\epsilon}^{-}$ in $\LLLs{1}$ such that
\begin{equation}
\sup_{n \in \N}\tilde{\p}\big ( \Renorm'(f_{n})\mathcal{B}^{-}_{n}(\tilde{f}_{n},\tilde{f}_{n}) \notin K_{\epsilon}^{-} \big ) \leq \frac{\epsilon}{2}.
\end{equation}
Moreover, in view of the uniform bound on the entropy dissipation (\ref{eq:AveragedBoundsOnApprox}), we can select a closed ball, $B_{M_\ep}$ of size $M
_\ep >0$ in $\LLLs{1}$ such that
\begin{equation}
\sup_{n \in \N} \tilde{\p}\big ( \mathcal{D}^0_{n}(\tilde{f}_{n}) \notin B_{M_\ep} \big ) \leq \frac{\epsilon}{2}.
\end{equation}
For each $j \in \N$, we define a set
\begin{equation}
K_{j,\epsilon}=\big \{f \in \LLLs{1} \mid \text{There exists }g \in K_{\epsilon}^{-}\text{ and }h \in B_{M_\ep} \text{ such that } f \leq j g+(\log j)^{-1} h  \big \}.
\end{equation}
The inequality in the definition of $K_{j,\ep}$ is understood to hold a.e. on $[0,T] \times \R^{2d}$.  
Next define the set $K_{\epsilon}$ via
\begin{equation}
K_{\epsilon} = \bigcap_{j \in \N} K_{j,\epsilon}.
\end{equation}
Note that if $\Gamma^\prime(\tilde{f}_n)\BCol_n^-(\tilde{f}_n,\tilde{f}_n) \in K_\ep^-$ and $\mathcal{D}^0_n(\tilde{f}_n) \in B_{M_\ep}$, then inequality (\ref{eq:B_+-B_--dis-bound}) implies that $\Gamma^\prime(\tilde{f}_n)\BCol_n^+(\tilde{f}_n,\tilde{f}_n)\in K_\ep$. It follows, by the contrapositive, that
\begin{equation}
\sup_{n \in \N}\tilde{\p} \big (\Renorm'(\tilde{f}_{n})\mathcal{B}^{+}_{n}(\tilde{f}_{n},\tilde{f}_{n}) \notin K_{\epsilon} \big ) \leq \sup_{n \in \N} \tilde{\p}\big (\Renorm'(\tilde{f}_{n})\mathcal{B}^{-}_{n}(\tilde{f}_{n},\tilde{f}_{n}) \notin K_{\epsilon}^{-} \big ) + \sup_{n \in \N}\tilde{\p}\big (\mathcal{D}^0_{n}(\tilde{f}_{n},\tilde{f}_{n}) \notin B_{M} \big ).
\end{equation}  
Since each term above is of order $\epsilon$, the proof of the Lemma will be complete if we verify that $K_{\epsilon}$ is a weakly compact subset of $\LLLs{1}$.  By classical compactness criteria, it suffices to verify the following:
\begin{align}
& \lim_{R \to \infty} \sup_{f \in K_{\epsilon}} \int_{0}^{T}\iint_{\R^{2d}} 1_{ \{|x|+|v|>R\}  } |f| \dx \dv \dt = 0. \label{eq:to-check-tightness}\\
&\lim_{\delta \to 0} \sup_{f \in K_{\epsilon}} \sup_{|E| \leq \delta} \int_{0}^{T}\iint_{\R^{2d}} 1_{E}|f| \dx \dv \dt = 0 \label{eq:to-check-unif-int},
\end{align}
where in (\ref{eq:to-check-unif-int}) the supremum is taken over all measurable $E \subseteq [0,T]\times\R^{2d}$ with Lebesgue measure $|E|<\delta$. To verify \eqref{eq:to-check-tightness} and  \eqref{eq:to-check-unif-int}, note that for all $j > 1$, by construction of $K_{\ep}$
\begin{equation}
\sup_{f \in K_{\epsilon}}\int_{0}^{T}\iint_{\R^{2d}} 1_{ \{|x|+|v|>R\}  } |f| \dx \dv \dt \leq j \sup_{g \in K_{\epsilon}^{-}}\int_{0}^{T}\iint_{\R^{2d}} 1_{ \{|x|+|v|>R\}  } |g| \dx \dv \dt+\frac{M_\ep}{\log j},
\end{equation}
and
  \begin{equation}
\sup_{f \in K_{\epsilon}}\sup_{|E|<\delta}\int_{0}^{T}\iint_{\R^{2d}} 1_{E} |f| \dx \dv \dt \leq j \sup_{g \in K_{\epsilon}^{-}}\sup_{|E|<\delta}\int_{0}^{T}\iint_{\R^{2d}} 1_{E} |g| \dx \dv \dt+\frac{M_\ep}{\log j}. 
\end{equation}
First taking $R\to \infty$ and using the $\LLLs{1}$ weak compactness of $K_{\epsilon}^{-}$ and then sending $j\to \infty$  yields \eqref{eq:to-check-tightness} and (\ref{eq:to-check-unif-int}).
\end{proof}

\begin{lem}\label{lem:renorm-approx-tightness}
For each $\Gamma \in \mathcal{R}^\prime$, the laws of $\{\Renorm(\tilde{f}_{n})\}_{n \in \N}$ are tight on $\CLLw{1} \cap \LLsMw{1}$.  
\end{lem}

\begin{proof}
  We will check that $\{\Renorm(\tilde{f}_{n})\}_{n \in \N}$ induces tight laws on the space $\LLsMw{1}$ by first verifying the requirements of the $L^{1}$ velocity averaging Lemma \ref{lem:L1-Velocity-Averaging} and deducing that for each $\varphi\in C^\infty_c(\Rd)$,  $\{\langle\Gamma(\tilde{f}_n), \varphi\rangle\}_{n\in\N}$ induces tight laws on $L^1_{t,x}$ and then applying Lemma \ref{lem:tightness-crit-vel-averaged-space} to conclude that $\{\Gamma(\tilde{f}_n)\}_{n\in\N}$ induces tight law on $\LLsMw{1}$.

Observe that for each $n\in \N$, $\Gamma(\tilde{f}_{n})$ is a weak martingale solution to the stochastic kinetic equation driven by $\Renorm'(\tilde{f}_{n})\mathcal{B}_{n}(\tilde{f}_{n},\tilde{f}_{n})$, starting from $\Renorm(\tilde{f}_{n}^{0})$, with noise coefficients $\sigma^{n}$. By Proposition \ref{prop:exist-approx-scheme} on the approximating scheme, and the fact that $\Gamma(z) \leqc |z|$, we can easily conclude that $\Gamma(\tilde{f}_{n})$ and $\Renorm'(\tilde{f}_{n})\mathcal{B}_{n}(\tilde{f}_{n},\tilde{f}_{n})$ belong to $L^{\infty-}(\Omega; \LLLs{1}\cap\LLLs{\infty})$ and $\Gamma(f^0_n)$ is in $\LLs{1}\cap\LLs{\infty}$. Also, by assumption, $\{\sigma^n\}_{n\in\N}$ satisfy Hypothesis \ref{hyp:Noise-Coefficients} uniformly. 

Next, since $|\Gamma(z)|\leqc |z|$, and $\{f_n^0\}_{n\in\N}$ is uniformly integrable, then $\{\Gamma(f_n^0)\}_{n\in\N}$ is uniformly integrable. Similarly, the uniform estimates (\ref{eq:AveragedBoundsOnApprox}) imply that for $p \in [1,\infty)$
\begin{equation}\label{eq:Weighted-Bounds-On-Renorm}
\sup_{n \in \N} \tilde{\E} \big\| (1+|x|^2 + |v|^2)\Gamma(\tilde{f}_n)\big \|_{\LLL{\infty}{1}}^{p}<\infty.
\end{equation}
Also, Lemma \ref{lem:renorm-collision-tightness-bounds} implies that $\{\Renorm'(\tilde{f}_{n})\mathcal{B}_{n}(\tilde{f}_{n},\tilde{f}_{n})\}_{n\in\N}$ is uniformly bounded in $L^{\infty-}(\Omega;\LLLs{1})$, while Lemmas \ref{lem:tightness:renorm-neg} and \ref{lem:tightness:renorm-pos} imply that $\{\Renorm'(\tilde{f}_{n})\mathcal{B}_{n}(\tilde{f}_{n},\tilde{f}_{n})\}_{n\in\N}$ also induce tight laws on $[\LLLs{p}]_{\mathrm{w}}$. Finally, we see by Chebyshev that
\begin{equation}
 \tilde{\P}(\|\langle \Gamma(\tilde{f}_n),\varphi\rangle\1_{|x|>R}\|_{L^1_{t,x}} > \eta) \leqc \frac{1}{\eta R^2 }\tilde{\E} \big\| (1+|x|^2 +|v|^2)\Gamma(\tilde{f}_n)\big \|_{\LLL{\infty}{1}},
\end{equation}
and therefore the right-hand side vanishes uniformly in $n$ as $R\to\infty$. Hence, we meet all the requirements of Lemma \ref{lem:L1-Velocity-Averaging} to conclude that 
$\{\langle\Gamma(\tilde{f}_n), \varphi\rangle\}_{n\in\N}$ induces tight laws on $L^1_{t,x}$.

To check that $\{\Renorm(\tilde{f}_{n})\}_{n \in \N}$ induces tight laws on the space $\CLLw{1}$, by Lemma \ref{Lem:Appendix:TightCrit} it suffices to show that for each $\varphi \in C^{\infty}_{c}(\R^{2d})$, the sequence $\{\langle \Renorm(\tilde{f}_{n}),\varphi \rangle\}_{n \in \N}$ induces tight laws on $C[0,T]$ and
\begin{equation}
  \begin{aligned}
    &\sup_n\tilde{\E}\|\Gamma(\tilde{f}_n)\|_{L^\infty_t(L^1_{x,v})} <\infty,\\
    &\lim_{R\to\infty}\sup_n\tilde{\E}\|\Gamma(\tilde{f_n})\1_{|x|^2+|v|^2 > R}\|_{L^\infty_t(L^1_{x,v})} = 0,\\
    &\lim_{L\to\infty}\sup_n\tilde{\E}\|\Gamma(\tilde{f}_n)\1_{|\Gamma(\tilde{f}_n)|>L}\|_{L^\infty_t(L^1_{x,v})} = 0.\\
  \end{aligned}
\end{equation}
The first two follow from (\ref{eq:Weighted-Bounds-On-Renorm}), while the last follows from the fact that $|\Gamma(z)|\leq C|z|$ for some constant $C$, implies that
\begin{equation}
  |\Gamma(\tilde{f}_n)|\1_{|\Gamma(\tilde{f}_n)|>L} \leq |\tilde{f}_n|\1_{|\tilde{f}_n| > L/C}
\end{equation}
and therefore
\begin{equation}
  \lim_{L\to\infty}\sup_n\tilde{\E}\|\Gamma(\tilde{f}_n)\1_{|\Gamma(\tilde{f}_n)|>L}\|_{L^\infty_t(L^1_{x,v})} \leq \lim_{L\to\infty}\frac{1}{\log{L/C}}\sup_n\tilde{\E}\|\tilde{f}_n\log{\tilde{f}_n}\|_{L^\infty_t(L^1_{x,v})} = 0.
\end{equation}
To see this, use the weak form to obtain the decomposition $\langle \Renorm(\tilde{f}_{n}),\varphi \rangle=I^{n,1}+I^{n,2}$, where the continuous processes $(I^{n,1}_{t})_{t=0}^{T}$ and $(I^{n,2}_{t})_{t=0}^{T}$ are defined via:
\begin{align}
&I^{n,1}_{t}=\iint_{\R^{2d}}\Renorm(f_{n}^0)\varphi \dx \dv
+ \int_{0}^{t}\iint_{\R^{2d}}\Renorm(\tilde{f}_{n})[v \cdot \nabla_{x}\varphi+\LStrat_{\sigma^n}\varphi]\dx\dv\ds \\
&\hspace{1in}+\sum_{k=1}^{\infty}\int_{0}^{t}\iint_{\R^{2d}}\Renorm(\tilde{f}_{n})\sigma^n_{k}\cdot \nabla_{v}\varphi\, \dx\dv \dee\beta_{k}(s). \\
&I^{n,2}_{t}=\int_{0}^{t}\iint_{\R^{2d}}\Renorm'(\tilde{f}_{n})\mathcal{B}_{n}(\tilde{f}_{n},\tilde{f}_{n})\varphi \dx\dv\ds.
\end{align}
Arguing as in Lemma \ref{lem:time-regularity}, using the uniform bounds, and a Sobolev embedding, there exists an an $\alpha >0$ and $p>1$ such that $\{I^{n,1}\}_{n \in \N}$ is a bounded sequence in $L^{p}(\tilde{\Omega}; C^{\alpha}_{t})$.  Next observe that by Lemmas \ref{lem:tightness:renorm-neg} and \ref{lem:tightness:renorm-pos}, the sequence 
$\{\partial_{t}I^{n,2}\}_{n \in \N}$ induces tight laws on $L^{1}[0,T]$ endowed with the weak topology.

Let $\epsilon>0$ and let $K_{\epsilon}^{1}$ be the closed ball of radius $\epsilon^{-1}$ in $C^{\alpha}_{t}$. In addition, choose a uniformly integrable subset of $L^{1}[0,T]$, denoted $\hat{K}_{\epsilon}^{2}$, such that
\begin{equation}
\sup_{n \in \N} \tilde{\p}( \partial_{t}I^{2,n} \notin \hat{K}_{\epsilon}^{2})<\epsilon.
\end{equation}
Define $K_{\epsilon}^{2}$ to be the anti-derivatives of $\hat{K}_{\epsilon}^{2}$, that is:
\begin{equation}
K_{\epsilon}^{2}= \big \{ f \in C[0,T] \mid f(0)=0 \text{ and there exists } g \in \hat{K}_{\epsilon}^{2} \text{ such that } \partial_{t}f=g \big \}.
\end{equation}
Finally, let $K_{\epsilon}$ be the algebraic sum (in $C[0,T]$) of $K_{\epsilon}^{1}$ and $K_{\epsilon}^{2}$.  In view of our decomposition, it follows that
\begin{equation}
\sup_{n \in \N} \tilde{\p}( \langle \Renorm(\tilde{f}_{n}),\varphi \rangle \notin K_{\epsilon}) \leq \sup_{n \in \N} \tilde{\p}( I^{n,1}\notin K_{\epsilon}^{1})+ \sup_{n \in \N} \tilde{\p}(I^{n,2}\notin K_{\epsilon}^{2}).
\end{equation}
Each of the probabilities above are of order $\epsilon$. Since, by construction, $K_{\epsilon}^{1}$ and $K_{\epsilon}^{2}$ are compact of $C[0,T]$ (by Arzel\`{a}-Ascoli), it follows that $K_\ep$ is itself compact in $C[0,T]$. This completes the proof.
\end{proof}
\begin{lem}\label{lem:renorm-approx-tightness-f}
The sequence $\{\tilde{f}_{n}\}_{n \in \N}$ induces tight laws on the space $\CLLw{1} \cap \LLsMw{1}$.  
\end{lem}
\begin{proof}
Let us begin by verifying that $\{\tilde{f}_{n}\}_{n \in \N}$ induces a tight sequence of laws on $\LLsMw{1}$.  From the uniform bounds, we know that $\{\tilde{f}_{n}\}_{n \in \N}$ is uniformly bounded in $L^{1}(\tilde{\Omega} \times [0,T] \times \R^{2d})$.  By the appendix Lemma \ref{lem:tightness-crit-vel-averaged-space}, it suffices to check that for each $\varphi \in C^{\infty}_{c}(\R^{d}_{v})$, the sequence $\{\langle \tilde{f}_{n},\varphi \rangle \}_{n \in \N}$ induces a tight sequence of laws on $\LLt{1}$. For this purpose, we will use the compactness criterion given in appendix Lemma \ref{lem:tightness-decomp}, together with Lemma \ref{lem:renorm-approx-tightness}. Indeed, recall the definition of $\Gamma_m(z)$ in equation \eqref{eq:TruncationDef}, then for each $m \in \N$, we have the decomposition
\begin{equation}
\langle \tilde{f}_{n},\varphi \rangle=\langle \Renorm_{m}(\tilde{f}_{n}),\varphi \rangle+\langle f_{n}-\Renorm_{m}(\tilde{f}_{n}),\varphi \rangle.
\end{equation}
By Lemma \ref{lem:renorm-approx-tightness}, the sequence $ \{ \langle \Renorm_{m}(\tilde{f}_{n}),\varphi \rangle \}_{n \in \N}$ induces a tight sequence of laws on $\LLt{1}$.  Hence, by Lemma \ref{lem:tightness-decomp}, it only remains to verify that
\begin{equation}
\lim_{m \to \infty}\sup_{n \in \N}\tilde{\E}\big \| \langle \tilde{f}_{n}-\Renorm_{m}(\tilde{f}_{n}),\varphi \rangle \big \|_{\LLt{1}}=0.
\end{equation}
Towards this end, note the elementary inequality: for all $R>1$ and $z>0$,
\begin{equation}
|\Renorm_{m}(z)-z|\leq \frac{R}{m}z + z1_{z\geq R} \leq \frac{R}{m}z+|\log R|^{-1}z|\log z|.
\end{equation}
Hence, for all $m \in \N$ and $R>1$, we have the inequality
\begin{equation}
\begin{aligned}
&\sup_{n \in \N} \tilde{\E}\big\| \langle \tilde{f}_{n}-\Renorm_{m}(\tilde{f}_{n}),\varphi \rangle \big \|_{\LLt{1}}\leq  \frac{R}{m}\|\varphi\|_{L^\infty_v}\sup_{n \in \N}\tilde{\E}\|\tilde{f}_{n} \|_{\LLLs{1}}\\
&\hspace{1in} + |\log R|^{-1}\|\varphi\|_{L^\infty_v}\sup_{n \in \N}\tilde{\E}\|\tilde{f}_{n}\log{\tilde{f}_{n}} \|_{\LLLs{1}}.
\end{aligned}
\end{equation}
Taking first $m \to \infty$ and then $R \to \infty$ gives the claim.

The next step is to check that the sequence $\{\tilde{f}_{n}\}_{n \in \N}$ induces a tight sequence of laws on $\CLLw{1}$. In view of the uniform bounds (\ref{eq:AveragedBoundsOnApprox}) and tightness criterion on $C_t([L^1_{x,v}]_{\mathrm{w}})$ given in Lemma \ref{Lem:Appendix:TightCrit}, it suffices to verify that for all $\varphi \in C^{\infty}_{c}(\R^{2d})$, the sequence $\{\langle \tilde{f}_{n},\varphi \rangle \}_{n \in \N}$ induces tight laws on the space $C[0,T]$.  Again, for each $m \in \N$ we have the decomposition
\begin{equation}
\langle \tilde{f}_{n},\varphi \rangle=\langle \Renorm_{m}(\tilde{f}_{n}),\varphi \rangle+\langle \tilde{f}_{n}-\Renorm_{m}(\tilde{f}_{n}),\varphi \rangle.
\end{equation}
Moreover, the sequence $\{ \langle \Renorm_{m}(\tilde{f}_{n}),\varphi \rangle \}_{n \in \N}$ induces tight laws on $C[0,T]$ by Lemma \ref{lem:renorm-approx-tightness}. Arguing in a similar way as above, we find that 
\begin{equation}
\lim_{m \to \infty}\sup_{n \in \N}\tilde{\E}\big \| \langle \tilde{f}_{n}-\Renorm_{m}(\tilde{f}_{n}),\varphi \rangle \big \|_{L^{\infty}_{t}}=0.
\end{equation}
Therefore by Lemma \ref{lem:tightness-decomp} $\{\langle \tilde{f}_n,\varphi\rangle\}_{n\in \N} $ is tight on $C[0,T]$.
\end{proof}

\subsection{Proof of Proposition \ref{prop:Grand-Skorohod}}
For each $n \in \N$, introduce random variables $\tilde{X}_{n},\tilde{Y}_{n}, $ and $\tilde{Z}_{n}$ by setting
\begin{align}
&\tilde{X}_{n}=\big ( (1+|x|^{2}+|v|^{2}+|\log \tilde{f}_{n}|)\tilde{f}_{n},\mathcal{D}^0_{n}(\tilde{f}_{n}) \big ) \\
&\tilde{Y}_{n}=\big (\tilde{f}_{n}, \{\tilde{\beta}_{k}\}_{k \in \N} \big ) \\
&\tilde{Z}_{n}= \Big\{ \big ( \Renorm_{m}(\tilde{f}_{n}),\gamma_{m}(\tilde{f}_{n}),\Renorm_{m}'(\tilde{f}_{n})\mathcal{B}^{-}_{n}(\tilde{f}_{n},\tilde{f}_{n}),\Renorm_{m}'(\tilde{f}_{n})\mathcal{B}^{+}_{n}(\tilde{f}_{n},\tilde{f}_{n}) \big ) \Big \}_{m \in \N} 
\end{align}
The random variables $\tilde{X}_{n},\tilde{Y}_{n}, $ and $\tilde{Z}_{n}$ induce laws defined on the spaces $E,F, $ and $G$ respectively, where 
\begin{align}
&E=[L^{1}_{t}(C_{0}(\R^{2d}))]'_{*}\times \MMMsw{\infty}\\
&F=\LLsMw{1} \cap \CLLw{1} \times [C_{t}]^{\infty} \\
&G=\big [ [\LLsMw{1}\cap\CLLw{1}]^{2} \times [\LLLs{1}]_{w}^{2} \big ]^{\infty}.
\end{align}
To be clear, we use $[L^{1}_{t}(C_{0}(\R^{2d}))]'_{*}$ to denote the dual of $L^{1}_{t}(C_{0}(\R^{2d}))$ endowed with the weak star topology.  

Our first observation is that the sequence $\{\tilde{X}_{n} \}_{n \in \N}$ induces tight laws on $E$.  For this, we use the fact that $\LLL{\infty}{1}$ embeds isometrically into the space $L^{\infty}_{t}(\mathcal{M}_{x,v})$, which in turn embeds isometrically into $[L^{1}_{t}(C_{0}(\R^{2d}))]'$ by classical duality results on Lebesgue-Bochner spaces. Also, $\LLLs{1}$ embeds isometrically into $\mathcal{M}_{t,x,v}$.   Since bounded sets in $\LLL{\infty}{1} \times \LLLs{1}$ are compact in E, the uniform bounds \eqref{eq:AveragedBoundsOnApprox} and Banach Alaoglu yield the tightness claim.  

Next we observe that $\{\tilde{Y}_{n}\}_{n \in \N}$ induces tight laws on $F$.
This follows from Lemma \ref{lem:renorm-approx-tightness-f} and classical facts about Brownian motions.  Finally, by Lemmas \ref{lem:tightness:renorm-neg}, \ref{lem:tightness:renorm-pos}, and \ref{lem:renorm-approx-tightness} it follows that the sequence $\{\tilde{Z}_{n}\}_{n \in \N}$ induces tight laws on $G$.  Combining these observations, we find that the sequence  $ \{ (\tilde{X}_{n},\tilde{Y}_{n},\tilde{Z}_{n}) \}_{n \in \N}$ induces tight laws on $E \times F \times G$.

Apply the Jakubowski/Skorohod Theorem  \ref{Thm:Appendix:Jakubowksi_Skorohod} (working on a subsequence if necessary) to obtain a new probability space $(\Omega, \mathcal{F},\p)$, random variables $(X,Y,Z)$ on $E \times F \times G$, and a sequence of maps $\{\tilde{T}_{n}\}$ satisfying Part 1 of Proposition \ref{prop:Grand-Skorohod}.  First observe that the uniform bounds and the explicit representation guarantees that $X_{n}(\omega) \in \LLL{\infty}{1} \times \LLLs{1}$ for almost all $\omega \in \Omega$ and $n \in \N$.  Thus, Part 1 now yields that $\{f_{n}\}_{n \in \N}$ satisfies the uniform bounds \eqref{eq:AveragedBoundsOnApprox} with $\E$ in place of $\tilde{\E}$.  This gives the first claim in Part 2 of Proposition \ref{prop:Grand-Skorohod}.   Theorem \ref{Thm:Appendix:Jakubowksi_Skorohod} also guarantees that the sequence $\{X_{n}\}_{n \in \N}$ defined by $X_{n}=\tilde{X}_{n} \circ \tilde{T}_{n}$ converges pointwise on $\Omega$ to $X$ in the space $E$.  In particular, there exists a random constant $C(\omega)$ such that
\begin{equation}
\begin{split}
&\sup_{n \in \N}\|(1+|x|^{2}+|v|^{2}+|\log f_{n}(\omega)|)f_{n}(\omega)\|_{[L^{1}_{t}(C_{0}(\R^{2d}))]'}\leq C(\omega). \\
&\sup_{n \in \N}\|\mathcal{D}^0_{n}(f_{n})(\omega)\|_{\MMMs{\infty}}\leq C(\omega).
\end{split}
\end{equation}
Using again the isometric embedding of $\LLL{\infty}{1}$ into $[L^{1}_{t}(C_{0}(\R^{2d}))]'$ and $\LLLs{1}$ into $\MMMs{\infty}$, together with the fact that $X_{n}(\omega) \in \LLL{\infty}{1} \times \LLLs{1}$, this completes the proof of Part 2.  To obtain the remaining parts of Proposition \ref{prop:Grand-Skorohod}, let $\overline{D}$ be the second component of $X$, and denote
\begin{align}
&Y=\big (f, \{\beta_{k}\}_{k \in \N} \big ). \\
&Z= \Big\{ \big ( \overline{{\Renorm}_{m}(f)},\overline{\gamma_{m}(f)},\mathcal{B}_{m}^{-},\mathcal{B}_{m}^{+} \big) \Big \}_{m \in \N}. 
\end{align}
Part 3 follows easily from Part 1 and the martingale representation theorem.  Part 4 follows from the pointwise convergence of $\{Y_{n}\}_{n \in \N}$ towards $Y$ in the space $F$.  Part 5 follows from the pointwise convergence of $\{Z_{n}\}_{n \in \N}$ towards $Z$ and $\{X_n\}_{n\in\N}$ to $X$.  This completes the proof of Proposition \ref{prop:Grand-Skorohod}.   


\subsection{Preliminary identification}\label{sec:Preliminary-identification}
As our first application of Proposition \ref{prop:Grand-Skorohod}, we send $n \to \infty$, but the limit passage is in a preliminary sense.  Namely, we do not yet obtain the renormalized form for $f$, but we obtain a stochastic kinetic equation for a strong approximation $\overline{\Renorm_{m}(f)}$. In fact, using Proposition \ref{prop:Grand-Skorohod}, we will prove:  
\begin{cor}\label{cor:PrelimMartingale}
For all $m \in \N$, the process $\overline{\Renorm_{m}(f)}$ is a renormalized weak martingale solution to the stochastic kinetic equation driven by $\mathcal{B}_{m}^{+}-\mathcal{B}_{m}^{-}$, starting from $\Renorm_{m}(f_{0})$, with noise coefficients $\sigma = \{\sigma_k\}_{k\in\N}$. Moreover, $\P$ almost surely,  $\overline{\Renorm_m(f)}$ belongs to $L^\infty_{t,x,v}$ and has strongly continuous sample paths in $C_t(L^{1}_{x,v})$.
\end{cor}
\begin{proof}
  Fix an $m \in \N$.  First, using the uniform bounds and the convergence results obtained in Proposition \ref{prop:Grand-Skorohod}, we verify the hypotheses of the stability result for martingale solutions of stochastic kinetic equations, Proposition \ref{prop:stability-weak-martingale}.  Namely, we will analyze the sequence $\{\Gamma_{m}(f_{n})\}_{n \in \N}$.  Once we verify Parts $1-3$ of Proposition \ref{prop:stability-weak-martingale}, we may conclude that  the process $\overline{\Renorm_{m}(f)}$ is a weak martingale solution to the stochastic kinetic equation driven by $\mathcal{B}_{m}^{+}-\mathcal{B}_{m}^{-}$, starting from $\Renorm_{m}(f_{0})$, with noise coefficients $\sigma = \{\sigma_k\}_{k\in\N}$. The next step will be to show that the solution is actually a renormalized weak martingale solution, applying the renormalization Proposition \ref{prop:Weak_Is_Renormalized}. Finally we will show strong continuity by applying Lemma \ref{lem:strong-cont} on our renormalized weak martingale solution.

To verify Part 1 of Proposition \ref{prop:stability-weak-martingale}, let us first check that the process $\Renorm_{m}(f_{n})$ is a weak martingale solution to the stochastic kinetic equation driven by $\Gamma_{m}'(f_{n})\mathcal{B}_{n}(f_{n},f_{n})$, starting from $\Renorm_{m}(f_{n}^{0})$, relative to the noise coefficients $\sigma^{n}$ and the Brownian motions $\{\beta_{k}^{n}\}_{k \in \N}$ obtained in \ref{prop:Grand-Skorohod}.  Indeed, $\{\tilde{f}_{n}\}_{n \in \N}$ is a renormalized weak martingale solution to the stochastic kinetic equation driven by $\mathcal{B}_{n}(\tilde{f}_{n},\tilde{f}_{n})$, starting from $f_{n}^{0}$, relative to the noise coefficients $\sigma^{n}$ and the Brownian motions $\{\tilde{\beta}_{k}^{n}\}_{k \in \N}$.  The claim can now be checked by using the explicit expression for $\{f_{n}\}_{n \in \N}$ and $\{\beta_{k}^{n}\}_{k \in \N}$ in terms of the maps $\{\tilde{T}_{n}\}_{n \in \N}$ together with the fact that $\Renorm_{m} \in \mathcal{R}^\prime$.

To verify Part 2 of Proposition \ref{prop:stability-weak-martingale}, from the uniform bounds in Proposition \ref{prop:Grand-Skorohod} and the fact that $\Renorm_{m}(z) \leq z$, it follows that the sequence $\{\Gamma_{m}(f_{n})\}_{n \in \N}$ is uniformly bounded in $L^{2}(\Omega; \LLL{\infty}{1})$.  Also, Lemma \ref{lem:renorm-collision-tightness-bounds} and Part 1 of Proposition \ref{prop:Grand-Skorohod} imply that $\{\Renorm'_{m}(f_{n})\mathcal{B}^{-}_{n}(f_{n},f_{n})\}_{n \in \N}$ and $\{\Renorm'_{m}(f_{n})\mathcal{B}^{+}_{n}(f_{n},f_{n})\}_{n \in \N}$ are uniformly bounded in $L^{2}(\Omega;\LLLs{1})$.  Combining this with the pointwise  convergences from Part 5 of Proposition \ref{prop:Grand-Skorohod}, we easily verify \eqref{eq:conv-at-fixed-time} and \eqref{eq:conv-time-ints}.

Finally Part 3 of Proposition \ref{prop:stability-weak-martingale} follows from the convergences from Part 4 of Proposition \ref{prop:Grand-Skorohod} together with Hypotheses \ref{hyp:regularized-noise-coeff} and \ref{hyp:regularized-idata} regarding the sequences $\{\sigma^{n}\}_{n \in \N}$ and $\{f_{n}^{0}\}_{n \in \N}$.

Next we argue that $\overline{\Gamma_m(f)}$ is actually a {\it renormalized} weak martingale solution. Indeed, by the conditions on the noise coefficients $\sigma$ in Hypotheses (\ref{eq:Noise-Assumption-3}) and (\ref{eq:Noise-Assumption-4}) this will follow from Proposition \ref{prop:Weak_Is_Renormalized} as soon as $\overline{\Gamma_m(f)} \in L^{\infty-}(\Omega\times[0,T]\times\R^{2d})$. To argue this, we note that since $\Gamma_m(z) \leq m$ and $\Gamma_m(z) \leq z$, this gives the following uniform bounds in $L^\infty_{\omega,t,x,v}$ and $L^1_{\omega,t,x,v}$,
\begin{equation}
  \begin{aligned}
    &\sup_n\|\Gamma_m(f_n)\|_{L^\infty(\Omega\times[0,T]\times\R^{2d})} < m <\infty\\
    &\sup_n\|\Gamma_m(f_n)\|_{L^1(\Omega\times[0,T]\times\R^{2d})} \leq T\sup_n\E\|f_n\|_{L^\infty_t(L^1_{x,v})} <\infty.
    \end{aligned}
  \end{equation}
Therefore, by interpolation, $\{\Gamma_m(f_n)\}_{n\in\N} \in L^p(\Omega\times[0,T]\times\R^{2d})$ uniformly in $n$ for each $p \in [1,\infty]$ and $m\geq 1$. Using the weak sequential compactness of $L^p(\Omega\times[0,T]\times\R^{2d})$ for $p\in(1,\infty)$, weak-* sequential compactness of $L^\infty(\Omega\times[0,T]\times\R^{2d})$, and the fact that by Proposition \ref{prop:Grand-Skorohod}, $\P$ almost surely, $\Gamma_m(f_n) \to \overline{\Gamma_m(f)}$ in $C_t([L^1_{x,v}]_w)$, we can conclude that the limit $\overline{\Gamma_m(f)}$ must belong to $L^{p}(\Omega\times[0,T]\times\R^{2d})$ for every $p \in [1,\infty]$.

Finally we show that process $t \mapsto \overline{\Gamma_m(f_t)}$ has continuous sample paths in $\LLs{1}$ with the strong topology. Observe that any sequence converging strongly in $\LLs{2}$ and weakly in $\LLs{1}$ also converges strongly in $\LLs{1}$. Therefore, since $\overline{\Gamma_m(f)}\in \CLLw{1}$ with probability one, it suffices to show that $\overline{\Gamma_m(f)}\in \CLL{2}$ with probability one. However, since $\overline{\Renorm_m(f)}$ is a renormalized weak martingale solution, by Lemma \ref{lem:strong-cont} it is sufficient to show that $\overline{\Renorm_m(f)}$ belongs to $L^\infty_t(L^2_{x,v})$ $\P$ almost surely. Since Proposition \ref{prop:Grand-Skorohod} implies that $\overline{\Gamma_m(f)}$ also belongs to $L^\infty_t(L^1_{x,v})$ $\P$ almost surely and $\overline{\Gamma_m(f)}$ belongs to $L^{\infty}(\Omega\times[0,T]\times\R^{2d})$, we can conclude, again by interpolation, that $\overline{\Gamma_m(f)}$ belongs to $L^\infty_t(L^2_{x,v})$ $\P$ almost surely. 
\end{proof}

In fact, this preliminary identification of $\overline{\Gamma_m(f)}$ allows us to upgrade the continuity properties on $f$ from weakly continuous to strongly continuous. This is the content of the following corollary.

\begin{cor} \label{lem:Strong-Renorm-Limit}
The sample paths of $f$ belong $\P$ almost surely to $C_t(L^1_{x,v})$. Moreover as $m \to \infty$, the sequence $\{ \overline{\Renorm_{m}(f)} \}_{m \in \N}$ converges $\p$ a.s. to $f$ in $\CLL{1}$.
\end{cor}
\begin{proof} 
Recall, by Corollary \ref{cor:PrelimMartingale}, $\overline{\Gamma_m(f)}$ belongs to $C_t(L^1_{x,v})$, hence it suffices to show that $\{ \overline{\Renorm_{m}(f)} \}_{m \in \N}$ converges $\p$ a.s. to $f$ in $\LLL{\infty}{1}$. This is accomplished by applying Proposition \ref{prop:Grand-Skorohod} to conclude that for each $t\in [0,T]$,  $f_n(t) - \Renorm_m(f_n(t)) \to f_t - \overline{\Renorm_m(f)_t}$ weakly in $\LLs{1}$, $\P$ almost-surely, then using weak lower semi-continuity of the $\LLs{1}$ norm to obtain the $\P$ almost-sure inequality
\begin{align}
\sup_{t \in [0,T]}\|f_{t}-\overline{\Renorm_{m}(f)}_{t} \|_{\LLs{1}} &\leq \liminf_{n \to \infty} \sup_{t \in [0,T]}\|f_{n}(t)-\Renorm_{m}(f_{n})(t)\|_{\LLs{1}} \\
& \leq \frac{1}{\sqrt{m}}\sup_n\|f_n\|_{\LLL{\infty}{1}} + \sup_n\|f_{n}\1_{f_{n} \geq \sqrt{m}} \|_{\LLL{\infty}{1}},
\end{align}
where in the last inequality we used the fact that
\begin{equation}
  |x - \Gamma_m(x)| \leq \frac{1}{\sqrt{m}}x + x \1_{|x| \geq \sqrt{m}}. 
\end{equation}
In view of Part 2 in Proposition \ref{prop:Grand-Skorohod}, for $\p$ almost all $\omega \in \Omega$, the sequence $\{f_{n}(\omega)\}_{n \in \N}$ is uniformly integrable in $\LLL{\infty}{1}$. Taking $m \to \infty$ on both sides of the inequality above completes the proof.
\end{proof}



\section{Analysis of the Renormalized Collision Operator}\label{sec:Renorm-Collision-Analysis}
In this section, we prepare for the passage of $m \to \infty$. By applying the renormalization lemma for martingale solutions of stochastic kinetic equations, we obtain the following immediate corollary.
\begin{cor}
For all $m \in \N$, the process $\log (1+ \overline{\Renorm_{m}(f)} )$ is a weak martingale solution to the stochastic kinetic transport equation driven by $(1+\overline{\Renorm_{m}(f)})^{-1}[\mathcal{B}_{m}^{+}-\mathcal{B}_{m}^{-}]$, starting from $\log ( 1+ \Renorm_{m}(f_{0}) )$.
\end{cor}

Our primary focus is to analyze the limiting behavior of the sequence $\{ \mathcal{B}_{m}^{+}\}_{m \in \N}$. The main source of difficulty here is that this sequence is not bounded in $L^{1}(\Omega \times [0,T] \times \R^{2d})$. This is natural in the sense that we expect $\mathcal{B}_{m}^{+}$ to be close to $\mathcal{B}^{+}(f,f)$ as we relax the truncation parameter $m \in \N$. In fact, we know that the main strategy in dealing with $\mathcal{B}^{+}(f,f)$ is to renormalize with $\Gamma^\prime(f)\BCol^+(f,f)$ before we can hope for an estimate it in $L^{1}(\Omega \times [0,T] \times \R^{2d})$. 
The main result of this section is the following:
\begin{prop}\label{prop:RenormCollision} For any $\phi\in \LLLs{\infty}$ as $m \to \infty$, the following convergences hold:
\begin{align}
\bigg\{\bigg\langle\frac{\mathcal{B}_{m}^{-}}{1+\overline{\Renorm_{m}(f)}},\phi \bigg\rangle \bigg \}_{m \in \N} \to  \, \, \bigg\langle\frac{\mathcal{B}^{-}(f,f)}{1+f}, \phi\bigg \rangle \quad &\text{in} \quad L^2(\Omega), \\
\bigg \{\bigg\langle\frac{\mathcal{B}_{m}^{+}}{1+\overline{\Renorm_{m}(f)}}, \phi\bigg\rangle \bigg \}_{m \in \N} \to \, \, \bigg\langle\frac{\mathcal{B}^{+}(f,f)}{1+f}, \phi \bigg\rangle \quad &\text{in} \quad L^2(\Omega).
\end{align}
\end{prop}
The most challenging part of the analysis is analyzing the positive part of the collision operator.  To analyze the $m \to \infty$ limit, we must analyze the consequences of the pointwise (in $\omega$) convergence of $f_{n}(\omega)$ towards $f(\omega)$ in the space $\LLsMw{1}$. In fact, this has not been used so far in the proof.
\begin{lem} \label{Lem:AveragedPositivePart}
As $n \to \infty$, the following convergence holds $\p$ almost surely:
\begin{equation}
  \frac{\BCol^{+}_n(f_{n},f_{n})}{1+\langle f_{n}\rangle} \to \frac{\BCol^{+}(f,f)}{1+\langle f\rangle} \quad \text{in} \quad L^{1}_{t,x}(\mathcal{M}_{v}^{*}).
\end{equation}
\end{lem}
\begin{proof}
The proof follows essentially the same manipulations as in \cite{diperna1989cauchy} and \cite{Golse2005-zh}, carried out pointwise in $\omega \in \Omega$.  We sketch the proof only to convince the reader that the compactness properties obtained in Proposition \ref{prop:Grand-Skorohod} are sufficient to deduce the claim in the same way as for the deterministic theory, without pulling any further subsequences (potentially depending on $\omega$).  Let $\varphi \in C_{c}(\R^{d}_{v})$.  We will fix an $\omega \in \Omega$ and mostly omit dependence on this variable throughout the proof. A change of variables from $(v,v_*) \to (v^\prime,v^\prime_*)$ and an application of Fubini yields the identities
\begin{equation} \label{eq:UsefulId}
\begin{split}
\left \langle  \frac{\BCol^{+}_n(f_{n},f_{n})}{1+\langle f_{n} \rangle} \, \varphi \right \rangle 
&= \left \langle f_{n} \frac{ \mathcal{L}_{n}f_{n} }{1+\langle f_{n} \rangle} \right \rangle. \\
\left \langle  \frac{\BCol^{+}(f,f)}{1+\langle f  \rangle} , \varphi \right \rangle
&= \left \langle f \frac{ \mathcal{L}f }{1+\langle f  \rangle}  \right \rangle,
\end{split}
\end{equation}
where $\mathcal{L}_{n}$ is the linear operator on $L^{1}(\R^{d}_{v})$ defined by
\begin{equation}
\mathcal{L}_{n}f(v) = \iint_{\R^{d} \times \S^{d-1}}f_{*}b_{n}(v-v_{*})\varphi'\dee v_{*}\dee \theta,
\end{equation}
and $\mathcal{L}$ is defined analogously, but with $b$ replacing $b_{n}$.  Since $\{f_{n}(\omega)\}_{n \in \N}$ converges to $f(\omega)$ in $\LLsMw{1}$ and is tight as a sequence in $\LLLs{1}$, while $b_{n}$ converges to $b$ pointwise on $\R^{d} \times \S^{d-1}$ and is bounded in $L^{\infty}(\R^{d} \times \S^{d-1})$ by Hypothesis \ref{hyp:regularized-collision}, one can deduce that, $\P$ almost surely, both $\{\mathcal{L}_nf_n\}_{n\in\N} \to \mathcal{L}f$ and $\{\langle f\rangle\}_{n\in\N} \to \langle f\rangle$ in measure on $[0,T] \times \R^{2d}$ and $[0,T]\times\R^d$ respectively. Therefore, $\P$ almost surely
\begin{equation}\label{eq:CleverRewrite}
\left \{ \frac{ \mathcal{L}_{n}f_{n} }{1+\langle f_{n} \rangle} \right \}_{n \in \N} \to \frac{\mathcal{L}f}{ 1+\langle f\rangle }\quad \text{in measure on}\quad [0,T] \times \R^{2d}.
\end{equation}
Using the uniform bounds on $\{b_n\}_{n\in\N}$ in $L^{\infty}(\R^{d} \times \S^{d-1})$, the sequence in \eqref{eq:CleverRewrite} is also uniformly bounded in $\LLLs{\infty}$, pointwise in $\omega$. Applying the second part of the product lemma \ref{lem:product_lemma} gives
\begin{equation}
\left \{ f_{n}\frac{ \mathcal{L}_{n}f_{n} }{1+\langle f_{n} \rangle} \right \}_{n \in \N} \to f\frac{\mathcal{L}f}{ 1+\langle f\rangle } \quad \text{in} \quad \LLsMw{1}.
\end{equation}
An approximation argument (since $1$ does not belong to $C_0(\Rd_v$)) and the pointwise (in $\omega$) uniform bounds on $\{f_{n}\}_{n \in \N}$ from Proposition \ref{prop:Grand-Skorohod} yields the $\p$ almost sure convergence
\begin{equation}
\left \langle f_{n} \frac{ \mathcal{L}_{n}f_{n} }{1+\langle f_{n}  \rangle} \right \rangle \to \left \langle f \frac{ \mathcal{L}f }{1+\langle f  \rangle}  \right \rangle \quad \text{in} \quad \LLt{1}. 
\end{equation}
In view of the identities \eqref{eq:UsefulId}, this completes the proof.
\end{proof}
The purpose of the next lemma is to reduce our analysis of $\mathcal{B}_{m}^{+}$ to regions where there are no concentrations in $\{f_{n}\}_{n \in \N}$.
\begin{lem} \label{lem:LinfRed}
As $R \to \infty$, the following limit holds $\p$ almost surely:
\begin{equation}
\frac{\BCol_{n}^{+}\left (f_{n},f_{n}\right)}{1+\langle f_{n}\rangle} \1_{f_{n}>R} \to 0 \quad \text{in} \quad \LLsMw{1},
\end{equation}
uniformly in $n \in \N$.
\end{lem}
\begin{proof}
  Let $\varphi \in C_{0}(\R^{d}_{v})$ be a non-negative function. Fix an $\omega \in \Omega$ and mostly omit dependence throughout the proof. The bound \eqref{eq:AckerydBd} yields the following inequality on $\Omega \times [0,T] \times \R^{2d}$: for all $K>1$, 
\begin{equation}
\BCol^{+}_{n}\left (f_{n},f_{n}\right) \leq (\log K)^{-1}\mathcal{D}^0_{n}(f_{n}) + K \BCol^{-}_{n}\left ( f_{n},f_{n}\right ).
\end{equation}
Hence, for almost every $(\omega,t,x) \in \Omega \times [0,T] \times \R^{d}$, we find that
\begin{align}
&\left \langle \frac{\BCol^{+}_{n} \left (f_{n},f_{n}\right )}{1+\langle f_{n}\rangle} \1_{f_{n}>R}\, \varphi  \right \rangle \\
&\leq (\log K)^{-1}\| \varphi\|_{L^{\infty}_{v}}\mathcal{D}_{n}\left (f_{n} \right )+ K \left \langle \frac{\BCol^{-}_{n} \left (f_{n},f_{n}\right )}{1+\langle f_{n}\rangle} \1_{f_{n}>R}\, \varphi \right \rangle.
\end{align}
Next we observe that pointwise in $\Omega$,
\begin{equation}
 \bigg \| \left \langle \frac{\BCol^{-}_{n} \left (f_{n},f_{n}\right )}{1+\langle f_{n} \rangle} \1_{f_{n}>R}\, \varphi \right \rangle \bigg \|_{\LLt{1}}
 \leq \| \overline{b}_{n}\|_{L^{\infty}_{v}}\| \varphi\|_{L^{\infty}_{v}}\|f_{n}\1_{f_{n}>R} \|_{\LLLs{1}}.
\end{equation}
By Proposition \ref{prop:Grand-Skorohod}, $\{f_{n}(\omega)\}_{n \in \N}$ is uniformly integrable in $\LLLs{1}$ and $\{\overline{b}_{n}\}_{n \in \N}$ is uniformly bounded in $L^{\infty}_{v}$, passing $R \to \infty$ yields
\begin{equation}\label{eq:finalIneq}
\limsup_{R \to \infty}\sup_{n \in \N} \bigg \| \left\langle \frac{\BCol^{+}_n \left (f_{n},f_{n}\right )}{1+\langle f_{n}\rangle} \1_{f_{n}>R}\, \varphi  \right \rangle \bigg\|_{\LLt{1}} \leq (\log K)^{-1}\| \varphi\|_{L^{\infty}_{v}}\sup_{n \in \N}\|\mathcal{D}_{n}\left (f_{n}\right)\|_{\LLt{1}},
\end{equation}
pointwise in $\Omega$.  By Proposition \ref{prop:Grand-Skorohod}, there exists a constant $C(\omega)$ such that
\begin{equation}
\sup_{n \in \N}\|\mathcal{D}_{n}( f_{n})(\omega)\|_{\LLt{1}} \leq C(\omega).
\end{equation}
Sending $K \to \infty$ on both sides of \eqref{eq:finalIneq} we find
\begin{equation}
  \lim_{R\to \infty}\sup_n\left\|\left\langle \frac{\BCol_{n}^{+}\left (f_{n},f_{n}\right)}{1+\langle f_{n}\rangle} \1_{f_{n}>R}\, \varphi \right\rangle\right\|_{L^1_{t,x}} \to 0.
\end{equation}
Since we can always split any $\varphi\in C_0(\Rd_v)$ into positive and negative parts also in $C_0(\Rd_v)$, the above convergence holds for any $\varphi \in C_0(\Rd_v)$, completing the proof.
\end{proof}
The next step is to apply Lemma \ref{lem:LinfRed} to obtain another Lemma written below.

\begin{lem}\label{lem:B_+-lim-continuity}
As $m \to \infty$, the following limit holds $\p$ almost surely:
\begin{equation}
\frac{\BCol_{m}^{+}}{1+\langle f\rangle } \to \frac{\BCol^+(f,f)}{1+\langle f\rangle } \quad \text{in} \quad \LLsMw{1}.
\end{equation}
\end{lem}
\begin{proof}
Let $\varphi \in C_{0}(\R^{d}_{v})$ be non-negative. Fix $\omega \in \Omega$ throughout and mostly omit.  The first step is to observe that for each fixed $m \in \N$, pointwise in $\Omega$,
\begin{equation}\label{eq:Lower-semicont}
\begin{split}
&\bigg \| \left \langle \frac{\BCol_{m}^{+}-\BCol^{+}\left (f,f \right)}{1+ \langle f \rangle }\,\varphi \right \rangle \bigg \|_{\LLt{1}} \\
& \leq \liminf_{n \to \infty} \bigg \| \left \langle \frac{ \Renorm_{m}'(f_{n}) \BCol_{n}^{+}\left( f_{n},f_{n}  \right)-\BCol^{+}_{n}\left (f_{n},f_{n} \right)}{1+\langle f_{n}\rangle}\,\varphi \right \rangle \bigg \|_{\LLt{1}}.
\end{split}
\end{equation}
Indeed, this follows from the following two observations.  In view of Lemma \ref{lem:LinfRed},
\begin{equation}
 \left \langle \frac{\BCol^{+}_{n} \left (f_{n},f_{n} \right )}{1+ \langle f_{n}\rangle }\, \varphi \right \rangle \to  
 \left \langle \frac{ \BCol^{+} \left (f,f \right )}{1+ \langle f\rangle}\,\varphi \right \rangle \quad \text{strongly in} \quad \LLt{1},
\end{equation}
pointwise in $\Omega$.  By Proposition \ref{prop:Grand-Skorohod}, $ \{ \Renorm_{m}'(f_{n}) \BCol_{n}^{+}(f_{n},f_{n})(\omega) \}_{n \in \N}$ converges to $\BCol_{m}^{+}(\omega)$ weakly in $\LLLs{1}$ and $\{f_{n}(\omega)\}_{n \in \N}$ converges to $f(\omega)$ in $\LLsMw{1}$. Therefore using the uniform bounds on $\{f_n(\omega)\}$ we conclude that $\langle f_{n}(\omega),1 \rangle$ converges to $\langle f(\omega),1 \rangle$ in measure on $[0,T] \times \R^{2d}$. Therefore, the product Lemma \ref{lem:product_lemma} yields the $\p$ almost sure convergence
\begin{equation}
\frac{\Renorm_{m}' \left ( f_{n} \right) \BCol_{n}^{+}\left( f_{n},f_{n}  \right)}{1+\langle f_{n},1 \rangle} \to \frac{\BCol_{m}^{+}}{1+ \langle f,1 \rangle } \quad \text{weakly in} \quad \LLLs{1}. 
\end{equation}
Now the desired inequality follows from the lower semi-continuity of the $\LLLs{1}$ norm with respect to weak convergence.  

The next step is to observe that for all $R>1$,
\begin{equation}\label{eq:splittingHighLow}
\begin{split}
&\bigg \| \left \langle \frac{ \Renorm_{m}'(f_{n}) \BCol_{n}^{+}\left( f_{n},f_{n}  \right)-\BCol^{+}_{n}\left (f_{n},f_{n} \right)}{1+\langle f_{n}\rangle}\,\varphi \right \rangle \bigg \|_{\LLt{1}}\\
&\hspace{.5in}\leq \left [1-\Big (1+\frac{R}{m}\Big )^{-2} \right] \bigg \| \left \langle \frac{\BCol^{+}_n \left (f_{n},f_{n} \right )}{1+ \langle f_{n} \rangle }\, \varphi \right \rangle \bigg \|_{\LLt{1}} + \bigg \| \left \langle \frac{\BCol^{+}_n \left (f_{n},f_{n} \right )}{1+ \langle f_{n} \rangle }1_{f_{n}>R}\, \varphi \right \rangle \bigg \|_{\LLt{1}}.
\end{split}
\end{equation}
Indeed, writing $1=1_{f_{n}<R}+1_{f_{n}\geq R}$ and recalling that $\Renorm_{m}'(x)=(1+\frac{x}{m})^{-2}$, we find the following upper and lower bounds hold pointwise in $\Omega \times [0,T] \times \R^{2d}$
\begin{equation}
\begin{split}
\frac{ \BCol_{n}^{+}\left (f_{n},f_{n}\right) }{ (1+\frac{f_{n}}{m})^{2} } &\leq \BCol_{n}^{+}\left (f_{n},f_{n}\right).\\
\frac{ \BCol_{n}^{+}\left (f_{n},f_{n}\right) }{ (1+\frac{f_{n}}{m})^{2} } &\geq \frac{ \BCol_{n}^{+}\left (f_{n},f_{n}\right) }{ (1+\frac{R}{m})^{2} }-\BCol_{n}^{+}\left (f_{n},f_{n}\right)1_{f_{n}\geq R} .
\end{split}
\end{equation}
Subtracting $\BCol_{n}^{+}\left (f_{n},f_{n}\right)$ on both sides, pairing with $\varphi$, dividing by $1+\langle f_{n}\rangle$, and integrating over $[0,T] \times \R^{d}$ gives the claim.  

Using \eqref{eq:Lower-semicont}, we may pass $n \to \infty$ on both side of \eqref{eq:splittingHighLow}, pointwise in $\Omega$. Appealing to Lemma \ref{Lem:AveragedPositivePart} to pass the limit in the first term on the right-hand side of (\ref{eq:splittingHighLow}), we find that for each $m \in \N$ and $R>1$, the following inequality holds pointwise in $\Omega$
\begin{align}
&\bigg \| \left \langle \frac{\BCol_{m}^{+}-\BCol^{+}\left (f,f \right)}{1+\langle f \rangle}\,\varphi \right \rangle \bigg \|_{\LLt{1}} \\
&\leq \left [1-\left(1+\frac{R}{m}\right)^{-2} \right] \bigg \| \left \langle  \frac{\BCol^{+}\left (f,f \right)}{1+ \langle f\rangle }\, \varphi \right \rangle \bigg \|_{\LLt{1}}+\sup_{n \in \N} \bigg \| \left \langle \frac{\BCol^{+}_n \left (f_{n},f_{n} \right )}{1+ \langle f_{n}\rangle }1_{f_{n}>R}\, \varphi \right \rangle \bigg \|_{\LLt{1}}.
\end{align}
Passing $m \to \infty$ yields for each $R>1$, pointwise in $\Omega$
\begin{equation}
\limsup_{m \to \infty}\bigg \| \left \langle \frac{\BCol_{m}^{+}-\BCol^{+}\left (f,f \right)}{1+\langle f\rangle }\,\varphi \right \rangle \bigg \|_{\LLt{1}}
\leq \sup_{n \in \N} \bigg \| \left \langle \frac{\BCol^{+}_n \left (f_{n},f_{n} \right )}{1+ \langle f_{n}\rangle  }1_{f_{n}>R}\, \varphi \right \rangle \bigg \|_{\LLt{1}}.
\end{equation}
Finally, sending $R \to \infty$ and applying Lemma \ref{lem:LinfRed} to remove the peaks completes the proof.
\end{proof}
\subsection{Proof of Proposition \ref{prop:RenormCollision}}
Finally, we can apply our lemmas in order to obtain our main Proposition.
\begin{proof}[Proof of Proposition \ref{prop:RenormCollision}]
  Let us begin with the analysis of the  negative part $\mathcal{B}_{m}^{-}$.  The first point is to observe that for all $\omega \in \Omega$, we may identify $\mathcal{B}_{m}^{-}(\omega)=\overline{\gamma_{m}(f)}\,\overline{b}*_vf(\omega)$. Indeed, recall that $\{ \Renorm_{m}'(f_{n})\mathcal{B}_{n}^{-}(f_{n},f_{n})(\omega) \}_{n \in \N}$ converges to $\mathcal{B}_{m}^{-}(\omega)$ weakly in $\LLLs{1}$ by Proposition \ref{prop:Grand-Skorohod}. On one hand, since $\{\overline{b}_{n}*_vf_{n}(\omega)\}_{n \in \N}$ is uniformly integrable in $L^1_{t,x,v}$ and converges in measure on $[0,T] \times \R^{2d}$ to $\overline{b}*_vf(\omega)$, then by Vitali convergence
  \begin{equation}
    \{\overline{b}_{n}*_vf_{n}(\omega)\}_{n \in \N} \to \overline{b}*_vf(\omega)\quad\text{in}\quad L^1_{t,x,v}.
  \end{equation}
  On the other hand, $\{f_{n}\Renorm_{m}'(f_{n})(\omega)\}_{n \in \N}$ converges weakly to $\overline{\gamma_{m}(f)}(\omega)$ in $\LLLs{1}$, and is uniformly (in $n$) bounded in $L^{\infty}_{t,x,v}$, then (up to a subsequence) $\{f_{n}\Renorm_{m}'(f_{n})(\omega)\}_{n \in \N}$ converges to $\overline{\gamma_m(f)}$ in $[L^\infty_{t,x,v}]^*$. Therefore (up to a subsequence),  since this is a weak-* $L^\infty$ - strongly $L^1$ product limit, we obtain
  \begin{equation}\label{eq:B-renorm-identification}
    \Gamma^\prime_m(f_n)\mathcal{B}_{n}^{-}(f_{n},f_{n})(\omega) \to \overline{\gamma_{m}(f)}\,\overline{b}*_vf(\omega) \quad \text{in} \quad [L^1_{t,x,v}]_{w}.
  \end{equation}
However, since $\{ \Renorm_{m}'(f_{n})\mathcal{B}_{n}^{-}(f_{n},f_{n})(\omega) \}_{n \in \N}$ converges to $\mathcal{B}_{m}^{-}(\omega)$ in $[\LLLs{1}]_{\mathrm{w}}$ the above convergence holds for the whole sequence and the claimed identification holds.

Next, by Corollary \ref{lem:Strong-Renorm-Limit}, $\overline{\Gamma_m(f)}(\omega) \to f(\omega)$ in $\LLLs{1}$, and by an analogous argument one can show $\overline{\gamma_m(f)}(\omega)\to f(\omega)$ in $\LLLs{1}$. This allows us to deduce that $\p$ almost surely,
\begin{equation}
\left\{\frac{\BCol_m^-}{1+\overline{\Renorm_{m}(f)}}\right \}_{m \in \N} \to \frac{B^-(f,f)}{1+f} \quad \text{in measure on}\quad [0,T] \times \R^{2d}.
\end{equation}
Since $\gamma_m(z) = z\Gamma^\prime_m(z) = (1+\frac{z}{m})^{-1}\Gamma_m(z)$, then $\overline{\gamma_{m}(f)}\leq \overline{\Gamma_m(f)}$ pointwise for each $m\in\N$. This yields the pointwise inequality
\begin{equation}\label{eq:B-_m-bound}
\frac{\mathcal{B}_{m}^{-}}{1+\overline{\Renorm_{m}(f)}} \leq \overline{b}*_vf.
\end{equation}
A double application of Lebesgue dominated convergence (first in $[0,T]\times\R^{2d}$ and then in $\Omega$) using the bound above and the fact that $f\in L^2(\Omega; \LLLs{1})$ allows us to complete the first part of the proof (in fact it gives strong convergence in $L^2(\Omega; L^1_{t,x,v})$).

To treat the positive part of the renormalized collision operator, observe that for each $m,n \in \N$, the bound \eqref{eq:AckerydBd} gives the pointwise bound
\begin{equation}
  \frac{\Gamma_m^\prime(f_n)\BCol_n^+(f_n,f_n)}{1+ \overline{\Gamma_m(f)}} \leq \frac{1}{\log{K}}\SDis^0_n(f_n) +  K\frac{\Gamma^{\prime}_m(f_n)B^{-}_n(f_n,f_n)}{1+\overline{\Gamma_m(f)}}.
\end{equation}
Next we pair with a positive $\phi \in C_0([0,T]\times\R^{2d})$ and pass the $n\to\infty$ limit on both sides of the inequality above and use the convergence of $\mathcal{D}^0_n(f_n)$ to $\overline{\SDis^0(f)}$ in $\mathcal{M}_{t,x,v}^*$ given in Proposition \ref{prop:Grand-Skorohod} and the inequality (\ref{eq:B-_m-bound}) to obtain
\begin{equation}
  \left\langle\frac{\BCol_m^+}{1+ \overline{\Gamma_m(f)}},\phi\right\rangle \leq \frac{1}{\log{K}}\left\langle \overline{\SDis^0(f)}, \phi\right\rangle +  K\left\langle\bar{b}*_v f,\phi\right\rangle.
\end{equation}
For the second term on the right-hand side above we used the convergence (\ref{eq:B-renorm-identification}) and the poinwise bound $\overline{\gamma_{m}(f)}\leq \overline{\Gamma_m(f)}$. Furthermore, using the fact that $(1+\overline{\Gamma_m(f)})^{-1}\BCol_m^+$ is in $L^1_{t,x,v}$ and taking $\phi$ to be a suitable approximation of the identity allows us to conclude the almost everywhere $\Omega\times[0,T]\times\R^{2d}$ inequality
\begin{equation}\label{eq:abs-L1-bound}
\frac{\mathcal{B}_{m}^{+}}{1+\overline{\Renorm_{m}(f)}} \leqs \overline{\SDis^0(f)}_{\mathrm{ac}} + \overline{b}*f,
\end{equation}
where $\overline{\SDis^0(f)}_{\mathrm{ac}}$ is the density of the absolutely continuous part of $\overline{\SDis^0(f)}$.

To finish the proof, we write
\begin{equation}
\frac{\mathcal{B}_{m}^{+}}{1+\overline{\Renorm_{m}(f)}} = \frac{1+\langle f,1\rangle }{1+\overline{\Renorm_{m}(f)}} \frac{\mathcal{B}_{m}^{+}}{1+\langle f\rangle}.
\end{equation}
By Corollary \ref{lem:Strong-Renorm-Limit}, 
\begin{align}
  &\left \{ \frac{1}{1+\overline{\Renorm_{m}(f)}} \right \}_{m \in \N} \to \frac{1}{1+f} \quad \text{in measure on} \quad [0,T] \times \R^{2d}, \\
  \intertext{and by Lemma \ref{lem:B_+-lim-continuity}}
&\left \{ \frac{\mathcal{B}_{m}^{+}}{1+\langle f\rangle} \right \}_{m \in \N} \to \frac{ \mathcal{B}^{+}(f,f) }{1+\langle f\rangle} \quad \text{in}\quad \LLsMw{1}.
\end{align}
The product limit Lemma \ref{lem:product_lemma}, gives $\P$ almost surely
\begin{equation}
  \left\{ \frac{\BCol_m^+}{(1+\langle f\rangle)(1+\overline{\Gamma_m(f)})}\right\}_{m\in\N} \to \frac{B^+(f,f)}{(1+\langle f\rangle) (1+f)}\quad \text{in}\quad \LLsMw{1},
\end{equation}
and therefore we can conclude (using the fact that $\langle f\rangle$ is independent of $v$), for each $\varphi\in C_0(\Rd_v)$,
\begin{equation}
  \left\{\left\langle\frac{\BCol_m^+}{1+\overline{\Gamma_m(f)}}\,\varphi\right\rangle\right\}_{m\in\N} \to \left\langle\frac{B^+(f,f)}{1+f}\,\varphi\right\rangle \quad \text{in measure on}\quad [0,T]\times\Rd_x.
\end{equation}

In view of the bound (\ref{eq:abs-L1-bound}) we would like to again use a double application of the dominated convergence theorem (first in $[0,T]\times \Rd_x$ and then in $\omega$) to complete the proof. Indeed in order to apply dominated convergence in $\Omega$ it suffices to show that $\overline{\SDis^0(f)}_{ac} \in L^2(\Omega;\LLLs{1})$. To show this, choose $\phi\in C_0([0,T]\times\R^{2d})$ non-negative. By the $\P$ almost sure convergence of $\SDis_n^0(f_n)$ in Proposition $\Meas_{t,x,v}^*$, $\{|\langle \SDis_n^0(f_n), \phi\rangle|^2\}_{n\in\N}$ converges to $|\langle \overline{\SDis^0(f)}, \phi\rangle|^2$ $\P$ almost surely. It follows by Fatou's Lemma (in $\Omega$) that
\begin{equation}
  \E|\langle \overline{\SDis^0(f)}_{\mathrm{as}}, \phi\rangle|^2\leq \E|\langle \overline{\SDis^0(f)}, \phi\rangle|^2 \leq  \sup_n \E |\langle \SDis_n^0(f_n), \phi\rangle|^2 \leq \|\phi\|_{\LLLs{\infty}}^2 \sup_n\E \|\SDis_n(f_n)\|^2_{\LLt{1}}.
\end{equation}

Since $\overline{\SDis^0(f)}_{\mathrm{as}} \geq 0$, we may replace $\phi$ by a sequence of non-negative functions $\{\phi_k\}_{k\in\N} \subseteq C_0(\Rd)$, $\phi_k \to 1$ pointwise and monotonically. Then, passing $k\to \infty$ using monotone convergence and using the uniform bounds on $\SDis_n(f_n)$ yields the result.
\end{proof}


\section{Proof of Main Result}\label{sec:Main-Proof}
\begin{proof}[Proof of Theorem \ref{thm:MainResult}]
We begin by proving estimates (\ref{eq:Moment-ent-dis-estimates-thm}). Recall that Proposition \ref{prop:Grand-Skorohod} implies that $\{f_{n}\}_{n \in \N}$ converges to $f$ in $\CLLw{1}$ with probability one. We begin by showing the bound on $(1+|x|^2 +|v|^2)f$.  Let $B_{R}$ denote the ball of radius $R>0$ in $\R^{2d}$.  It follows that $\p$ almost surely,
\begin{equation}
\|(1+|x|^{2}+|v|^{2})1_{B_{R}}f_{n}\|_{\LLL{\infty}{1}}^{p} \to \|(1+|x|^{2}+|v|^{2})1_{B_{R}}f\|_{\LLL{\infty}{1}}^{p}.
\end{equation}
By Fatou's Lemma in $\Omega$, we find that
\begin{equation}
\E\|(1+|x|^{2}+|v|^{2})1_{B_{R}}f\|_{\LLL{\infty}{1}}^{p} \leq \sup_{n \in \N} \E\|(1+|x|^{2}+|v|^{2})1_{B_{R}}f_{n}\|_{\LLL{\infty}{1}}^{p}< \infty,
\end{equation}
in view of Part 2 of Proposition \ref{prop:Grand-Skorohod}. Sending $R \to \infty$ and applying Fatou's Lemma once more yields
\begin{equation}
  \E\|(1+|x|^2 + |v|^2)f\|^p_{\LLL{\infty}{1}} < \infty.
\end{equation}

To show the bounds on $f|\log{f}|$ and $\SDis(f)$, we recall the proof of Lemma \ref{lem:entropy-Bounds}, where we showed that $\{f_n\}_{n\in\N}$ satisfies the following entropy equation $\P$- almost surely for each $t\in[0,T]$,
\begin{equation}\label{eq:entropy-identity-n}
\iint_{\R^{2d}} f_n(t)\log{f_n(t)}\dx\dv  + \int_0^t\int_{\Rd}\SDis_n(f_n(s))\dx\ds = \iint_{\R^{2d}} f_0\log{f_0}\dx\dv.
\end{equation}
Since $z \mapsto z\log{z}$ is convex, and $\{f_n\}_{n\in\N} \to f$ in $C_t([\LLs{1}]_w)$ $\P$ almost surely, then, by lower semi-continuity and the non-negativity of $\SDis_n(f_n)$, the following inequality holds pointwise in $\Omega\times[0,T]$,
\begin{equation}
  \iint_{\R^{2d}} f\log{f}\dx\dv \leq \iint_{\R^{2d}} f_0\log{f_0}\dx\dv. 
\end{equation}
From this point on, we may follow the arguments in Section \ref{subsec:Entropy-Bound} to conclude
\begin{equation}
  \E\|f\log{f}\|_{\LLL{\infty}{1}}^p <\infty.
\end{equation}

To show the bound on the dissipation $\SDis(f)$, we remark that a standard modification of the proof of Lemma \ref{Lem:AveragedPositivePart} allows us to conclude the $\P$ almost surely
\begin{equation}
 \frac{f^\prime_nf^\prime_{n,*}}{1+\ep\langle f_n\rangle} \to \frac{f^\prime f^\prime_{*}}{1+\ep\langle f\rangle}\quad \text{in} \quad [L^1([0,T]\times\R^{3d}_{x,v,v*}\times\S^{d-1})]_w,
\end{equation}
for each $\ep >0$. Similarly, by the product limit Lemma \ref{lem:product_lemma}, we may also conclude that $\P$ almost surely,
\begin{equation}
  \frac{f_n f_{n,*}}{1+\ep\langle f_n\rangle} \to \frac{f f_{*}}{1+\ep\langle f\rangle} \quad \text{in}\quad [L^1([0,T]\times\R^{3d}_{x,v,v_*})]_w.
\end{equation}
Notice that the function $(x,y) \mapsto (x-y)(\log{x} - \log{y})$ is convex on $\R^2_+$. Therefore, by lower semi-continuity we may conclude that $\P$ almost surely, for every $t\in [0,T]$ and each $\ep >0$
\begin{equation}
\begin{aligned}
  &\int_0^t\iiint_{\R^{3d}\times{\S^{d-1}}} \frac{d(f)b}{1+\ep\langle f \rangle}\dee{\theta}\dv\dv_*\dx\ds\\
&\hspace{1in} \leq \liminf_n \int_0^t\iiint_{\R^{3d}\times{\S^{d-1}}} \frac{d(f_n)b}{1+\ep\langle f_n\rangle}\dee{\theta}\dv\dv_*\dx\ds\\
 &\hspace{1in}\leq \liminf_n \int_0^t\int_{\Rd} \SDis_n(f_n)\dx\ds.
\end{aligned}
\end{equation}
Taking $\ep \to 0$, by the monotone convergence theorem, gives
\begin{equation}
  \int_0^t \int_{\Rd} \SDis(f)\dx\ds \leq \liminf_n \int_0^t\int_{\Rd} \SDis_n(f_n)\dx\ds.
\end{equation}
Passing $n\to\infty$ on both sides of (\ref{eq:entropy-identity-n}) yields, the global entropy inequality (\ref{eq:global-entropy-ineq}),
\begin{equation}
  \iint_{\R^{2d}} f(t)\log{f}(t)\dx\dv + \int_0^t\int_{\Rd} \SDis(f)(s)\dx\ds \leq \int_{\R^{2d}} f_0\log{f_0}\dx\dv.
\end{equation}
Whereby we obtain the bound
\begin{equation}
  \|\SDis(f)\|_{\LLLs{1}} \leq \|f\log{f}\|_{\LLL{\infty}{1}} + \|f_0\log{f_0}\|_{\LLs{1}}.
\end{equation}
Using the bound on $f\log{f}$ above, gives
\begin{equation}
  \E \|\SDis(f)\|_{\LLLs{1}}^p <\infty.
\end{equation}

Next we show the conservation laws (\ref{eq:local-mass-conv}-\ref{eq:global-entropy-ineq}). In fact we have already shown (\ref{eq:global-entropy-ineq}) in the computation above. To show (\ref{eq:local-mass-conv}-\ref{eq:global-energy-ineq}), recall that $f_n$ satisfies for each $\varphi\in C^\infty_c(\R^{2d})$
\begin{equation}\label{eq:approximate-Boltzmann-n}
  \begin{aligned}
    &\langle f_n, \varphi\rangle = \langle f_n^0,\varphi\rangle  + \int_0^t\langle f_n(s), v\cdot\nabla_x\varphi + \mathcal{L}_{\sigma^n}\varphi\rangle\ds\\
    &\hspace{1in}+\int_0^t\langle f_n(s),\sigma^n_k\cdot\nabla_v\varphi\rangle\dee\beta^n_k(s) + \int_0^t \langle\BCol_n(f_n,f_n),\varphi\rangle\ds
  \end{aligned}
\end{equation}
in distribution in $x,v$. Using the $\P$ almost sure moment estimates provided by property 2 in Proposition \ref{prop:Grand-Skorohod} and the boundedness of the truncated collision operator $\BCol_n(f_n,f_n)$
\begin{equation}\label{eq:p-as-moment-estimate}
  \begin{aligned}
     \|(1+|x|^2+|v|^2)^kf_n\|_{L^\infty_t(L^1_{x,v})} <\infty, \quad \|(1+|x|^2 +|v|^2)^k\BCol_n(f_n,f_n)\|_{L^1_{t,x,v}} <\infty.
  \end{aligned}
\end{equation}
It is straight forward to use these estimates to upgrade equation (\ref{eq:approximate-Boltzmann-n}) to a class of test functions $\varphi(x,v)$ with polynomial growth. Indeed, choosing the test function $\varphi$ to be $1,v$ and $\tfrac{1}{2}|v|^2$ in give the following balances,
\begin{equation}\label{eq:local-mass-conv-n}
  \int_{\Rd} f_n(t) \dv + \Div_x\int_0^t\int_{\Rd}vf_n(s)\dv\ds = \int_{\Rd} f_0\dv  \quad \text{in}\quad \mathcal{D}^\prime_x,
\end{equation}
\begin{equation}\label{eq:local-momentum-n}
  \E\iint_{\R^{2d}} vf_n(t) \dv\dx = \E\int_0^t\iint_{\R^{2d}}(\mathcal{L}_{\sigma^n}v)f_n(s)\dv\dx\ds + \int_{\R^{2d}} vf_0^n\dv\dx,
\end{equation}
\begin{equation}\label{eq:local-energy-n}
  \E\iint_{\R^{2d}} \frac{1}{2}|v|^2 f_n(t) \dv\dx = \E\int_0^t\iint_{\R^{2d}}\frac{1}{2}(\mathcal{L}_{\sigma^n}|v|^2)f_n(s)\dv\dx\ds + \int_{\R^{2d}} \frac{1}{2}|v|^2f_0^n\dv\dx.
\end{equation}

In order the pass the limit in $n$ above, will find it useful to prove the following extension of the product limit Lemma \ref{lem:product_lemma} for the sequence $\{f_n\}_{n\in\N}$.
\begin{lem}
  Let $\{\phi_n\}_{n\in\N}$ be a sequence of functions in $[L^\infty_{x,v}]_{\loc}$ converging pointwise a.e to $\phi$ satisfying the uniform growth assumption
  \begin{equation}\label{eq:phi-stict-sub-quad}
    \lim_{R\to\infty}\sup_n\left\|\frac{\phi_n(x,v)}{1+|x|^2+|v|^2}\1_{B_R^c}\right\|_{L^\infty_{x,v}} = 0, \quad \left\|\frac{\phi(x,v)}{1+|x|^2+|v|^2}\1_{B_R^c}\right\|_{L^\infty_{x,v}} <\infty.
  \end{equation}
  where $B_R\subset \R^{2d}$ is the ball of radius $R$. Then,
  \begin{equation}
     \iint_{\R^{2d}} \phi_nf_n\dv\dx \to \iint_{\R^{2d}} \phi f \dv\dx \quad \text{in} \quad [L^2(\Omega\times[0,T])]_{\mathrm{w}}.
  \end{equation}
\end{lem}
\begin{proof}
  Proposition \ref{prop:Grand-Skorohod} implies that $\P$ almost surely $\{f_n\}_{n\in\N} \to f$ in $C_t([L^1_{x.v}]_{\mathrm{w}})$. Since $\phi_n \1_{B_R}$ is uniformly bounded in $L^\infty_{x,v}$ and converges in pointwise a.e. to $\phi \1_{B_R}$ the product limit Lemma \ref{lem:product_lemma} implies that $\P$ almost surely for each $t\in[0,T]$
  \begin{equation}\label{eq:local-product-limi}
   \int_{\R^{2d}}\phi_n\1_{B_R}f_n(t)\dv\dx \to \int_{\R^{2d}}\phi\1_{B_R}f(t)\dv\dx
  \end{equation}
  Now, letting $C_1 <\infty$ denote the (random) constant such that
  \begin{equation}
    \sup_n\|(1+|x|^2+|v|^2)(|f_n|+ |f|)\|_{L^\infty_t(L^1_{x,v})} < C_1,
  \end{equation}
  and $C_2<\infty$ be such that
  \begin{equation}
   \left\|\frac{|\phi(x,v)| + |\phi_n(x,v)|}{1+|x|^2+|v|^2}\1_{B_R^c}\right\|_{L^\infty_{x,v}} < C_2.
  \end{equation}
  We have
  \begin{equation}
    \begin{aligned}
      \left|\int_{\R^{2d}}(\phi f -  \phi_nf_n)\1_{B_R^c}\dv\dx\right| &= \left|\int_{\R^{2d}}(\phi - \phi_n)f\1_{B_R^c}\dv\dx\right| + \left|\int_{\R^{2d}}\phi_n(f-f_n)\1_{B_R^c}\dv\dx\right|\\
      &\hspace{-.5in}\leq C_2 \|(1+|x|^2+|v|^2)f\1_{B_R^c}\|_{L^\infty_t(L^1_{x,v})} + C_1\sup_n\left\|\frac{\phi_n(x,v)}{1+|x|^2+|v|^2}\1_{B_R^c}\right\|_{L^\infty_{x,v}}.
    \end{aligned}
  \end{equation}
  Taking the $\limsup$ of both sides and and sending $R\to \infty$, we conclude
  \begin{equation}
    \lim_{R\to\infty}\limsup_n\left|\int_{\R^{2d}}(\phi f -  \phi_nf_n)\1_{B_R^c}\dv\dx\right| = 0.
  \end{equation}
  Therefore, in light of the decomposition
  \begin{equation}
    \int_{\R^{2d}}(\phi f -  \phi_nf_n)\dv\dx = \int_{\R^{2d}}(\phi f -  \phi_nf_n)\1_{B_R}\dv\dx + \int_{\R^{2d}}(\phi f -  \phi_nf_n)\1_{B_R^c}\dv\dx,
  \end{equation}
and (\ref{eq:local-product-limi}), we conclude that for all $\phi_n$ satisfying (\ref{eq:phi-stict-sub-quad}), $\P$ almost surely, and for each $t\in [0,T]$,
    \begin{equation}
      \int_{\R^{2d}}\phi_nf_n\dv\dx \to \int_{\R^{2d}}\phi f\dv\dx.
    \end{equation}
    Moreover by the average moment estimate on $\{f_n\}_{n\in\N}$,
 \begin{equation}
  \left\{\iint_{\R^{2d}} \phi_nf_n \dv\dx\right\}_{n\in \N}\quad  \text{is bounded in}\quad  L^2(\Omega\times[0,T]),
\end{equation}
and therefore by Vitali convergence we may conclude that
\begin{equation}
  \iint_{\R^{2d}} \phi_nf_n\dv\dx \to \iint_{\R^{2d}}\phi f \dv\dx \quad \text{in} \quad [L^2(\Omega\times[0,T])]_{\mathrm{w}}.
\end{equation}
  \end{proof}

Immediately we can use this Lemma to pass the limit in each term of (\ref{eq:local-mass-conv-n}). Taking the derivative in time gives the local conservation law (\ref{eq:local-mass-conv}). Also using the fact that $\mathcal{L}_{\sigma^n}v = \sigma^n_k\cdot\nabla_v\sigma^n_k$ is bounded in $L^\infty_{x,v}$ and converges pointwise to $\mathcal{L}_{\sigma}v$, we may also pass the limit in each term of (\ref{eq:local-momentum-n}) to obtain (\ref{eq:global-momentum-bal}).

Now, note that we cannot pass the limit directly in the energy equation (\ref{eq:local-energy-n}) since $\frac{1}{2}|v|^2$ does not satisfy (\ref{eq:phi-stict-sub-quad}). However, $\mathcal{L}_{\sigma^n}|v|^2$ does satisfy (\ref{eq:phi-stict-sub-quad}), and so upon cutting of the domain on the left hand side of (\ref{eq:local-energy-n}) can pass the limit in $n$ and conclude for each $R>0$,
\begin{equation}
  \E\int_{\R^{2d}} \frac{1}{2}\1_{|v|<R}|v|^2f(t) \leq \E\int_0^t\iint_{\R^{2d}}\frac{1}{2}(\mathcal{L}_{\sigma}|v|^2)f(s)\dv\dx\ds + \int_{\R^{2d}} \frac{1}{2}|v|^2f_0\dv\dx.
\end{equation}
Apply the monotone convergence theorem to the left-hand side and sending $R\to \infty$ gives the desired inequality (\ref{eq:global-energy-ineq}).

Next, we prove that $f$ verifies the conditions of Definition \ref{def:Renorm-Martingale-BM}. Begin by observing that for each $n \in \N$, $\tilde{f}_{n}$ has the property that for each $(t, \omega) \in [0,T] \times \Omega$, the quantity $\tilde{f}_{n}(t,\omega)$ is a non-negative element of $\LLs{1}$.  Since $f_{n}$ is given explicitly as $f_{n}=\tilde{f}_{n}\circ \tilde{T}_{n}$, it inherits this property.  Finally, Proposition \ref{prop:Grand-Skorohod} implies that $\{f_{n}(t,\omega) \}_{n \in \N}$ converges to $f(t,\omega)$ weakly in $\LLs{1}$.  Since weak convergence is order preserving, this shows that $f$ satisfies Part 1 of Definition \ref{def:Renorm-Martingale-BM}.  Also, by Corollary \ref{lem:Strong-Renorm-Limit}, $f: \Omega \times [0,T] \to L^{1}_{x,v}$ has continuous sample paths. 

In view of Definition \ref{def:Martingale-Sol} and Remark \ref{rem:stoch-integral-form-of-eq}, Parts 2 and 3 of Definition \ref{def:Renorm-Martingale-BM} will follow once we check that for each $\Renorm \in \mathcal{R}$, the process $\Renorm(f)$ is a weak martingale solution to the stochastic kinetic equation driven by $\Renorm'(f)\mathcal{B}(f,f)$, starting from $\Renorm(f^{0})$. In fact, the problem can be reduced further. 

Let us show that it suffices to verify $\log(1+f)$ is a weak martingale solution driven by $(1+f)^{-1}\mathcal{B}(f,f)$, starting from $\log(1+f^{0})$. Assume for the moment this property of $\log(1+f)$ and let $\Renorm \in \mathcal{R}$ be arbitrary. Since we showed $f \in L^{2}(\Omega; \LLL{\infty}{1})$, it follows that $\log(1+f)$ belongs to $L^{2}(\Omega \times [0,T] \times \R^{2d})$. Hence, by Proposition \ref{prop:Weak_Is_Renormalized}, $\log(1+f)$ is a renormalized solution. We would like to renormalize by a $\beta$ such that $\beta \circ \log(1+x)=\Renorm(x)$, or equivalently $\beta(x)=\Renorm(e^{x}-1)$, but this is not quite admissible in the sense of Definition \ref{def:Renorm-Martingale} since $\Gamma \in \mathcal{R}$ need not imply boundedness of $\beta''$. Instead, we proceed by a sequence of approximate renormalizations $\{\beta_{k}\}_{k \in \N}$ where   
$\beta_{k}(x)=\Renorm_{k}(e^{x}-1)$ and $\{\Renorm_{k}\}_{k \in \N}$ have the following properties: for each $k \in \N$, $\Renorm_{k}$ is compactly supported (and hence $\beta_{k}''$ is bounded), the pair $(\Renorm_{k},\Renorm_{k}') \to (\Renorm,\Renorm')$ pointwise in $\R_{+}$, and the following uniform bound holds
\begin{equation}
\sup_{k \in \N}\sup_{x \in \R_{+}}(1+x)|\Gamma_{k}'(x)|<\infty.
\end{equation}
By Proposition \ref{prop:Weak_Is_Renormalized}, it follows that $\Renorm_{k}(f)$ is a weak martingale solution driven by $\Renorm_{k}'(f)\mathcal{B}(f,f)$. Using the properties of $\{\Gamma_k\}_{k\in \N}$ and the fact that $f\in L^{2}(\Omega; \LLL{\infty}{1})$ and $(1+f)^{-1}\BCol(f,f)\in L^2(\Omega; \LLLs{1})$, it is straight forward to use the stability result, Proposition \ref{prop:stability-weak-martingale}, to pass $k\to\infty$ and conclude that $\Gamma(f)$ is a weak martingale solution driven by $\Gamma^\prime(f)\BCol(f,f)$ starting from $\Gamma(f_0)$. 

Thus, it remains to show that $\log(1+f)$ is a weak martingale solution to the stochastic kinetic equation driven by $(1+f)^{-1}\mathcal{B}(f,f)$, starting from $\log(1+f_{0})$.  For this, we use once more our stability result.  Recall that for each $m \in \N$, the process $\log (1+ \overline{\Renorm_{m}(f)} )$ is a weak martingale solution to the stochastic kinetic equation driven by $(1+\overline{\Renorm_{m}(f)})^{-1}[\mathcal{B}_{m}^{+}-\mathcal{B}_{m}^{-}]$, starting from $\log ( 1+ \Renorm_{m}(f_{0}) )$. First observe that that for all $\varphi \in C^{\infty}_{c}(\R^{2d})$, the sequence $\{\langle \log(1+\overline{\Renorm_{m}}), \varphi \rangle\}_{m \in \N}$ converges in $L^{2}(\Omega ; C_{t})$ towards $\langle \log(1+f),\varphi \rangle$.  Indeed, this follows from Corollary \ref{lem:Strong-Renorm-Limit}, the almost everywhere inequality $\overline{\Renorm_{m}} \leq f$, and the estimates (\ref{eq:Moment-ent-dis-estimates-thm}). Next, for each $t \in [0,T]$ we can use Proposition \ref{prop:RenormCollision} with $\phi=\1_{[0,t]}\varphi$ to conclude that 
\begin{equation}
\int_{0}^{t} \bigg\langle \frac{\mathcal{B}_{m}}{1+\overline{\Renorm_{m}}},\varphi \bigg\rangle \ds \to \int_{0}^{t}\bigg\langle \frac{\mathcal{B}(f,f)}{1+f}, \varphi \bigg\rangle \ds \quad \text{in} \quad L^{2}(\Omega).
\end{equation}
Using these facts together with the stability result Proposition \ref{prop:stability-weak-martingale}, we may pass $m \to \infty$ and complete the proof.
\end{proof}


\begin{appendices}
\addtocontents{toc}{\protect\setcounter{tocdepth}{0}} 
\section{Compactness and tightness criterion}

Let $(\Omega, \mathcal{F},\p)$ be a probability space and $(E,\tau,\mathcal{B}_{\tau})$ be a topological space endowed with its Borel sigma algebra.  A mapping $X : \Omega \to (E,\tau)$ is called an ``$E$ valued random variable'' provided it is a measurable mapping between these spaces.  Every $E$ valued valued random variable induces a probability measure on $(E,\tau,\mathcal{B}_{\tau})$ by pushforward.  A  sequence of probability measures $\left \{\p_{n} \right \}_{n \in \N}$ on $\mathcal{B}_{\tau}$ is said to be ``tight'' provided that for each $\epsilon >0$ there exists a $\tau$ compact set $K_{\epsilon}$ such that $\p_{n}(K_{\epsilon}) \geq 1-\epsilon$ for all $n \in \N$.  

\begin{definition}
A topological space $(E,\tau)$ is called a Jakubowski space provided it admits a countable sequence continuous functionals which separate points.
\end{definition} 
Our main interest in such spaces is the following fundamental result given in \cite{jakubowski1997non}.
\begin{thm} \label{Thm:Appendix:Jakubowksi_Skorohod}
Let $(E,\tau)$ be a Jakubowski space.  Suppose $\{\tilde{X}_{n} \}_{n \in \N}$ is a sequence of $E$ valued random variables on a probability spaces $(\tilde{\Omega}, \F,\p)$ inducing tight laws with respect to the topology $\tau$. Then there exists a new probability space $(\Omega , \F, \p)$ endowed with an $E$ valued random variable X and a sequence of measurable maps $ \{ \widetilde{T}_{n} \}_{n \in \N}$ 
$$ \widetilde{T}_{n} : (\Omega , \F, \p) \to (\tilde{\Omega}, \tilde{\F},\tilde{\p})$$
with the following two properties:
\begin{enumerate}
\item For each $n \in \N$, the measure $\tilde{\p}_{n}$ is the pushforward of $\p$ by $\widetilde{T}_{n}$.
\item The new sequence  $\{X_{n}\}_{n \in \N}$ defined via $X_{n}=\widetilde{X}_{n} \circ \widetilde{T}_{n}$ converges $\p$ a.s. to $X$ (with respect to the topology $\tau$).
\end{enumerate}
\end{thm}
 

We begin by recalling the following `compact plus small ball'' criterion for compactness in Frechet spaces.

\begin{lem}\label{lem:compact+small}
  Let $F$ be a Fr\'{e}chet space. Then $U\subset F$ is precompact in $F$ if for every $\epsilon>0$, there exists a compact set $K_\epsilon\subset F$, such that
  \begin{equation}
    U \subset K_\epsilon + B_\epsilon,
  \end{equation}
where $B_\ep$ is a $\rho$-ball centered at $0$ of radius $\epsilon$, for a given metric $\rho$.
\end{lem}
\begin{proof}
  Fix $\ep >0$ and let $K_\ep$ be the compact set defined as above. Since $K_\ep$ is compact and $F$ is a metric space, it is totally bounded. Therefore there exists a finite collection of points $\{x_i\}_{i=1}^N$ so that $K_\ep \subseteq \bigcup_{i=1}^NB_\ep(x_i)$. However, since
  \begin{equation}
    K\subseteq \bigcup_{i=1}^NB_\ep(x_i) + B_\ep(0) \subseteq \bigcup_{i=1}^NB_{2\ep}(x_i), 
  \end{equation}
  then $K$ is totally bounded and therefore precompact in $F$.
\end{proof}

In the stochastic setting, we make use of the analogous version as a tightness criterion.

\begin{lem}\label{lem:tightness-decomp}
Let $F$ be a Frechet space and $\{X_{n}\}_{n \in \N}$ be a sequence of $F$-valued random variables. Assume that for all $L \in \R_{+}$ there exists a decomposition
\begin{equation}
X_{n}=Y_{n}^{L}+Z_{n}^{L},
\end{equation}
where $\{Y_{n}^{L}\}_{n \in \N}$ induces a tight sequence of laws on $F$.  If in addition, $Z_n^L$ satisfies for every $\eta >0$,
\begin{equation}
  \lim_{L\to\infty}\sup_n\P\left(Z_{n}^L \notin B_\eta\right) = 0.
\end{equation}
Then $X_n$ induces tight laws on $F$.
\end{lem}
\begin{proof}
  Fix $\ep >0$ and choose a sequence $\{L_j\}_{j\in\N}$ so that
  \begin{equation}
    \sup_n \P\left( Z^{L_j}_n \notin B_{1/j}\right) < \ep/2^j.
  \end{equation}
  By the tightness of $Y_n^L$, for each $j\in \N$ there is a compact set $K_j\subseteq F$ such that
  \begin{equation}
    \sup_n \P\left(Y_n^{L_j} \in K_j\right) < \ep/2^j
  \end{equation}
  By the classical compactness criterion, Lemma \ref{lem:compact+small}, the set
  \begin{equation}
    K = \bigcap_j (K_j + B_{1/j}).
  \end{equation}
  is compact in $E$. It follows that
  \begin{equation}
    \sup_n \P(X_n \notin K) \leq \sum_{j}\left(\sup_n\P\left(Y_n^{L_j} \notin K_j\right) + \sup_n\P\left(Z_n^{L_j} \notin B_{1/j}\right)\right) < 2\ep.
  \end{equation}
  Therefore $\{X_n\}_{n\in\N}$ induce tight laws on $F$.
\end{proof}

Next, we recall the classical Dunford-Pettis compactness criterion on $[L^1]_{\mathrm{w},\loc}$.
\begin{lem}\label{lem:Dunford-Pettis}
  Let $K$ be a bounded subset of $[L^1(\Rd)]_\loc$, then $K$ is precompact in $[L^1(\Rd)]_{\mathrm{w},\loc}$ if and only if the following limit holds
  \begin{equation}
    \lim_{L\to\infty} \sup_{f\in K} \|f\1_{|f|> L}\|_{L^1}= 0.
  \end{equation}
\end{lem}

In the stochastic setting, the corresponding tightness condition is:
\begin{lem}\label{lem:L1-tightness-crit}
  Let $\mu_n$ be a sequence of probability measures on $L^1(\Rd)_\loc$, then $\{\mu_n\}_{n\in\N}$ are tight on $[L^1(\Rd)]_{\mathrm{w},\loc}$ if and only if for every $\eta > 0$ the following limit hold
  \begin{equation}\label{eq:L1-tight-limits}
      \lim_{L\to\infty} \sup_n\mu_n\left\{f: \|f\1_{|f|> L}\|_{L^1} > \eta\right\} = 0.
  \end{equation}
\end{lem}

\begin{proof}
  First suppose that the limits (\ref{eq:L1-tight-limits}) hold. Let $\ep >0$ and choose a sequence $\{L_k\}$ such that
    \begin{equation}
        \sup_n\mu_n\left\{f : \|f\1_{|f|> L_k}\|_{L^1} > 1/k \right\} <\ep 2^{-k}.
    \end{equation}
    Define the closed set
    \begin{equation}
      A_k = \left\{f : \|f\1_{|f|> L_k}\|_{L^1} \leq 1/k \right\}.
    \end{equation}
    Then by the classical compactness criterion in Lemma \ref{lem:Dunford-Pettis},
    \begin{equation}
      K = \bigcap_k A_k
    \end{equation}
    is a compact set in $[L^1(\Rd)]_{\mathrm{w},\loc}$. Furthermore, we have
    \begin{equation}
      \sup_n\mu_n(K) \leq \sum_k\sup_n\mu_n(A_k) < \ep. 
    \end{equation}
    Therefore $\{\mu_n\}$ are tight on $[L^1(\Rd)]_{\mathrm{w},\loc}$.

    Next suppose that $\{\mu_n\}$ are tight on $[L^1(\Rd)]_{\mathrm{w}}$. And let $K$ be a compact subset of $[L^1]_{\mathrm{w}}$ such that $\sup_n \mu_n(K^c) < \ep$. For each $\eta >0$ it follows by the compactness criterion in Lemma \ref{lem:Dunford-Pettis} that for large enough $L$ (depending on $\eta$), the following set is empty
    \begin{equation}
      \left\{ f\in K : \|f\1_{|f|>L}\|_{L^1} > \eta\right\} = \emptyset.
    \end{equation}
    Therefore for large enough $L$ we have
    \begin{equation}
      \sup_n\mu_n\left\{f : \|f\1_{|f|>L}\|_{L^1} > \eta \right\} \leq \sup_n\mu_n(K^c) < \ep.
    \end{equation}
\end{proof}

We now introduce a useful tightness criterion for probability measures on $\CLw{1}$. First we will need a basic criterion for compactness in $\CLw{1}$. 
\begin{lem}\label{lem:Compactness-crit}
  Let $K \subseteq C([0,T];[L^1(\Rd)]_w)$ and denote for each $\varphi\in C^\infty_c(\Rd)$, the set
  \begin{equation}
    K_\varphi = \left\{ \langle f,\varphi\rangle\,:\, f\in K\right\} \subseteq C([0,T]).
  \end{equation}
Then $K$ is precompact in $C([0,T]\,;\,[L^1(\Rd)]_w)$ if any only if $K$ is a weakly precompact subset of $L^\infty([0,T];L^1(\Rd))$ and $K_\varphi$ equicontinuous in $C([0,T])$ for each $\varphi \in C^\infty_c(\Rd)$.
\end{lem}

This gives rise to the following tightness criterion on $\CLw{1}$.

\begin{lem} \label{Lem:Appendix:TightCrit}
  Let $\{\mu_n\}_{n\in\N}$ be a sequence of probability  measures on $C([0,T],[L^1(\Rd)]_w)$, and for any $\varphi\in C^\infty_c(\Rd)$, let $\{\nu_n^\varphi\}_{n\in\N}$ be the sequence of measures on $C([0,T])$ induced by the mapping $f\mapsto \langle f,\varphi\rangle$. Then the measures $\{\mu_n\}_{n\in\N}$ are tight if and only if $\{\nu_n^\varphi\}_{n\in\N}$ are tight for every $\varphi \in C^\infty_c$ and for every $\eta >0$ we have
       \begin{equation}
      \lim_{M\to \infty} \sup_n\mu_n\Big\{ f: \|f\|_{L^\infty_t(L^1)} > M\Big\} = 0,
    \end{equation}
    \begin{equation}
      \lim_{L\to\infty} \sup_n\mu_n\Big\{ f : \|f\1_{|f|>L}\|_{L^\infty_t(L^1)} >\eta\Big\} = 0,
    \end{equation}
and 
    \begin{equation}
      \lim_{R\to\infty} \sup_n\mu_n\Big\{ f : \|f \1_{B_R^c}\|_{L^\infty_t(L^1)} >\eta \Big\} = 0,
    \end{equation}
\end{lem}

\begin{proof}
Define for any function $f\in C([0,T])$ and $\delta >0$ the modulus of continuity
\begin{equation}
  \omega_\delta(f) := \sup_{|t-s|<\delta}|f(t)-f(s)|.
\end{equation}
We prove sufficiency first. Let $\ep >0$, and let $\{\varphi_j\}$ be a dense subset of $C_c^\infty(\Rd)$. Then by the classical tightness criterion for functions in $C([0,T])$, we can conclude that for each $\eta>0$ and $\varphi_j$, we have
\begin{equation}
  \lim_{\delta\to 0} \sup_n\mu_n\Big\{ f : \omega_{\delta}(\langle f, \varphi_j\rangle) > \eta\Big\} = 0.
\end{equation}

Therefore for each $j,k \geq 0$ we may choose values $(M_k,L_k,R_k,\delta_{k,j})$ so that
\begin{equation}
  \sup_n\mu_n\Big\{ f: \|f\|_{L^\infty_t(L^1)} > M_k\Big\} <\ep 2^{-k}
\end{equation}
  \begin{equation}
    \sup_n\mu_n\Big\{ f : \|f\1_{|f|>L_k}\|_{L^\infty_t(L^1)} >1/k\Big\} <\ep 2^{-k}
  \end{equation}
  \begin{equation}
    \sup_n\mu_n\Big\{ f : \|f \1_{B_{R_k}^c}\|_{L^\infty_t(L^1)} >1/k \Big\} < \ep 2^{-k}
  \end{equation}
  \begin{equation}
    \sup_n\mu_n\Big\{ f: \omega_{\delta_{k,j}}(\langle f,\varphi_j\rangle) > 1/k\Big\} <\ep 2^{-k-j}.
  \end{equation}
Define the closed sets,
\begin{equation}
\begin{aligned}
  &A_k = \Big\{f: \|f\|_{L^\infty_t(L^1)} \leq M_k\Big\}
  &B_k = \Big\{ f : \|f\1_{|f|>L_k}\|_{L^\infty_t(L^1)} \leq 1/k\Big\}\\
  &C_k = \Big\{ f : \|f \1_{B_{R_k}^c}\|_{L^\infty_t(L^1)} \leq 1/k \Big\}
  &D_{k,j} = \Big\{ f: \omega_{\delta_{k,j}}(\langle f,\varphi_j\rangle)\leq 1/k\Big\}.
\end{aligned}
\end{equation}
and let
\begin{equation}
  K = \bigcap_{j,k} A_k\cap B_k\cap C_k \cap D_{k,j}.
\end{equation}
By the compactness criterion in Lemma \ref{lem:Compactness-crit} it is straight forward to verify that $K$ is a compact subset of $C([0,T]\,;\, [L^1]_w)$. Furthermore, we have
\begin{equation}
  \mu_n(K^c) \leq \sum_k \mu_n(A_k^c) + \sum_k \mu_k(B_k^c) + \sum_k \mu_k(C_k^c) + \sum_{k,j} \mu_n (D_{k,j}^c) < 4\ep,
\end{equation}
whereby tightness follows.

To prove necessity. We remark that since $f\mapsto \langle f,\varphi\rangle$ is continuous from $C([0,T]\,;\, [L^1]_w)$ to $C([0,T])$ for every $\varphi\in C^\infty_c(\Rd)$, then tightness of $\{\mu_n\}_{n\in\N}$ automatically implies tightness of $\{\nu^\varphi_n\}_{n\in\N}$. Now let $\ep>0$ and let $K$ be the compact subset of $C([0,T]\,;\, [L^1]_w)$ such that $\sup_n\mu_n(K^c) <\ep$. Fix an $\eta >0$. The compactness criterion in Lemma \ref{lem:Compactness-crit} implies that there exist $(M^\prime,L^\prime, R^\prime)$ such that for and $M > M^\prime$, $L>L^\prime$, $R>R^\prime$ the following sets are empty
\begin{equation}
\begin{aligned}
   &\Big\{ f\in K: \|f\|_{L^\infty_t(L^1)} > M\Big\} = \emptyset,\\
   &\Big\{ f\in K : \|f\1_{|f|>L}\|_{L^\infty_t(L^1)} >\eta\Big\} = \emptyset,\\
   &\Big\{ f\in K : \|f \1_{B_R^c}\|_{L^\infty_t(L^1)} >\eta \Big\} = \emptyset,
\end{aligned}
\end{equation}
Therefore, for such $M, L$ and $R$ large enough, we have 
\begin{equation}
  \begin{aligned}
    &\mu_n\Big\{ f: \|f\|_{L^\infty_t(L^1)} > M\Big\} \leq \mu_n(K^c) < \ep,\\
&\mu_n\Big\{ f : \|f\1_{|f|>L}\|_{L^\infty_t(L^1)} >\eta\Big\} \leq \mu_n(K^c) <\ep\\
    &\mu_n\Big\{ f : \|f \1_{B_R^c}\|_{L^\infty_t(L^1)} >\eta \Big\} \leq \mu_n(K^c) <\ep.
  \end{aligned}
\end{equation}
This completes the proof.
\end{proof}

We have the following representation and compactness criterion for $\LLsMw{p}$.
\begin{lem}\label{lem:representation-pettis}
  The space $\LLsMw{p}$ $p\in [1,\infty]$ is continuously linearly isomorphic to $\mathcal{L}(C_0(\Rd),L^p_{t,x})$ the space of continuous linear operators from $C_0(\Rd)$ to $L^p_{t,x}$ under the topology of pointwise convergence. Similarly $[\LLsMw{p}]_\loc$ is continuously linearly isomorphic to $\mathcal{L}(C_0(\Rd),[L^p_{t,x}]_\loc)$.
\end{lem}
\begin{proof}
  For each $f\in \LLsMw{p}$ we can trivially associate a bounded linear operator $S_f : C_0(\Rd) \to L^p_{t,x}$, by $S_f\phi = \langle f,\phi\rangle$, clearly the map $f\mapsto S_f$ is one-to-one, linear, and continuous from $\LLsMw{p}$ to $\mathcal{L}(C_0(\Rd),L^p_{t,x})$ with it's pointwise topology.

Conversely for each bounded linear operator $S\in \mathcal{L}(C_0(\Rd),L^p_{t,x})$ one may define for each $g\in L^q_{t,x}$, $ q= p/(p-1)$, the bounded linear functional $h_g:C_0(\Rd) \to \R$, by $h_g\phi = \langle S\phi , g\rangle$ which, by the Riesz-Markov theorem can be represented by a measure $f_g \in \Meas_v$, satisfying
  \begin{equation}
    h_g\phi = \langle f_g, \phi\rangle = \langle S\phi, g\rangle.
  \end{equation}
Since the mapping $g\mapsto f_g$ is clearly a continuous linear mapping from $L^q_{t,x}$ to $\Meas_v^*$, one can readily prove that for any bounded Borel $E\subset [0,T]\times\Rd$, that $\nu(E) = f_{\1_E}$ defines an $\Meas_v$ valued measure that $\dt\dx$ absolutely continuous and of $\sigma$ finite variation. Since $\Meas_v$ is a dual space, it has the weak-* Radon-Nikodym property (see \cite{musial2002pettis} Theorem 9.1) and therefore there is a measurable function $f_S:[0,T]\times\Omega \to \Meas_v^*$ such that $|\langle f_S,\phi\rangle| \in [L^1_{t,x}]_{\loc}$ and
\begin{equation}
  \langle S\phi , 1_{E}\rangle = \langle \nu(E), \phi\rangle = \iint_{E} \langle f_S, \phi\rangle \dx\dt.
\end{equation}
Using density of simple functions in $L^q_{t,x}$ we can conclude
\begin{equation}\label{eq:inverse-rep-formula}
  \langle S\phi, g\rangle = \int_0^T\int_{\Rd} \langle f_S, \phi\rangle g \dx\dt,
\end{equation}
for any $g\in L^q_{t,x}$. Taking the sup in $g\in L^q_{t,x}$, $\|g\|_{L^q} = 1$, on both sides of (\ref{eq:inverse-rep-formula}) we find
\begin{equation}
  \|\langle f,\phi \rangle\|_{L^p_{t,x}} = \|T\phi\|_{L^p_{t,x}} <\infty
\end{equation}
and therefore $f\in \LLsMw{p}$. Moreover this identity implies that the mapping $S\mapsto f_S$ is continuous from $\mathcal{L}(C_0(\Rd),L^p_{t,x})$ with it's pointwise topology to $\LLsMw{p}$, while identity (\ref{eq:inverse-rep-formula}) implies that $S\mapsto f_S$ is linear and one-to-one.

The proof on $[\LLsMw{p}]_{\loc}$ is similar and can be proved by the above argument on compact sets of $[0,T]\times\R^d$.
\end{proof}
\begin{lem}
\label{lem:compact-char-vel-avg}
Let $K$ be subset of $[\LLsMw{p}]_{\loc}$, $p\in [1,\infty]$, and let $\{\phi_k\}_{k=1}^\infty \subseteq C^\infty_c(\Rd_v)$ be a countable dense subset of $C_0(\Rd)$. Define the map $\Pi_{\phi_k} :[\LLsMw{p}]_{\loc} \to [L_{t,x}^{p}]_\loc$ by
\begin{equation}
  \Pi_{\phi_k}(f) = \langle f,\phi_k\rangle.
\end{equation}
Then $K$ is a compact subset of $[\LLsMw{p}]_\loc$ if and only if $K$ is bounded in $\LLsMw{p}$ and $\Pi_{\phi_{j}}K$ is compact in $[\LLt{p}]_{\loc}$ for all $j\geq 1$.
\end{lem}
\begin{proof}

  Let $\{f_n\}_{n=1}^\infty \subseteq K$, and assume that $j\geq 1$, $\{\langle f_n, \phi_j\rangle\}_{n=1}^\infty$ is compact in $[L_{t,x}^{p}]_{\loc}$. By a standard argument we may produce a diagonal subsequence, still denoted $\{f_n\}_{n=1}^\infty$, such that $\langle f_{n}, \phi_j\rangle$ converges as $n\to \infty$ for each $j\geq 1$. Identify $[\LLsM{p}]_{\loc}$ with $\mathcal{L}(C_0(\Rd); [L_{t,x}^{p}]_{\loc})$ as in Lemma \ref{lem:representation-pettis}, and for each $f\in [\LLsM{p}]_{\loc}$ let $T_f$ denote the corresponding element of $\mathcal{L}(C_0(\Rd); [L_{t,x}^{p}]_{\loc})$. Since $\{f_n\}_{n=1}^\infty$ is bounded in $[\LLsMw{p}]_{\loc}$, we have for any compact set $C\subset [0,T]\times\Rd$,
  \begin{equation}
    \sup_n \|T_{f_n}\phi\|_{L^p_{t,x}(C)} = \|\langle f_n,\phi\rangle\|_{L^p_{t,x}(C)} <\infty.
  \end{equation}
  By the uniform boundedness principle,
  \begin{equation}
    \sup_n \|\1_{C}T_{f_n}\|_{\mathcal{L}(C_0(\Rd); [L_{t,x}^{p}]_{\loc})} < \infty.
  \end{equation}
Therefore the mappings $\phi \mapsto \1_CT_{f_{n}}\phi = \1_C\langle f_n, \phi\rangle$ are equicontinuous. Since $\{\phi_j\}_{j=1}^\infty$ is dense in $C_0(\Rd)$, this equicontinuity implies that for each $\phi\in C_0(\Rd)$, $\{\1_C\langle f_{n},\phi\rangle\}_{n=1}^\infty$ is Cauchy in $L_{t,x}^{p}(C)$ and therefore $\{\langle f_{n},\phi\rangle\}_{n=1}^\infty$ is convergent in $[L^p_{t,x}]_{\loc}$. This limit defines a mapping $f: C_0(\Rd) \to [L_{t,x}^{p}]_{\loc}$, by 
\begin{equation}
f(\phi) \equiv \lim_{k\to\infty} \langle f_{n}, \phi\rangle.
\end{equation}
It is a simple consequence of the linearity of $\langle f_{n},\,\cdot\,\rangle$ and the boundedness of $\{f_n\}_{n=1}^\infty$, that the limiting $f$ belongs to $\mathcal{L}(C_0(\Rd), [L_{t,x}^{p}]_{\loc})$, and therefore belongs to $[\LLsMw{p}]_{\loc}$. Therefore $K$ is sequentially compact. Compactness of $K$ now follows from the fact that $[\LLsMw{p}]_{\loc}$ is a sequential space. 

The converse is simple. If $K$ is compact, since $\Pi_{\phi_j}$ are continuous, $\Pi_{\phi_j} K$ are compact in $[L_{t,x}^{p}]_{\loc}$.
\end{proof}
\begin{lem}\label{lem:tightness-crit-vel-averaged-space}
  Let $(\Omega,\mathcal{F},\P)$ be a probability space, and let $\{f_{n}\}_{n \in \N}$ be a bounded sequence in $L^{p}(\Omega \times [0,T] \times \R^{2d})$ for some $p\in [1,\infty]$. Then $\{f_{n}\}_{n \in \N}$ induces a tight family of laws on $[\LLsMw{p}]_{\loc}$ if and only if for all $\varphi \in C^{\infty}_{c}(\R^{d}_{v})$, the sequence $\{\langle f_{n},\varphi \rangle \}_{n \in \N}$ induces a tight family of laws on $[\LLt{p}]_{\loc}$.
  
\end{lem}
\begin{proof}
Clearly if $ \{f_n\}_{n\in\N}$ induce tight laws on $[\LLsMw{p}]_{\loc}$ then for each $\varphi \in C_c^\infty(\Rd_v)$, since the mapping $f\mapsto \langle f, \varphi \rangle$ is continuous from $[\LLsMw{p}]_{\loc}$ to $[\LLt{p}]_\loc$, $\{\langle f_n, \varphi\rangle\}_{n\in\N}$ is tight on $[\LLt{p}]_{\loc}$.

We proceed in the other direction by explicitly constructing a set $K$ which is compact in $[\LLsMw{p}]_\loc$ which has uniformly small probability. Fix and $\ep >0$ and let $\{\varphi_j\}_{j=1}^\infty\subseteq C^\infty_c(\Rd_v)$ be a dense subset of $C_0(\Rd_v)$. Since $\{\langle f_n, \varphi_j\rangle\}_{n\in\N}$ induce tight laws in $[\LLt{p}]_{\loc}$, then for each $j\in\N$ there exist a compact set $K_j$ in $[\LLt{p}]_{\loc}$ such that
  \begin{equation}
    \sup_n \P\{ \langle f_n, \varphi_j\rangle \notin K_j\} < \ep 2^{-j}.
  \end{equation}
  Define, as in Lemma \ref{lem:compact-char-vel-avg}, $\Pi_{\varphi_j}f = \langle f,\varphi_j\rangle$. Since $\Pi_{\varphi_j}$ is continuous from $[\LLsMw{p}]_\loc$ to  $[\LLt{p}]_\loc$, the pre-images $\Pi_{\phi_j}^{-1}K_j$ are closed in $[\LLsMw{p}]_{\loc}$. Let $C = \sup_n\E\|f_n\|_{\LLLs{p}}$ and define
  \begin{equation}
    B = \left\{f \in \LLLs{p}\,:\, \|f\|_{\LLLs{p}} \leq C\ep^{-1}\,\right\}
  \end{equation}
  and note that $B$ is a bounded subset of $[\LLsMw{p}]_{\loc}$. Now, define the closed set
\begin{equation}
  K = \bigcap_{j=1}^\infty \left(B\cap\Pi_{\varphi_j}^{-1}K_j\right),
\end{equation}
and note that $K \subseteq B$ is a bounded subset of $[\LLsMw{p}]_{\loc}$, and for each $j\in\N$, $\Pi_{\varphi_j} K$ is a closed subset of $K_j$, so the set $\Pi_{\varphi_j} K$ is compact in $[\LLt{p}]_{\loc}$. Therefore Lemma \ref{lem:compact-char-vel-avg} implies that $K$ is compact in $[\LLsMw{p}]_{\mathrm{loc}}$. We conclude the proof with
\begin{equation}
    \P\left\{ f_n \notin K\right\} \leq \P\left\{\|f\|_{\LLLs{p}} > C\ep^{-1}\right\}+ \sum_{j=1}^\infty\P\left\{\langle f_n,\varphi_j\rangle \notin K_j\right\} < 2\ep.
\end{equation}
\end{proof}

\section{}

The following product-limit lemma can be established in a classical way, using Egorov's theorem.
\begin{lem}\label{lem:product_lemma}
Let  $\{g_{n}\}_{n \in \N}$ and $\{h_{n}\}_{n \in \N}$ be sequences in $\LLLs{1}$.  Assume that $\{g_{n}\}_{n \in \N}$ is uniformly bounded in $\LLLs{\infty}$ and converges to $g$ in measure on $[0,T] \times \R^{2d}$. Then we have the following:
\begin{enumerate}
\item If the sequence $\{h_{n}\}_{n \in \N}$ converges to $h$ in $[\LLLs{1}]_{w}$,  then the sequence of products $\{g_{n}h_{n}\}_{n \in \N}$ converge to $gh$ in $[\LLLs{1}]_{w}$.  
\item If the sequence $\{h_{n}\}_{n \in \N}$ converges to $h$ in $[\LLLs{1}]_{w} \cap \LLsMw{1}$, then the sequence of products $\{g_{n}h_{n}\}_{n \in \N}$ converge to $gh$ in $\LLsMw{1}$.
\end{enumerate}
\end{lem}

The next lemma provides a procedure for identifying a continuous, adapted process as a series of one dimensional stochastic integrals.  

\begin{lem} \label{Lem:Appendix:Three_Martingales_Lemma}
Let $(\Omega, \mathcal{F},\p, \{ \mathcal{F}^{t}\}_{t=0}^{T}, \{ \beta_{k}\}_{k=1}^{\infty} )$ be a stochastic basis and let $\left ( M_{t} \right )_{t=0}^{T}$ be a continuous $(\mathcal{F}_{t})_{t=0}^{T}$ martingale with the quadratic variation process $(\int_{0}^{t}|f_{s}|^{2}_{\ell^2(\N)})_{t=0}^{T}$. Moreover, assume that for each $k \in \N$ the cross variation of $\left ( M_{t} \right )_{t=0}^{T}$ with $\beta_{k}$ is given by the process $(\int_{0}^{t}f_{k}(s)ds)_{t=0}^{T}$.
Under these hypotheses, the martingale may be identified as
$$
M_{t}=\sum_{k=1}^{\infty} \int_{0}^{t}f_{k}(s)\dee\beta_{k}(s).
$$
\end{lem}


\section{} \label{sec:Vel-Aver-Appendix}
\begin{proof}[Proof of Lemma \ref{lem:L2_velocity_averaging_x}]
For convenience we denote the velocity averaged process by
\begin{equation}
  \rho^\phi(t,x;\omega) = \int_{\Rd}f(t,x,v;\omega)\varphi(v)\dv.
\end{equation}
To begin, we assume that $f$ is regular enough for all the following computations to be well defined. Let $\Ft_x$ denote the Fourier transform in $x$ and let $\xi$ be the corresponding Fourier variable, for simplicity denote $\ft{f} = \Ft_x(f)$ and $\ft{g}= \Ft_x(g)$. Taking the Fourier transform of both sides of (\ref{eq:StochKin}) in It\^{o} form gives
\begin{equation}
\partial_t \ft{f}  + iv\cdot\xi\ft{f} +  \Ft_x(\Div_v(f\sigma_k\,\dot{\beta}_k)) = \Ft_x(\LStrat_{\sigma}f) + \ft{g}.
\end{equation}
If $|\xi| \leq 1$ we have the simple estimate
\begin{equation}
  \E\int_{0}^T\int_{\R^d}|\xi|^{1/3}|\ft{\rho}^\phi|^2\,\1_{|\xi|\leq 1}\,\dxi\dt \leq \|\phi\|_{L^\infty_v}^2\E\|f\|_{\LLLs{2}}^2.
\end{equation}
To show the $H_x^{1/6}$ estimate, it suffices to consider $|\xi| \geq 1$. We will find it useful to solve this equation with the addition of a damping term on both sides (corresponding to a pseudo-differential operator acting on $f$ in $x$). Let $\lambda\in C^\infty(\Rd_\xi)$, we now consider
\begin{equation}
\partial_t \ft{f}  + iv\cdot\xi\ft{f} +  \Ft_x(\Div_v(f\sigma_k\,\dot{\beta}_k)) + \lambda\ft{f} = \Ft_x(\LStrat_{\sigma}f) + \ft{g} + \lambda \ft{f}.
\end{equation}
Solving this via Duhammel, we find
\begin{equation}\label{eq:Duhammel-Vel-avg-L2}
\begin{split}
&\ft{f}(t,\xi,v) = e^{-(\lambda(\xi)+iv\cdot \xi)t}\ft{f_0}(\xi,v) + \lambda(\xi)\int_0^t e^{-(\lambda(\xi)+iv\cdot\xi)(t-s)}\ft{f}(s,\xi,v)\,\ds \\
&\hspace{.4in}+ \int_0^t e^{-(\lambda(\xi)+iv\cdot\xi)(t-s)}\ft{g}(s,\xi,v)\,\ds + \int_0^t e^{-(\lambda(\xi)+iv\cdot\xi)(t-s)}\Ft_x(\LStrat_\sigma f)(s,\xi,v)\,\ds \\
&\hspace{.4in}- \sum_{k=1}^\infty\int_0^t e^{-(\lambda(\xi)+iv\cdot\xi)(t-s)}\Ft_x(\Div_v(\sigma_kf))(s,\xi,v)\dee\beta_k(s).
\end{split}
\end{equation}
Let $\phi \in C^\infty_c(\Rd_v)$, upon multiplying both sides of (\ref{eq:Duhammel-Vel-avg-L2}) by $\phi$ and integrating in $v$, we see that the velocity average $\ft{\rho}\,{}^\phi$ satisfies
\begin{equation}
\label{eq:Duhammel-for-vel-avg}
\begin{split}
&\ft{\rho}\,{}^\phi(t,\xi) = \int_{\R^d}e^{-(\lambda(\xi) + iv\cdot\xi)t}\phi(v)\ft{f}_0(\xi,v)\,\dv\\
&\hspace{.5in}+\int_0^t \left(\int_{\Rd} e^{-(\lambda(\xi)+iv\cdot\xi)(t- s)}\ft{\Gamma}_0(s,\xi,v)\,\dv\right)\ds\\
&\hspace{.5in}- \sum_{k=1}^\infty\int_0^t\left(\int_{\Rd} e^{-(\lambda(\xi) + i v\cdot \xi)(t-s)}\Ft_x(\phi\Div_v(\sigma_kf))(s,\xi,v)\,\dv\right)\dee\beta_k(s)
\end{split}
\end{equation}
Where $\Gamma_0$ is defined so that
\begin{equation}\label{eq:gamma_0-D}
  \ft{\Gamma_0}(t,\xi,v) = \phi(v)\left(\lambda(\xi)\ft{f}(t,\xi,v) + \ft{g}(t,\xi,v) + \Ft_x(\LStrat_\sigma f)(t,\xi,v)\right).
\end{equation}
Note that the $v$ integrals in equation (\ref{eq:Duhammel-for-vel-avg}), can be written as a Fourier transform in $v$. We will denote such a Fourier transform in both $x$ and $v$ as $\Ft_{x,v}$, and denote by $\eta$ the Fourier variable dual to $v$. We find
\begin{equation}
\label{eq:Duhammel-for-vel-avg-after-ft-v}
\begin{split}
&\ft{\rho}\,{}^\phi(t,\xi) = e^{-\lambda(\xi) t}\Ft_{x,v}(\phi(v)f_0)(\xi,\xi t) +\int_0^te^{-\lambda(\xi)(t-s)}\Ft_{x,v}(\Gamma_0)(s,\xi,\xi(t-s))\ds\\
&\hspace{1in}- \sum_{k=1}^\infty\int_0^te^{-\lambda(\xi)(t-s)}\Ft_{x,v}(\phi\Div_v(\sigma_kf))(s,\xi,\xi(t-s))\dee\beta_k(s)\\
&\hspace{.5in}= I_1 + I_2 + I_3.
\end{split}
\end{equation}
The first term, $I_1$, we can bound 
\begin{equation}
|I_1|(t,\xi) \leq |\Ft_{x,v}(\phi\,f_0)(\xi,\xi t)|.
\end{equation}
For the second term, $I_2$, we have by Cauchy-Schwartz
\begin{equation}
\begin{aligned}
  |I_2|^2(t,\xi) &\leq \left(\int_0^t e^{-2\lambda(\xi)(t-s)}\,\ds\right)\left(\int_0^t \Bigl(e^{-\lambda(\xi)(t-s)}|\Ft_{x,v}(\Gamma_0)(s,\xi,\xi(t-s))|\Bigr)^2\,\ds\right)\\
&\leq \frac{1}{2\lambda(\xi)}\int_0^t \Bigl|e^{-\lambda(\xi)(t-s)}\Ft_{x,v}(\Gamma_0)(s,\xi,\xi(t-s))\Bigr|^2\,\ds.
\end{aligned}
\end{equation}
The term, $I_3(t,\xi)$ is a Martingale with quadratic variation
\begin{equation}
  \int_0^t\sum_{k=1}^\infty \Bigl(e^{-\lambda(\xi)(t-s)}|\Ft_{x,v}(\Gamma_k)(s,\xi,\xi(t-s))|\Bigr)^2\,\ds,
\end{equation}
where $\Gamma_k(t,x,v) = \phi\Div_v(\sigma_k f)(t,x,v)$. We conclude by the BDG inequality that
\begin{equation}
  \E|I_3|^2(t,\xi) \leq \E\int_0^t \sum_{k=1}^\infty \Bigl(e^{-\lambda(\xi)(t-s)}|\Ft_{x,v}(\Gamma_k)(s,\xi,\xi(t-s))|\Bigr)^2\,\ds.
\end{equation}
and therefore
\begin{equation}
\begin{split}
  &\E|\ft{\rho}\,{}^{\phi}(t,\xi)|^2\leq \E|\Ft_{x,v}(\phi\,f_0)(\xi,\xi t)|^2\\
&\hspace{1in} + \frac{1}{2\lambda(\xi)} \E\int_0^t \Bigl(e^{-\lambda(\xi)(t-s)}|\Ft_{x,v}(\Gamma_0)(s,\xi,\xi (t-s))|\Bigr)^2\ds\\
 &\hspace{1in} + \E\int_0^t\sum_{k=1}^\infty \Bigl(e^{-\lambda(\xi)(t-s)}|\Ft_{x,v}(\Gamma_k)(s,\xi,\xi(t-s))|\Bigr)^2\ds.
\end{split}
\end{equation}

The following identities can be readily verified
\begin{equation}
  \Gamma_k = \phi\, \Div_v(\sigma_k f) = \Div_v(\phi\, \sigma_k f) - \nabla \phi \cdot \sigma_k f,
\end{equation}
and 
\begin{equation}
\begin{split}
  \phi\,\LStrat_\sigma f &= \nabla_v^2 : (D_\sigma\,\phi\, f) - 2\Div_v(D_\sigma \nabla \phi f) + \nabla^2_v\phi : D_\sigma\,f - \Div_v(G_\sigma\,\phi\, f) + \nabla\phi \cdot  G_\sigma f,
\end{split}
\end{equation}
where we have denoted for convenience
\begin{equation}
  D_\sigma = \sum_{k=1}^\infty \sigma_k\tensor\sigma_k\quad\text{and}\quad  G_\sigma = \sum_{k=1}^\infty \sigma_k\cdot\nabla\sigma_k.
\end{equation}
This implies 
\begin{equation}
  \Ft_{x,v}(\Gamma_k) = i\eta\cdot\Ft_{x,v}(\phi\,\sigma_k f) - \Ft_{x,v}(\nabla\phi\cdot \sigma_k f).
\end{equation}
and
\begin{equation}
\begin{split}
  \label{eq:Lstrat-Fourier}
  &\Ft_{x,v}(\phi\, \LStrat_\sigma f) = -\eta\tensor\eta : \Ft_{x,v}(D_\sigma\,\phi\,f) -  2i\eta\cdot\Ft_{x,v}(D_\sigma\nabla\phi\,f) + \Ft_{x,v}(\nabla_v^2\phi : D_\sigma \,f)\\
&\hspace{1.5in} - i\eta\cdot\Ft_{x,v}(G_\sigma\,\phi\, f) + \Ft_{x,v}(\nabla\phi\cdot G_\sigma f).
\end{split}
\end{equation}

Using that $z^pe^{-\lambda z} \leq C_p \lambda^{-p}$, where $C_p$ is constant depending on $p$, we may bound
\begin{equation}
  \begin{split}
  e^{-\lambda z}|\Ft_{x,v}(\Gamma_k)(s,\xi,z\, \xi)| \leqc \lambda^{-1}|\xi|\,|\Ft_{x,v}(\phi\,\sigma_kf)(s,\xi,z\, \xi)| + |\Ft_{x,v}(\nabla \phi\cdot \sigma_k f)(s,\xi,z\, \xi)|
  \end{split}
\end{equation}
and using the definition of $\Gamma_0$, (\ref{eq:gamma_0-D}), and (\ref{eq:Lstrat-Fourier}) we can bound 
\begin{equation}
  \begin{split}
    &e^{-\lambda z}|\Ft_{x,v}(\Gamma_0)(s,\xi, z\,\xi)|\\
&\hspace{.5in}\leqc \lambda|\Ft_{x,v}(\phi\,f)(s,\xi ,z \xi)| + |\Ft_{x,v}(\phi\,g)(s,\xi ,z\, \xi)| + \lambda^{-2}|\xi|^2|\Ft_{x,v}(\phi D_\sigma f)(s,\xi ,z\, \xi)|\\
&\hspace{1in} + \lambda^{-1}|\xi||\Ft_{x,v}(D_{\sigma}\nabla\phi\,f)(s,\xi, z\,\xi)| + |\Ft_{x,v}(\nabla^2_v\phi\,:D_\sigma f)(s,\xi,z\,\xi)|\\
&\hspace{1in} + \lambda^{-1}|\xi||\Ft_{x,v}(\phi G_\sigma f)(s,\xi,z\, \xi)| + |\Ft_{x,v}(\nabla\phi\cdot G_{\sigma}f)(s,\xi,z\, \xi )|.
  \end{split}
\end{equation}
Integrating $\E|\ft{\rho}^\phi(t,\xi)|^2$ over $[0,T]$ and using the previous two bounds we get for a.e $\xi \in \Rd$,
\begin{equation}\label{eq:timeintegrated-fourier-vel-avg}
  \begin{split}
    &\E\int_0^T|\ft{\rho}^\phi(t,\xi)|^2\dt \leqc \E\int_0^T |\Ft_{x,v}(\phi\,f_0)(\xi,\xi t)|^2\dt\\
&\hspace{.3in} + \E\int_0^T\int_0^t\Bigg\{\lambda|\Ft_{x,v}(\phi\,f)(s,\xi ,(t-s) \xi)|^2 + \lambda^{-1}|\Ft_{x,v}(\phi\,g)(s,\xi ,(t-s)\, \xi)|^2\\
&\hspace{.3in}+ \lambda^{-5}|\xi|^4|\Ft_{x,v}(\phi D_\sigma f)(s,\xi ,(t-s)\, \xi)|^2 + \lambda^{-3}|\xi|^2|\Ft_{x,v}(D_{\sigma}\nabla\phi\,f)(s,\xi, (t-s)\,\xi)|^2\\
&\hspace{.3in} + \lambda^{-1}|\Ft_{x,v}(\nabla^2_v\phi\,:D_\sigma f)(s,\xi,(t-s)\,\xi)|^2 + \lambda^{-3}|\xi|^2|\Ft_{x,v}(\phi G_\sigma f)(s,\xi,(t-s)\, \xi)|^2\\
&\hspace{.3in} + \lambda^{-1}|\Ft_{x,v}(\nabla\phi\cdot G_{\sigma}f)(s,\xi,(t-s)\, \xi )|^2 + \sum_{k=1}^\infty \lambda^{-2}|\xi|^2\,|\Ft_{x,v}(\phi\,\sigma_kf)(s,\xi,(t-s)\, \xi)|^2\\
&\hspace{.3in} + \sum_{k=1}^\infty|\Ft_{x,v}(\nabla \phi\cdot \sigma_k f)(s,\xi,(t-s)\, \xi)|^2\Bigg\}\,\ds\dt.
  \end{split}
\end{equation}

Let's remark that, apart from the initial data, the above estimate is comprised entirely of integrals of the form
\begin{equation}
  \int_0^T \int_0^t |\Ft_{x,v} (h)(s,\xi,(t-s)\xi)|^2\ds\dt.
\end{equation}
Following the technique in \cite{Bouchut1999-nh}, such integrals can be estimated by changing variables to $(z,s) = (|\xi|(t-s),s)$, using Fubini, applying the classical trace theorem on the one dimensional integral in the $z$ variable, and applying Plancharel. We find that for any $\gamma > (d-1)/2$, 
\begin{equation}
\begin{split}
  \int_0^T \int_0^t |\Ft_{x,v} (h)(s,\xi,(t-s)\xi)|^2\ds\dt &\leq |\xi|^{-1}\int_0^T \int_{-\infty}^\infty \left|\Ft_{x,v}(h) \left(s,\xi, z\frac{\xi}{|\xi|}\right)\right|^2\dz \ds\\
&\leqc |\xi|^{-1}\int_0^T\int_{\Rd} (1+|v|^2)^{\gamma}|\Ft_{x}(h)(s,\xi,v)|^2\dv\ds,
 \end{split} 
\end{equation}
and for the initial data,
\begin{equation}
  \int_0^T |\Ft_{x,v}(\phi\,f_0)(\xi,\xi t)|^2\dt \leqc |\xi|^{-1}\int_{\Rd} (1+ |v|^2)^\gamma|\Ft_{x}(h)(\xi,v)|^2\dv.
\end{equation}
Applying the above two estimates term by term to (\ref{eq:timeintegrated-fourier-vel-avg}), we can readily estimate for a.e. $\xi$,
\begin{equation}
  \begin{split}
    &\E\int_0^T |\rho^\phi(t,\xi)|^2\dt \leq C_{\sigma,\phi}M(\xi)\Big(\int_{\Rd}|\ft{f}_0(\xi,v)|^2+ \E\int_0^T\int_{\Rd}|\ft{f}(\xi,v,s)|^2\dv\ds\\
 &\hspace{1in} +\E\int_{0}^T\int_{\Rd}|\ft{g}(\xi,v,s)|^2\dv\ds\Big),
  \end{split}
\end{equation}
where
\begin{equation}
M(\xi) = \frac{|\xi|^3}{\lambda(\xi)^5} + \frac{|\xi|}{\lambda(\xi)^3}+ \frac{|\xi|}{\lambda(\xi)^2} + \frac{1}{|\xi|\lambda(\xi)} + \frac{\lambda(\xi)}{|\xi|} + \frac{1}{|\xi|},
\end{equation}
and
\begin{equation}
  C_{\sigma,\phi} \leqc \|(|\phi|^2 + |\nabla\phi|^2 + |\nabla^2\phi|^2)(1+|v|^2)^{\gamma}\|_{L^\infty_v}\Big\|\sum_{k=1}^\infty (|\sigma_k|^2 + |\sigma\cdot\nabla\sigma_k|)\Big\|_{L^\infty_v}
\end{equation}
Choosing $\lambda(\xi) = |\xi|^{2/3}$, (really take $\lambda(\xi) = (\epsilon + |\xi|^2)^{1/3}$ and take $\epsilon \to 0$) we conclude that
\begin{equation}
M(\xi) =  3|\xi|^{-1/3} + 2|\xi|^{-1} + |\xi|^{-5/3} = \leq 6|\xi|^{-1/3} \text{ if } |\xi|\geq 1.
\end{equation}
Therefore
\begin{equation}
\begin{split}
&\E\int_{0}^T\int_{\Rd}|\xi|^{1/3}|\ft{\rho}^\phi|^2\1_{|\xi|\geq 1}\;\dee\xi\ds \leq C_{\sigma,\phi}\Big(\|f_0\|_{L^2_{x,v}}^2 + \E\|f\|_{L^2_{t,x,v}}^2 + \E\|g\|_{L^2_{t,x,v}}^2\Big),
\end{split} 
\end{equation}
whereby we have the desired inequality using the Fourier characterization of $H^{1/6}_x$.

The above proof can be extended to weak solutions $f \in L^2_{\omega,t,x,v}$, by first mollifying the equation in $(x,v)$ as in the proof of theorem \ref{prop:Weak_Is_Renormalized} and including the commutators with the term $g$ (along with another stochastic integral). The above computation, with the addition of a stochastic integral to the right-hand-side, still apply and the resulting estimates are computed in terms of the the $L^2_{\omega,t,x,v}$ norm of the right-hand-side, the commutator contribution will then vanish as the mollification parameter goes to $0$. Furthermore we may pass the limit in each term on the right-hand side using the properties of mollifiers. The resulting $H^{1/6}$ estimate on the mollified velocity average can then be used to conclude the associated $H^{1/6}$ estimate on the limiting $f$ by a monotone convergence argument on the Fourier side.

\end{proof}


\end{appendices}


\bibliographystyle{abbrv}
\bibliography{bibliography}

\begin{thebibliography}{10}

\bibitem{alexandre2002boltzmann}
R.~Alexandre and C.~Villani.
\newblock On the {B}oltzmann equation for long-range interactions.
\newblock {\em Communications on pure and applied mathematics}, 55(1):30--70,
  2002.

\bibitem{arkeryd1984loeb}
L.~Arkeryd.
\newblock {L}oeb solutions of the {B}oltzmann equation.
\newblock {\em Archive for rational mechanics and analysis}, 86(1):85--97,
  1984.

\bibitem{arkeryd2010stability}
L.~Arkeryd, R.~Esposito, R.~Marra, and A.~Nouri.
\newblock Stability for {R}ayleigh--{B}enard convective solutions of the
  {B}oltzmann equation.
\newblock {\em Archive for rational mechanics and analysis}, 198(1):125--187,
  2010.

\bibitem{Asano1987-dr}
K.~Asano.
\newblock On the global solutions of the initial boundary value problem for the
  {B}oltzmann equation with an external force.
\newblock {\em Transp. Theory Stat. Phys.}, 16(4-6):735--761, 1987.

\bibitem{bailleul2015unbounded}
I.~Bailleul and M.~Gubinelli.
\newblock Unbounded rough drivers.
\newblock {\em arXiv preprint arXiv:1501.02074}, 2015.

\bibitem{bardos1991fluid}
C.~Bardos, F.~Golse, and D.~Levermore.
\newblock Fluid dynamic limits of kinetic equations. {I}. {F}ormal derivations.
\newblock {\em Journal of Statistical Physics}, 63(1-2):323--344, 1991.

\bibitem{Bellomo1989-wg}
N.~Bellomo, M.~Lachowicz, A.~Palczewski, and G.~Toscani.
\newblock On the initial value problem for the {B}oltzmann equation with a
  force term.
\newblock {\em Transp. Theory Stat. Phys.}, 18(1):87--102, 1989.

\bibitem{bixon1969boltzmann}
M.~Bixon and R.~Zwanzig.
\newblock {B}oltzmann-{L}angevin equation and hydrodynamic fluctuations.
\newblock {\em Physical Review}, 187(1):267, 1969.

\bibitem{Bouchut1999-nh}
F.~Bouchut and L.~Desvillettes.
\newblock Averaging lemmas without time {F}ourier transform and application to
  discretized kinetic equations.
\newblock {\em Proceedings of the Royal Society of Edinburgh, Section: A
  Mathematics}, 129(01):19--36, 1999.

\bibitem{bris2008existence}
C.~L. Bris and P.-L. Lions.
\newblock Existence and uniqueness of solutions to {F}okker--{P}lanck type
  equations with irregular coefficients.
\newblock {\em Communications in Partial Differential Equations},
  33(7):1272--1317, 2008.

\bibitem{brzezniak2016existence}
Z.~Brze{\'z}niak, F.~Flandoli, and M.~Maurelli.
\newblock Existence and uniqueness for stochastic 2d {E}uler flows with bounded
  vorticity.
\newblock {\em Archive for Rational Mechanics and Analysis}, 221(1):107--142,
  2016.

\bibitem{Cercignani1988-aw}
C.~Cercignani.
\newblock {\em The {B}oltzmann {E}quation and {I}ts {A}pplications:}.
\newblock Applied Mathematical Sciences. Springer New York, 1988.

\bibitem{Cercignani2013-vz}
C.~Cercignani, R.~Illner, and M.~Pulvirenti.
\newblock {\em The mathematical theory of dilute gases}, volume 106.
\newblock Springer Science \& Business Media, 2013.

\bibitem{coghi2014propagation}
M.~Coghi and F.~Flandoli.
\newblock Propagation of chaos for interacting particles subject to
  environmental noise.
\newblock {\em arXiv preprint arXiv:1403.1981}, 2014.

\bibitem{debussche2015invariant}
A.~Debussche and J.~Vovelle.
\newblock Invariant measure of scalar first-order conservation laws with
  stochastic forcing.
\newblock {\em Probability Theory and Related Fields}, 163(3-4):575--611, 2015.

\bibitem{delarue2014noise}
F.~Delarue, F.~Flandoli, and D.~Vincenzi.
\newblock Noise prevents collapse of vlasov-poisson point charges.
\newblock {\em Communications on Pure and Applied Mathematics},
  67(10):1700--1736, 2014.

\bibitem{diperna1989cauchy}
R.~J. DiPerna and P.-L. Lions.
\newblock On the {C}auchy problem for {B}oltzmann equations: global existence
  and weak stability.
\newblock {\em Annals of Mathematics}, pages 321--366, 1989.

\bibitem{diperna1989ordinary}
R.~J. DiPerna and P.-L. Lions.
\newblock Ordinary differential equations, transport theory and {S}obolev
  spaces.
\newblock {\em Inventiones mathematicae}, 98(3):511--547, 1989.

\bibitem{Duan2006-ie}
R.~Duan, T.~Yang, and C.~Zhu.
\newblock {L}1 and {BV}-type stability of the {B}oltzmann equation with
  external forces.
\newblock {\em Journal of Differential Equations}, 227(1):1--28, 2006.

\bibitem{esposito1998solutions}
R.~Esposito, R.~Marra, and J.~Lebowitz.
\newblock Solutions to the {B}oltzmann equation in the {B}oussinesq regime.
\newblock {\em Journal of statistical physics}, 90(5-6):1129--1178, 1998.

\bibitem{fedrizzi2016regularity}
E.~Fedrizzi, F.~Flandoli, E.~Priola, and J.~Vovelle.
\newblock Regularity of stochastic kinetic equations.
\newblock {\em arXiv preprint arXiv:1606.01088}, 2016.

\bibitem{flandoli2011interaction}
F.~Flandoli.
\newblock The interaction between noise and transport mechanisms in pdes.
\newblock {\em Milan Journal of Mathematics}, 79(2):543--560, 2011.

\bibitem{flandoli2010well}
F.~Flandoli, M.~Gubinelli, and E.~Priola.
\newblock Well-posedness of the transport equation by stochastic perturbation.
\newblock {\em Inventiones mathematicae}, 180(1):1--53, 2010.

\bibitem{fox1970contributions}
R.~F. Fox and G.~E. Uhlenbeck.
\newblock Contributions to nonequilibrium thermodynamics. ii. fluctuation
  theory for the {B}oltzmann equation.
\newblock {\em Physics of Fluids (1958-1988)}, 13(12):2881--2890, 1970.

\bibitem{gallagher2013newton}
I.~Gallagher, L.~Saint-Raymond, and B.~Texier.
\newblock {\em From {N}ewton to {B}oltzmann: hard spheres and short-range
  potentials}.
\newblock European mathematical society, 2013.

\bibitem{gess2014long}
B.~Gess and P.~E. Souganidis.
\newblock Long-time behavior, invariant measures and regularizing effects for
  stochastic scalar conservation laws.
\newblock {\em arXiv preprint arXiv:1411.3939}, 2014.

\bibitem{golse1988regularity}
F.~Golse, P.-L. Lions, B.~Perthame, and R.~Sentis.
\newblock Regularity of the moments of the solution of a transport equation.
\newblock {\em Journal of functional analysis}, 76(1):110--125, 1988.

\bibitem{golse2002velocity}
F.~Golse and L.~Saint-Raymond.
\newblock Velocity averaging in {L}1 for the transport equation.
\newblock {\em Comptes Rendus Mathematique}, 334(7):557--562, 2002.

\bibitem{Golse2005-zh}
F.~Golse and L.~Saint-Raymond.
\newblock Hydrodynamic limits for the {B}oltzmann equation.
\newblock {\em RIVISTA DI MATEMATICA DELLA UNIVERSIT\`{A} DI PARMA},
  4**:1--144, 31~Dec. 2005.

\bibitem{jabin2004real}
P.-E. Jabin and L.~Vega.
\newblock A real space method for averaging lemmas.
\newblock {\em Journal de math{\'e}matiques pures et appliqu{\'e}es},
  83(11):1309--1351, 2004.

\bibitem{jakubowski1997non}
A.~Jakubowski.
\newblock The almost sure {S}korokhod representation for subsequences in
  nonmetric spaces.
\newblock {\em Teor. Veroyatnost. i Primenen.}, 42(1):209--216, 1997.

\bibitem{kunita1997stochastic}
H.~Kunita.
\newblock {\em Stochastic flows and stochastic differential equations},
  volume~24.
\newblock Cambridge university press, 1997.

\bibitem{landau1959fluid}
L.~Landau and E.~Lifshitz.
\newblock Fluid mechanics, 1959.
\newblock {\em Course of Theoretical Physics}, 1959.

\bibitem{Lions1994-jf}
P.~L. Lions.
\newblock Compactness in {B}oltzmann's equation via {F}ourier integral
  operators and applications. {III}.
\newblock {\em Journal of Mathematics of Kyoto University}, 34(3):539--584,
  1994.

\bibitem{Lions2011-2012}
P.-L. Lions, B.~Perthame, and P.~E. Souganidis.
\newblock Stochastic averaging lemmas for kinetic equations.
\newblock {\em S{\'e}minaire Laurent Schwartz -- EDP et applications},
  2011-2012:1--17, 2011-2012.

\bibitem{hofmanova2015bhatnagar}
H.~Martina.
\newblock A {B}hatnagar-{G}ross-{K}rook approximation to stochastic scalar
  conservation laws.
\newblock In {\em Annales de L Institut Henri Poincare-Probabilites Et
  Statistiques}, volume~51, page~4, 2015.

\bibitem{montroll2012fluctuation}
E.~Montroll.
\newblock {\em Fluctuation phenomena}, volume~7.
\newblock Elsevier, 2012.

\bibitem{musial2002pettis}
K.~Musial.
\newblock Pettis integral.
\newblock {\em Handbook of measure theory}, 1:531--586, 2002.

\bibitem{pardoux1975equations}
{\'E}.~Pardoux.
\newblock {\em Equations aux d{\'e}riv{\'e}es partielles stochastiques non
  lineaires monotones: Etude de solutions fortes de type {I}to}.
\newblock PhD thesis, 1975.

\bibitem{MR2683475}
L.~Saint-Raymond.
\newblock {\em Hydrodynamic limits of the {B}oltzmann equation}, volume 1971 of
  {\em Lecture Notes in Mathematics}.
\newblock Springer-Verlag, Berlin, 2009.

\bibitem{spohn1981fluctuations}
H.~Spohn.
\newblock Fluctuations around the {B}oltzman equation.
\newblock {\em Journal of Statistical Physics}, 26(2):285--305, 1981.

\bibitem{spohn1983fluctuation}
H.~Spohn.
\newblock Fluctuation theory for the {B}oltzmann equation.
\newblock {\em Nonequilibrium Phenomena I: The {B}oltzmann Equation},
  1:225--251, 1983.

\bibitem{spohn2012large}
H.~Spohn.
\newblock {\em Large scale dynamics of interacting particles}.
\newblock Springer Science \& Business Media, 2012.

\bibitem{ueyama1980stochastic}
H.~Ueyama.
\newblock The stochastic {B}oltzmann equation and hydrodynamic fluctuations.
\newblock {\em Journal of Statistical Physics}, 22(1):1--26, 1980.

\bibitem{ukai2005global}
S.~Ukai, T.~Yang, and H.~Zhao.
\newblock Global solutions to the {B}oltzmann equation with external forces.
\newblock {\em Analysis and Applications}, 3(02):157--193, 2005.

\bibitem{van1996weak}
A.~W. Van Der~Vaart and J.~A. Wellner.
\newblock {\em Weak Convergence}.
\newblock Springer, 1996.

\bibitem{villani2002review}
C.~Villani.
\newblock A review of mathematical topics in collisional kinetic theory.
\newblock {\em Handbook of mathematical fluid dynamics}, 1:71--305, 2002.

\end{thebibliography}

\end{document}